\documentclass[a4paper]{article}

\usepackage{amsfonts,mathrsfs}
\usepackage{amsmath,amssymb,amsthm, amscd}
\usepackage{tikz}
\usetikzlibrary{arrows,calc,matrix,topaths,positioning,scopes,shapes,decorations,decorations.markings} 
\usepackage{dsfont}
\usepackage{epsfig}
\usepackage{authblk}
\usepackage{tikz-cd}
\usepackage{hyperref}
\usepackage{adjustbox}

\usepackage[normalem]{ulem}
\usepackage{cancel}
\newcommand{\quotient}[2]{{\raisebox{.2em}{$#1$}\left/\raisebox{-.2em}{$#2$}\right.}}
\newcommand{\SL}{\operatorname{SL}}

\newcommand{\End}{\operatorname{End}}

\newcommand{\Rep}{\operatorname{Rep}}
\newcommand{\tr}{\operatorname{Tr}}
\newcommand{\Mod}{\operatorname{Mod}}

\newcommand{\sslash}{\mathbin{/\mkern-6mu/}}

\newcommand{\Hom}{\operatorname{Hom}}
\newcommand{\Res}{\operatorname{Res}}
\newcommand{\lt}{\operatorname{lt}}
\newcommand{\St}{\operatorname{St}}
\newcommand{\ord}{\operatorname{ord}}
\newcommand{\id}{\operatorname{id}}

\newcommand{\heightexch}[3]{
	\begin{tikzpicture}[baseline=-0.4ex,scale=0.5, >=stealth]
	\draw [fill=gray!60,gray!45] (-.7,-.75)  rectangle (.4,.75)   ;
	\draw[#1] (0.4,-0.75) to (.4,.75);
	\draw[line width=1.2] (0.4,-0.3) to (-.7,-.3);
	\draw[line width=1.2] (0.4,0.3) to (-.7,.3);
	\draw (0.65,0.3) node {\scriptsize{$#2$}}; 
	\draw (0.65,-0.3) node {\scriptsize{$#3$}}; 
	\end{tikzpicture}
}

\newcommand{\resneg}{
\tikz[scale=0.1,baseline=2ex]{
\draw (2,0)  arc (0:180:1);
\draw (0,3) arc  (180:360:1);
}}

\newcommand{\respos}{
\tikz[scale=0.2,baseline=0ex]{
\draw (0,0)  arc (90:-90:0.6);
\draw (1.5,0) arc  (90:270:0.6);
}}

\newcommand{\crosspos}{
\tikz[scale=0.1,baseline=1ex]{
\draw (-1,-1) -- (1,1);
\draw (-1,1) -- (-0.5,0.5);
\draw (0.5,-0.5) -- (1,-1);
}}

\newcommand{\crossneg}{
\tikz[scale=0.1,baseline=1ex]{
\draw (-1,1) -- (1,-1);
\draw (-1,-1) -- (-0.5,-0.5);
\draw (0.5,0.5) -- (1,1);
}}

\newcommand{\heightexchright}[3]{
	\begin{tikzpicture}[baseline=-0.4ex,scale=0.5, >=stealth]
	\draw [fill=gray!60,gray!45] (-.7,-.75)  rectangle (.4,.75)   ;
	\draw[#1] (-0.7,-0.75) to (-0.7,.75);
	\draw[line width=1.2] (0.4,-0.3) to (-.7,-.3);
	\draw[line width=1.2] (0.4,0.3) to (-.7,.3);
	\draw (-1,0.3) node {\scriptsize{$#2$}}; 
	\draw (-1,-0.3) node {\scriptsize{$#3$}}; 
	\end{tikzpicture}
}

\newcommand{\heightcurve}{
\begin{tikzpicture}[baseline=-0.4ex,scale=0.5]
\draw [fill=gray!20,gray!45] (-.7,-.75)  rectangle (.4,.75)   ;
\draw[-] (0.4,-0.75) to (.4,.75);
\draw[line width=1.2] (-.7,-0.3) to (-.4,-.3);
\draw[line width=1.2] (-.7,0.3) to (-.4,.3);
\draw[line width=1.15] (-.4,0) ++(-90:.3) arc (-90:90:.3);
\end{tikzpicture}
}

\newcommand{\heightcurveright}{
\begin{tikzpicture}[baseline=-0.4ex,scale=0.5]
\draw [fill=gray!20,gray!45] (-.7,-.75)  rectangle (.4,.75)   ;
\draw[-] (-0.7,-0.75) to (-.7,.75);
\draw[line width=1.2] (0.1,-0.3) to (.4,-.3);
\draw[line width=1.2] (0.1,0.3) to (.4,.3);
\draw[line width=1.15] (.1,0) ++(90:.3) arc (90:270:.3);
\end{tikzpicture}
}

\begin{document}

\theoremstyle{plain}
\newtheorem{theorem}{Theorem}[section]
\newtheorem{main_theorem}[theorem]{Main Theorem}
\newtheorem{proposition}[theorem]{Proposition}
\newtheorem{corollary}[theorem]{Corollary}
\newtheorem{corollaire}[theorem]{Corollaire}
\newtheorem{lemma}[theorem]{Lemma}
\theoremstyle{definition}
\newtheorem{notations}[theorem]{Notations}
\newtheorem*{notations*}{Notations}
\newtheorem{definition}[theorem]{Definition}
\newtheorem{Theorem-Definition}[theorem]{Theorem-Definition}
\theoremstyle{remark}
\newtheorem{remark}[theorem]{Remark}
\newtheorem{example}[theorem]{Example}


\title{Classical shadows of stated skein representations at roots of unity}
\author{Julien Korinman}

\affil[1]{\normalsize Institut Montpelli\'erain Alexander Grothendieck - UMR 5149 Universit\'e de Montpellier. Place Eug\'ene Bataillon, 34090 Montpellier France
	\\ \small email: \texttt{julien.korinman@gmail.com}}

\author{Alexandre Quesney}
\affil[2]{\normalsize Universidad Polit\'ecnica de Madrid (UPM)
                Campus de Montegancedo / Avenida de Montepr\'incipe s.n.
                28660 Boadilla del Monte, Madrid,
                Spain
	\\ \small email: \texttt{alexandre.quesney@upm.es}}
%
%

\date{}
\maketitle


\begin{abstract} 
We extend some results of Bonahon-Wong, Bullock and Turaev concerning the skein algebras of closed surfaces to L\^e's stated skein algebras associated to open surfaces. We prove that the stated skein algebra with deforming parameter $+1$ embeds canonically into the center of the stated skein algebra whose deforming parameter is an odd root unity. We also construct an isomorphism between the stated skein algebra at $+1$ and the algebra of regular functions of the relative $\SL_2$-character variety of the surface. As a result, we associate to each isomorphism class of irreducible or local representations of the stated skein algebra, an invariant which is a point in the relative character variety.
\vspace{2mm}
\end{abstract}


\section{Introduction}

 A \textit{punctured surface} is a pair $\mathbf{\Sigma}=(\Sigma,\mathcal{P})$, where $\Sigma$ is a compact oriented surface and $\mathcal{P}$ is a (possibly empty) finite subset of $\Sigma$ which intersects non-trivially each boundary component.  We write $\Sigma_{\mathcal{P}}:= \Sigma \setminus \mathcal{P}$.
  The set $\partial \Sigma\setminus \mathcal{P}$ consists of a disjoint union of open arcs which we call \textit{boundary arcs}.
   \\ \textbf{Warning:} In this paper, the punctured surface $\mathbf{\Sigma}$ will be called open if the surface $\Sigma$ has non empty boundary and closed if $\Sigma$ is closed. This convention differs from the traditional one, where some authors refer to open punctured surface a punctured surface $\mathbf{\Sigma}=(\Sigma, \mathcal{P})$ with $\Sigma$ closed and $\mathcal{P}\neq \emptyset$ (in which case $\Sigma_{\mathcal{P}}$ is not closed).
\vspace{2mm}
\par We will consider two related objects associated to a punctured surface, namely the Kauffman-bracket skein algebra and the $\SL_2(\mathbb{C})$-character variety. These objects have been well studied in the case where the punctured surface is closed. They were recently generalized to open punctured surfaces {in such a way that they have} a nice behaviour relatively to the operation of gluing two boundary arcs together. The goal of this paper is to extend some classical results concerning skein algebras and character varieties to the case of open punctured surfaces. Before we state the main results, let us give a brief historical background.

\subsubsection*{Historical background}
\textit{Closed surfaces.} 
In \cite{CullerShalenCharVar}, Culler and Shalen defined the $\SL_2(\mathbb{C})$ character variety $\mathcal{X}_{\SL_2}(M)$ of a manifold $M$ whose fundamental group is finitely generated. This affine variety is closely related to the moduli space of flat connections on a trivial $\SL_2(\mathbb{C})$ bundle over $M$ and, therefore, it is related to Chern-Simons topological quantum field theory, gauge theory and low-dimensional topology; see \cite{LabourieCharVar, MarcheCours09, MarcheCharVarSkein} for surveys.  
If $\Sigma$ is a closed oriented surface, the smooth part of $\mathcal{X}_{\SL_2}(\Sigma)$  carries a symplectic form, first defined in \cite{AB} in the context of gauge theory. 
This symplectic structure was used by Goldman \cite{Goldman86} to equip the algebra of regular functions $\mathbb{C}[\mathcal{X}_{\SL_2}(\mathbf{\Sigma})]$ with a Poisson bracket.  
A similar Poisson structure for character varieties of punctured closed surfaces was introduced by Fock and Rosly in \cite{FockRosly} (see also \cite{AlekseevKosmannMeinrenken} for an alternative construction) in the differential geometric context. 
\vspace{2mm}
\par	 Turaev \cite{Tu88}, Hoste and Przytycki \cite{HP92} independently defined the \textit{Kauffman-bracket skein algebra} $\mathcal{S}_A(\mathbf{\Sigma})$ as a tool to study the Jones polynomial and the $\mathrm{SU}(2)$ Witten-Reshetikhin-Turaev TQFTs. 
Skein algebras are defined for any commutative unital ring $\mathcal{R}$ together with an invertible element $A\in \mathcal{R}^{\times}$ and a closed punctured surface $\mathbf{\Sigma}$. 
\vspace{2mm}
\par Skein algebras are deformations of the algebra of regular functions of character varieties of closed punctured surfaces. 
In particular, this means that there is an isomorphism of Poisson algebras between $\mathcal{S}_{+1}(\mathbf{\Sigma})$ and $\mathbb{C}[\mathcal{X}_{\SL_2}(\mathbf{\Sigma})]$. 
In more details, this relies  on a (non canonical) isomorphism from $\mathcal{S}_{+1}(\mathbf{\Sigma})$ to $\mathcal{S}_{-1}(\mathbf{\Sigma})$ (\cite{Barett}). The latter algebra carries a natural Poisson bracket (see Section 2.5). 
An isomorphism of algebras between $\mathcal{S}_{-1}(\mathbf{\Sigma})$ and $\mathbb{C}[\mathcal{X}_{\SL_2}(\mathbf{\Sigma})]$  was defined by Bullock \cite{Bullock}, assuming that the skein algebra is reduced, \textit{i.e.} that its nilradical is null. This latter fact was later proved independently in \cite{PS00} and \cite{ChaMa}. Turaev showed in \cite{Turaev91} that Bullock's isomorphism is Poisson.

\vspace{2mm}
\par In TQFT,  skein algebras appear through their non-trivial finite dimensional representations. Skein algebras admit such representations if and only if the parameter $A$ is a root of unity. A recent result of Bonahon and Wong in \cite{BonahonWong1} states, in particular, that when $A$ has odd order, there exists an embedding of $\mathcal{S}_{+1}(\mathbf{\Sigma})$ into the center of $\mathcal{S}_A(\mathbf{\Sigma})$. Since each simple representation induces a character 
on the center of the skein algebra, using Bullock's isomorphism, one can associate to each isomorphism class of simple representation a point in the character variety. This invariant is called \textit{the classical shadow} of the representation. 
\vspace{2mm} \\
\textit{Open surfaces.} 
In \cite{LeStatedSkein}, L\^e generalized the Kauffman-bracket skein algebras to open punctured surfaces based on the original work of Bonahon-Wong in \cite{BonahonWongqTrace}.
We call it  \textit{stated skein algebra} and  denote it  by  $\mathcal{S}_{\omega}(\mathbf{\Sigma})$. 
It depends on an invertible element $\omega \in \mathcal{R}^{\times}$. When the surface is closed, it coincides with the classical skein algebra with parameter $A=\omega^{-2}$.  
An important feature of the stated skein algebra is its behavior under gluing of surfaces. 
More precisely, let $a$ and $b$ be two boundary arcs of an open punctured surface $\mathbf{\Sigma}$, and let us denote by  $\mathbf{\Sigma}_{|a\#b}$ the surface obtained from $\mathbf{\Sigma}$ by gluing $a$ and $b$. 
L\^e showed that there is an injective algebra morphism 
\begin{equation}\label{eq: le morph}
i_{|a\#b} : \mathcal{S}_{\omega}(\mathbf{\Sigma}_{| a\#b}) \hookrightarrow \mathcal{S}_{\omega}(\mathbf{\Sigma}) 
\end{equation}
which is coassociative in that it does not depend on the order we glue the arcs \textit{i.e.} for four distinct boundary arcs $a,b,c,d$, one has $i_{|a\#b}\circ i_{|c\#d} = i_{|c\#d}\circ i_{|a\#b}$.
In particular, for each topological triangulation $\Delta$ of $\mathbf{\Sigma}$, one has an injective morphism of algebras 
\begin{equation}\label{eq: le morph2}
i^{\Delta} : \mathcal{S}_{\omega}(\mathbf{\Sigma}) \hookrightarrow \otimes_{\mathbb{T} \in F(\Delta)} \mathcal{S}_{\omega}(\mathbb{T}). 
\end{equation}
Here $\mathbb{T}$ denotes the triangle, \textit{i.e.} a disc with three punctures on its boundary. A punctured surface is \textit{triangulable} if it can be obtained from a disjoint union of triangles by gluing some pair of boundary arcs (=faces of triangles) together. A \textit{topological triangulation}  is the data of such a union of triangles together with the pairs of glued boundary arcs. In Equation \eqref{eq: le morph2}, the tensor product runs over the faces of the triangulation; see Section 2 for precise definitions.  

As applications, L\^e provided a simple proof that the algebra $\mathcal{S}_{\omega}(\mathbf{\Sigma})$ has no zero divisor  (the case of closed triangulable punctured surfaces was proved earlier in \cite{BonahonWongqTrace} using the quantum trace and the case of closed unpunctured surfaces was proved in \cite{PrzytyckiSikora_SkeinDomain}) and he obtained a simpler formulation of Bonahon and Wong's quantum trace map of \cite{BonahonWongqTrace}.  
\vspace{2mm}
\par Motivated by L\^e's construction, the first author defined in \cite{KojuTriangularCharVar} a generalization of character varieties to open punctured surfaces. We denote it by $\mathcal{X}_{\SL_2}(\mathbf{\Sigma})$. This (relative) character variety is a Poisson affine variety which coincides with the classical one when the surface is closed. 
It shares a similar gluing property than the stated skein algebra, namely, there exist injective Poisson morphisms $i_{|a\#b}: \mathbb{C}[\mathcal{X}_{\SL_2}(\mathbf{\Sigma}_{|a\#b})] \hookrightarrow \mathbb{C}[\mathcal{X}_{\SL_2}(\mathbf{\Sigma})] $ and $i^{\Delta} : \mathbb{C}[\mathcal{X}_{\SL_2}(\mathbf{\Sigma})] \hookrightarrow \otimes_{\mathbb{T}\in F(\Delta)} \mathbb{C}[\mathcal{X}_{\SL_2}(\mathbb{T})]$ between the Poisson algebras of regular functions. 
However, the Poisson structure on $\mathbb{C}[\mathcal{X}_{\SL_2}(\mathbf{\Sigma})]$ depends on a choice of an orientation $\mathfrak{o}$ of the boundary arcs of the punctured surface. We denote by $\{\cdot, \cdot\}^{\mathfrak{o}}$ its Poisson bracket.

\subsubsection*{Main results}

Let $\mathbf{\Sigma}$ be a punctured surface. 
 L\^e's morphism \eqref{eq: le morph2} embeds the skein algebra of a triangulated surface into a tensor product of the skein algebras of the triangle. However, it does not provide a full description of the stated skein algebra in terms of these smaller pieces. In a first result we provide such a description; it goes as follows.  
Remark that \eqref{eq: le morph} endows the skein algebra of the bigon  $\mathbb{B}$ (\textit{i.e.} a disc with two punctures on its boundary) with a bialgebra structure. It is in fact a Hopf algebra and one can show that it is canonically isomorphic to the classical quantum $\SL_2$ algebra $\mathcal{O}_q[\SL_2]$ described in \cite{Kassel, ChariPressley} (with $q= \omega^{-4}$). 
Note also that \eqref{eq: le morph} induces Hopf comodule maps: 
$\Delta_a^L : \mathcal{S}_{\omega}(\mathbf{\Sigma}) \rightarrow \mathcal{S}_{\omega}(\mathbb{B}) \otimes \mathcal{S}_{\omega}(\mathbf{\Sigma})$ 
and 
$\Delta_b^R : \mathcal{S}_{\omega}(\mathbf{\Sigma}) \rightarrow \mathcal{S}_{\omega}(\mathbf{\Sigma}) \otimes \mathcal{S}_{\omega}(\mathbb{B})$ 
obtained by gluing a bigon on a boundary arc, $a$ or $b$, of $\mathbf{\Sigma}$; see Section $2.2$ for details. 

\begin{theorem}\label{theorem1} 
	The following sequence is exact:
	$$ 0 \rightarrow \mathcal{S}_{\omega}(\mathbf{\Sigma}_{|a\#b}) \xrightarrow{i_{|a\#b}} \mathcal{S}_{\omega}(\mathbf{\Sigma}) \xrightarrow{\Delta_a^L - \sigma \circ \Delta_b^R} \mathcal{S}_{\omega}(\mathbb{B}) \otimes \mathcal{S}_{\omega}(\mathbf{\Sigma}),$$
	where $\sigma(x\otimes y)= y\otimes x$.
\end{theorem}

\par Theorem \ref{theorem1} can be reformulated using  coHochschild  cohomology, whose zeroth group (see Definition \ref{def_coHochschild} and \cite{HPS_CohochschildHom}) computes the skein algebra: 
$$ \mathcal{S}_{\omega}(\mathbf{\Sigma}_{|a\#b}) \cong \mathrm{coHH}^0 (\mathcal{O}_q[\SL_2] , \quad {}_a{\mathcal{S}_{\omega}(\mathbf{\Sigma})}_b),$$ where  $_a{\mathcal{S}_{\omega}(\mathbf{\Sigma})}_b$ is seen as a bicomodule over $\mathcal{O}_q[\SL_2]$ via the comodule maps $\Delta_a^L$ and $\Delta_b^R$.
\par  Theorem \ref{theorem1} provides, for  any topological triangulation $\Delta$ of $\mathbf{\Sigma}$,  an isomorphism of algebras  
\begin{equation*}
\mathcal{S}_{\omega}(\mathbf{\Sigma})\cong \mathrm{coHH}^0\left( \otimes_{e\in \mathring{\mathcal{E}}(\Delta)} \mathcal{O}_q[\SL_2],\quad \otimes_{\mathbb{T}\in F(\Delta)}  \mathcal{S}_{\omega}(\mathbb{T})\right),
\end{equation*} 
where the first tensor product runs over the inner edges of the triangulation and the second over the faces of the triangulation. Hence $\mathcal{S}_{\omega}(\mathbf{\Sigma})$ is completely determined by the combinatoric of the triangulation together with $\mathcal{S}_{\omega}(\mathbb{T})$ and its appropriated structures of comodule over $\mathcal{O}_q[\SL_2]$. This is a key feature in the proofs of the next two theorems.

\vspace{2mm}
\par Our second result is a generalization to open punctured surfaces of Bonahon and Wong's main theorem in \cite{BonahonWong1} in the case where the root of unity has odd order. Given $N\geq 1$, denote by $T_N(X)$ the $N$-th Chebyshev polynomial of first kind.
\begin{theorem}\label{theorem2} 
	Suppose that  $\omega$ is a root of unity of odd order $N\geq 1$. There exists an embedding 
	\begin{equation*}
	j_{\mathbf{\Sigma}} : \mathcal{S}_{+1}(\mathbf{\Sigma}) \hookrightarrow \mathcal{Z}\left( \mathcal{S}_{\omega}(\mathbf{\Sigma}) \right)
	\end{equation*}
	of the (commutative) stated skein algebra with parameter $+1$ into the center of the stated skein algebra with parameter $\omega$. Moreover, the morphism $j_{\mathbf{\Sigma}}$  is characterized by the property that it sends a closed curve $\gamma$ to $T_N(\gamma)$ and a stated arc $\alpha_{\varepsilon \varepsilon'}$ to $\alpha_{\varepsilon \varepsilon'}^{(N)}$, where $\alpha_{\varepsilon \varepsilon'}^{(N)}$ is the tangle made by stacking $N$ parallel copies of $\alpha_{\varepsilon \varepsilon'}$ on top of the others.
\end{theorem}
\par In Theorem \ref{theorem2}, we restrict ourselves to roots of unity of odd orders for simplicity.  Theorem \ref{theorem2} should be compared to \cite[Theorem $8.1$]{LePaprocki2018}.  A marked $3$-manifold is a pair $(M,\mathcal{N})$ where $M$ is 
an oriented $3$-manifold and $\mathcal{N}\subset \partial M$ is an oriented sub-manifold whose connected component are diffeomorphic to $[0,1]$. To such a pair and $\zeta \in \mathbb{C}^*$, the authors of \cite{LePaprocki2018} associate a vector space $\mathcal{S}_{\zeta}(M,\mathcal{N})$, which generalizes the M\"uller algebra. In \cite[Theorem $8.1$]{LePaprocki2018} and for a root of unity $\zeta$ such that $\zeta^4$ has arbitrary order $N>1$ (not necessary odd), the authors defined an injective linear map $\Phi_{\zeta}:\mathcal{S}_{(\zeta)^{N^2}}(M,\mathcal{N})\hookrightarrow \mathcal{S}_{\zeta}(M,\mathcal{N})$. If $(\Sigma, \mathcal{P})$ is a punctured surface with no inner punctures and non-trivial boundary, $(M,\mathcal{N}):=(\Sigma\times (0,1), \mathcal{P}\times(0,1))$ is a marked $3$-manifold and $\mathcal{S}_{\zeta}(M,\mathcal{N})$ is a subalgebra of the stated skein algebra $\mathcal{S}_{\zeta}(\Sigma, \mathcal{P})$. If $\zeta$ has odd order $N>1$, the embedding $j_{\mathbf{\Sigma}} $ of Theorem \ref{theorem2} restricts to the embedding $\Phi_{\zeta}$ of \cite[Theorem $8.1$]{LePaprocki2018}. A generalization of Theorem \ref{theorem2} for roots of unity of even order has been recently proved by Blooquist-L\^e in \cite[Theorem $1.2$]{BloomquistLe} though in this case the source of $j_{\mathbf{\Sigma}}$ is the skein algebra at $\eta:= \omega^{N^2}$ and the image is not always central but rather spanned by $(-1)^{1+N'}$-transparent elements, where $N':= \ord(\omega^4)$ (see \cite[Theorem $4.10$]{BloomquistLe} for details). Also a generalization of Theorem \ref{theorem2} for skein algebras of arbitrary connected reductive groups $G$ and for marked surfaces having $0$ or $1$ boundary arc was found by Ganev-Jordan-Safronov in \cite{GanevJordanSafranov_FrobeniusMorphism}.

\vspace{2mm}
	\par In the last result we generalize to open punctured surfaces Bullock's isomorphism of \cite{Bullock} and Turaev's theorem of \cite{Turaev91}; we  prove that the stated skein algebra is  a deformation of the relative character variety.  
	The fundamental result in this direction is as follows.  
\vspace{2mm}
	\par The $\mathbb{C}[[\hbar]]$-module $\mathcal{S}_{+1}(\mathbf{\Sigma})[[\hbar]]:= \mathcal{S}_{+1}(\mathbf{\Sigma})\otimes_{\mathbb{C}}\mathbb{C}[[\hbar]]$  is endowed with a star product $\star_{\hbar}$.
	The latter is obtained by pulling-back the product of $\mathcal{S}_{+1}(\mathbf{\Sigma})$ along an isomorphism  $\mathcal{S}_{+1}(\mathbf{\Sigma})[[\hbar]] \xrightarrow{\cong} \mathcal{S}_{\omega_{\hbar}}(\mathbf{\Sigma})$ of vector spaces, where $\omega_{\hbar}:= \exp\left( -\hbar/4 \right)$ (see Section \ref{sec_Poisson_skein} for details). 
	This equips  $\mathcal{S}_{+1}(\mathbf{\Sigma})$  with a Poisson algebra structure; its Poisson bracket $\left\{ \cdot, \cdot \right\}^{s}$ is defined  by  
	$$ f \star_{\hbar} g - g\star_{\hbar} f = \hbar \{ f, g\}^s \pmod{\hbar^2} \mbox{, for all }f,g \in \mathcal{S}_{+1}(\mathbf{\Sigma}). $$
	The superscript $s$ stands for "skein". See Section \ref{sec: explicit formula of bracket} for an explicit description.

\begin{theorem}\label{theorem3} 
	Suppose that  $\mathbf{\Sigma}$ has a topological triangulation $\Delta$. 
 Let $\mathfrak{o}_{\Delta}$ be an orientation of the edges of $\Delta$ and $\mathfrak{o}$ be the induced orientation of the boundary arcs of $\mathbf{\Sigma}$.
    There exists an isomorphism of Poisson algebras 
$$\Psi^{(\Delta, \mathfrak{o}_{\Delta})} : \left( \mathcal{S}_{+1}(\mathbf{\Sigma}), \{ \cdot, \cdot \}^s \right) \xrightarrow{\cong} \left( \mathbb{C}[\mathcal{X}_{\SL_2}(\mathbf{\Sigma})], \{\cdot, \cdot \}^{\mathfrak{o}} \right). $$
Moreover, the above isomorphism exists for small punctured surfaces (see Definition \ref{def small}), for which it only depends on  $\mathfrak{o}$. 
\end{theorem}

\par The isomorphism $\Psi^{(\Delta, \mathfrak{o}_{\Delta})}$ induces, by tensoring with $\mathbb{C}[[\hbar]]$, an isomorphism of vector spaces $\mathbb{C}[\mathcal{X}_{\SL_2}(\mathbf{\Sigma})] [[\hbar]] \xrightarrow{\cong}\mathcal{S}_{+1}(\mathbf{\Sigma})[[\hbar]] $. Denote by $\star_{(\Delta, \mathfrak{o}_{\Delta})}$ the product on $\mathbb{C}[\mathcal{X}_{\SL_2}(\mathbf{\Sigma})] [[\hbar]]$ obtained by pulling back the product $\star_{\hbar}$ by this isomorphism.

\begin{corollary}\label{corollary1} For any triangulable punctured surface $\mathbf{\Sigma}$, the algebra 
\\ $\left( \mathbb{C}[\mathcal{X}_{\SL_2}(\mathbf{\Sigma})] [[\hbar]] , \star_{(\Delta, \mathfrak{o}_{\Delta})} \right)$ is a deformation quantization of the character variety with Poisson structure given by $\mathfrak{o}$.
\end{corollary}

\vspace{2mm} \par Theorems \ref{theorem2} and \ref{theorem3}  allow us to extend Bonahon and Wong's \emph{classical shadow} \cite{BonahonWong1} to open punctured surfaces. 
Indeed, suppose that $\omega$ is a root of unity of odd order. 
A finite dimensional representation  $\mathcal{S}_{\omega}(\mathbf{\Sigma}) \rightarrow \End(V)$ that sends each element of the image of  $j_{\mathbf{\Sigma}} : \mathcal{S}_{+1}(\mathbf{\Sigma}) \hookrightarrow  \mathcal{S}_{\omega}(\mathbf{\Sigma})$ to scalar operators, induces a character on the algebra $\mathcal{S}_{+1}(\mathbf{\Sigma}) \cong \mathbb{C}[\mathcal{X}_{\SL_2}(\mathbf{\Sigma})]$, hence defines a point in $\mathcal{X}_{\SL_2}(\mathbf{\Sigma})$. 
To sum up, and calling \emph{central} these representations, one has the following.

\begin{corollary}\label{corollary2}
	When $\omega$ is a root of unity of odd order and $\mathbf{\Sigma}$ is triangulable, to each isomorphism class of central representations of the stated skein algebra $\mathcal{S}_{\omega}(\mathbf{\Sigma})$, one can associate an invariant which is a point in the relative character variety $\mathcal{X}_{\SL_2}(\mathbf{\Sigma})$.
\end{corollary}
\par Central representations include the families of irreducible representations, local representations and representations induced by simple modules of the balanced Chekhov-Fock algebras using the quantum trace map (see Section $3.3$ for details).

\vspace{2mm}
\par Soon after the prepublication of this paper on arXiv, Costantino and Le prepublished independently in \cite{CostantinoLe19} some results similar to Theorems \ref{theorem1} and \ref{theorem3}. More precisely, \cite[Theorem $4.7$]{CostantinoLe19} is identical to Theorem \ref{theorem1} whereas \cite[Theorem $8.12$]{CostantinoLe19} is closely related, though different, to our Theorem \ref{theorem3}. Instead of using the generalized character variety $\mathcal{X}_{\SL_2}(\mathbf{\Sigma})$ defined in \cite{KojuTriangularCharVar}, the authors defined a twisted character variety $\chi(\mathbf{\Sigma})$ (without Poisson structure) and constructed a canonical algebra isomorphism between the stated skein algebra in $+1$ and the algebra of regular functions of $\chi(\mathbf{\Sigma})$, whereas our isomorphism in Theorem \ref{theorem3} depends on the non-canonical choice $(\Delta, \mathfrak{o}_{\Delta})$ of a triangulation and an orientation of the edges (and is Poisson). Inspired by their enlightening approach, in this new version of the paper we add the following clarification of the isomorphism in Theorem \ref{theorem3}. As explained before, when the punctured surface is closed the "standard" isomorphisms between $\mathcal{S}_{+1}(\mathbf{\Sigma})$ and  $\mathbb{C}[\mathcal{X}_{\SL_2}(\mathbf{\Sigma})]$ are indexed by spin structures. In Section \ref{sec_relativespin}, we define the notion of \textit{relative spin structure} for punctured surfaces, which coincides with the standard definition when the punctured surface is closed. The motivation for this definition is its good behavior for the operation of gluing boundary arcs together. In particular we associate to each combinatorial data $(\Delta, \mathfrak{o}_{\Delta})$, appearing in Theorem \ref{theorem3}, a relative spin structure and prove the

\begin{theorem}\label{theoremN}
The isomorphism $\Psi^{(\Delta, \mathfrak{o}_{\Delta})}$ of Theorem \ref{theorem3} only depends on the relative spin structure associated to $(\Delta, \mathfrak{o})$.
\end{theorem}

In fact, in Theorem \ref{theorem5}, we provide explicit formulas for the value of $\Psi^{(\Delta, \mathfrak{o}_{\Delta})}$ on stated arcs and closed curves in terms of the relative spin structure. When the punctured surface is closed, we show that our isomorphism coincides with the standard isomorphism associated to classical spin structures. We also give, in Section \ref{sec_comparaisonCL}, a detailed comparison between the isomorphism in Theorem \ref{theorem3} and Costantino and L\^e isomorphism in  \cite[Theorem $8.12$]{CostantinoLe19}.
\vspace{2mm}
\par Even though our proof of Theorem \ref{theorem2} makes uses of triangulations, the theorem is proved for arbitrary punctured surfaces, including (non triangulable) closed surfaces without punctures thus providing an alternative proof of the results in \cite{BonahonWong1}. However, our proof of Theorem \ref{theorem3} only works for triangulable punctured surfaces (and for the bigon), so it does not provide an alternative proof of the result of \cite{Bullock} for closed unpunctured surfaces.

\subsubsection*{Plan of the paper}

In the second section we briefly recall from \cite{LeStatedSkein} the definition and general properties of the stated skein algebra and prove Theorem \ref{theorem1}. We then use the triangular decomposition to reduce the proof of Theorem \ref{theorem2} to the cases of the bigon and the triangle for which the proof is a simple computation. We eventually characterize the Poisson bracket arising in skein theory. In the third section, we briefly recall from \cite{KojuTriangularCharVar} the definition of character varieties for open surfaces. Again, using triangular decompositions, we reduce the proof of Theorem \ref{theorem3} to the cases of the bigon and the triangles for which the proof is elementary. We then introduce and study the notion of relative spin structure and give in Theorem \ref{theorem5} an explicit description of the isomorphism of \ref{theorem3}, from which Theorem \ref{theoremN} is a straightforward consequence. In the appendix, we prove a technical result needed in the proof of Theorem \ref{theorem2} and derive a generalization of the main theorem of \cite{BonahonMiraculous}.

\paragraph{Acknowledgments.} The  authors  thank F.Bonahon, F.Costantino, L.Funar, T.Q.T.L\^e and J.Toulisse for useful discussions. The first author also thanks the University of South California and the Federal University of S\~ao Carlos for their kind hospitality during the beginning of this work.  He acknowledges  support from the grant ANR  ModGroup, the GDR Tresses, the GDR Platon, CAPES,  the GEAR Network, the CNRS, the JSPS and the ERC DerSymp (Grant Agreement No. $768679$).  
The second author was supported by PNPD/CAPES-2013 during the first period of this project, and by "grant \#2018/19603-0, S\~ao Paulo Research Foundation (FAPESP)"  during the second period. The authors also thank the anonymous referees for valuable corrections and suggestions that improved the quality of the paper. The authors warmly thank T.Q.T.L\^e for correcting a mistake in a former version of the paper.

\begin{notations*}	All along the paper we reserve the following notations: $A:=\omega^{-2}$ and $q:=\omega^{-4}$.
\end{notations*}

\section{Stated skein algebras}

\subsection{Definitions and general properties of the stated skein algebras}\label{sec_definitions}

\par We briefly review from \cite{LeStatedSkein} the definition and main properties of the stated skein algebras.
\begin{definition}
A \textit{punctured surface} is a pair $\mathbf{\Sigma}=(\Sigma,\mathcal{P})$ where $\Sigma$ is a compact oriented surface and $\mathcal{P}$ is a finite subset of $\Sigma$ which intersects non-trivially each boundary component. A \textit{boundary arc} is a connected component of $\partial \Sigma \setminus \mathcal{P}$. The punctured surface is \textit{open} when $\partial \Sigma \neq \emptyset$ and \textit{closed} otherwise.
\end{definition}

\textbf{Definition of stated skein algebras}
\\   Let $\mathbf{\Sigma}=(\Sigma, \mathcal{P})$ be a punctured surface and write $\Sigma_{\mathcal{P}}:= \Sigma \setminus \mathcal{P}$.  A \textit{tangle} in $ \Sigma_{\mathcal{P}}\times (0,1)$   is a  compact framed, properly embedded $1$-dimensional manifold $T\subset \Sigma_{\mathcal{P}}\times (0,1)$ such that for every point of $\partial T \subset \partial \Sigma_{\mathcal{P}}\times (0,1)$ the framing is parallel to the $(0,1)$ factor and points to the direction of $1$.  Here, by framing, we refer to a thickening of $T$ to an oriented surface. Define the \textit{height} of a point $(v,h)\in \Sigma_{\mathcal{P}}\times (0,1)$ to be $h$.  If $b$ is a boundary arc and $T$ a tangle, the points of $\partial_b T := \partial T \cap b\times(0,1)$ are totally ordered by their height and we impose that no two points in $\partial_bT$ have the same height. A tangle has \textit{vertical framing} if for each of its points, the framing is parallel to the $(0,1)$ factor and points in the direction of $1$. Two tangles are isotopic if they are isotopic through the class of tangles that preserves the partial boundary height orders. By convention, the empty set is a tangle only isotopic to itself.

\vspace{2mm}
\par Every tangle is isotopic to a tangle with vertical framing. We can further isotope a tangle such that it is in general position with the standard projection  $\pi : \Sigma_{\mathcal{P}}\times (0,1)\rightarrow \Sigma_{\mathcal{P}}$ with $\pi(v,h)=v$, that is such that $\pi_{\big| T} : T\rightarrow \Sigma_{\mathcal{P}}$ is an immersion with at most transversal double points in the interior of $\Sigma_{\mathcal{P}}$. We call \textit{diagram} of $T$ the image $D=\pi(T)$ together with the over/undercrossing information at each double point. An isotopy class of diagram $D$ together with a total order of $\partial_b D=\partial D\cap b$ for each boundary arc $b$, define uniquely an isotopy class of tangle. Here isotopy of diagrams refers to isotopies where endpoints of diagrams are not allowed to cross. When choosing an orientation $\mathfrak{o}(b)$ of a boundary arc $b$ and a diagram $D$, the set $\partial_bD$ receives a natural total order $\leq_{\mathfrak{o}}$ by setting that the points are increasing when going in the direction of $\mathfrak{o}(b)$. We will represent tangles by drawing a diagram and an orientation (an arrow) for each boundary arc. When a boundary arc $b$ is oriented,   $\partial_b D$ is ordered by $\leq_{\mathfrak{o}}$ so according to the orientation. The data of an isotopy class of diagram $D$ and a choice $\mathfrak{o}$ of orientations of the  boundary arcs define uniquely an isotopy class of tangle $T$ by imposing that 
for every boundary arc $a$, for $v,w \in \partial_a D$ such that $v\leq_{\mathfrak{o}}w $ then the endpoint of $\partial_aT$ corresponding to $w$ has higher height than the endpoint corresponding to $v$.
 A \textit{state} of a tangle is a map $s:\partial T \rightarrow \{-, +\}$. A pair $(T,s)$ is called a \textit{stated tangle}. We define a \textit{stated diagram} $(D,s)$ in a similar manner.
\vspace{2mm}
\par \vspace{2mm}
\par  Let $\mathcal{R}$ be a commutative unital ring and $\omega\in \mathcal{R}^{\times}$ an invertible element.

\begin{definition}\label{def_stated_skein} 
  The \textit{stated skein algebra}  $\mathcal{S}_{\omega}(\mathbf{\Sigma})$ is the  free $\mathcal{R}$-module generated by isotopy classes of stated tangles in $\Sigma_{\mathcal{P}}\times (0, 1)$ modulo the relations \eqref{eq: skein 1} and \eqref{eq: skein 2}, which are,    
\\ 
the Kauffman bracket relations:
	\begin{equation}\label{eq: skein 1} 
\begin{tikzpicture}[baseline=-0.4ex,scale=0.5,>=stealth]	
\draw [fill=gray!45,gray!45] (-.6,-.6)  rectangle (.6,.6)   ;
\draw[line width=1.2,-] (-0.4,-0.52) -- (.4,.53);
\draw[line width=1.2,-] (0.4,-0.52) -- (0.1,-0.12);
\draw[line width=1.2,-] (-0.1,0.12) -- (-.4,.53);
\end{tikzpicture}
=\omega^{-2}
\begin{tikzpicture}[baseline=-0.4ex,scale=0.5,>=stealth] 
\draw [fill=gray!45,gray!45] (-.6,-.6)  rectangle (.6,.6)   ;
\draw[line width=1.2] (-0.4,-0.52) ..controls +(.3,.5).. (-.4,.53);
\draw[line width=1.2] (0.4,-0.52) ..controls +(-.3,.5).. (.4,.53);
\end{tikzpicture}
+\omega^{2}
\begin{tikzpicture}[baseline=-0.4ex,scale=0.5,rotate=90]	
\draw [fill=gray!45,gray!45] (-.6,-.6)  rectangle (.6,.6)   ;
\draw[line width=1.2] (-0.4,-0.52) ..controls +(.3,.5).. (-.4,.53);
\draw[line width=1.2] (0.4,-0.52) ..controls +(-.3,.5).. (.4,.53);
\end{tikzpicture}
\hspace{.5cm}
\text{ and }\hspace{.5cm}
\begin{tikzpicture}[baseline=-0.4ex,scale=0.5,rotate=90] 
\draw [fill=gray!45,gray!45] (-.6,-.6)  rectangle (.6,.6)   ;
\draw[line width=1.2,black] (0,0)  circle (.4)   ;
\end{tikzpicture}
= -(\omega^{-4}+\omega^{4}) 
\begin{tikzpicture}[baseline=-0.4ex,scale=0.5,rotate=90] 
\draw [fill=gray!45,gray!45] (-.6,-.6)  rectangle (.6,.6)   ;
\end{tikzpicture}
;
\end{equation}
the boundary relations:
\begin{equation}\label{eq: skein 2} 
\begin{tikzpicture}[baseline=-0.4ex,scale=0.5,>=stealth]
\draw [fill=gray!45,gray!45] (-.7,-.75)  rectangle (.4,.75)   ;
\draw[->] (0.4,-0.75) to (.4,.75);
\draw[line width=1.2] (0.4,-0.3) to (0,-.3);
\draw[line width=1.2] (0.4,0.3) to (0,.3);
\draw[line width=1.1] (0,0) ++(90:.3) arc (90:270:.3);
\draw (0.65,0.3) node {\scriptsize{$+$}}; 
\draw (0.65,-0.3) node {\scriptsize{$+$}}; 
\end{tikzpicture}
=
\begin{tikzpicture}[baseline=-0.4ex,scale=0.5,>=stealth]
\draw [fill=gray!45,gray!45] (-.7,-.75)  rectangle (.4,.75)   ;
\draw[->] (0.4,-0.75) to (.4,.75);
\draw[line width=1.2] (0.4,-0.3) to (0,-.3);
\draw[line width=1.2] (0.4,0.3) to (0,.3);
\draw[line width=1.1] (0,0) ++(90:.3) arc (90:270:.3);
\draw (0.65,0.3) node {\scriptsize{$-$}}; 
\draw (0.65,-0.3) node {\scriptsize{$-$}}; 
\end{tikzpicture}
=0,
\hspace{.2cm}
\begin{tikzpicture}[baseline=-0.4ex,scale=0.5,>=stealth]
\draw [fill=gray!45,gray!45] (-.7,-.75)  rectangle (.4,.75)   ;
\draw[->] (0.4,-0.75) to (.4,.75);
\draw[line width=1.2] (0.4,-0.3) to (0,-.3);
\draw[line width=1.2] (0.4,0.3) to (0,.3);
\draw[line width=1.1] (0,0) ++(90:.3) arc (90:270:.3);
\draw (0.65,0.3) node {\scriptsize{$+$}}; 
\draw (0.65,-0.3) node {\scriptsize{$-$}}; 
\end{tikzpicture}
=\omega
\begin{tikzpicture}[baseline=-0.4ex,scale=0.5,>=stealth]
\draw [fill=gray!45,gray!45] (-.7,-.75)  rectangle (.4,.75)   ;
\draw[-] (0.4,-0.75) to (.4,.75);
\end{tikzpicture}
\hspace{.1cm} \text{ and }
\hspace{.1cm}
\omega^{-1}
\heightexch{->}{-}{+}
- \omega^{-5}
\heightexch{->}{+}{-}
=
\heightcurve.
\end{equation}
According to our graphical conventions, in these skein relations, the boundary points are consecutive in the height order
The product of two classes of stated tangles $[T_1,s_1]$ and $[T_2,s_2]$ is defined by  isotoping $T_1$ and $T_2$  in $\Sigma_{\mathcal{P}}\times (1/2, 1) $ and $\Sigma_{\mathcal{P}}\times (0, 1/2)$ respectively and then setting $[T_1,s_1]\cdot [T_2,s_2]=[T_1\cup T_2, s_1\cup s_2]$. 
\end{definition}

\vspace{2mm}
\par 
\textbf{Bases for stated skein algebras}
\\ A closed component of a diagram $D$ is trivial if it bounds an embedded disc in $\Sigma_{\mathcal{P}}$. An open component of $D$ is trivial if it can be isotoped, relatively to its boundary, inside some boundary arc. A diagram is \textit{simple} if it has neither double points nor trivial component. The empty set is considered as a simple diagram. Let $\mathfrak{o}$ be an orientation of the boundary arcs of $\mathbf{\Sigma}$ and denote by $\leq_{\mathfrak{o}}$ the total orders induced on each boundary arc. A state $s: \partial D \rightarrow \{ - , + \}$ is $\mathfrak{o}-$\textit{increasing} if for any boundary arc $b$ and any points $x,y \in \partial_bD$, then $x<_{\mathfrak{o}} y$ implies $s(x)< s(y)$. Here we choose the convention $- < +$. 
We denote by $[D,s]\in \mathcal{S}_{\omega}(\mathbf{\Sigma})$ the class of the stated tangle associated to $(D,s)$ (note that $[D,s]$ depends on the orientation $\mathfrak{o}$).

\begin{definition}\label{def_basis}
We denote by $\mathcal{B}^{\mathfrak{o}}\subset \mathcal{S}_{\omega}(\mathbf{\Sigma})$ the set of classes  $[D,s]$ such that $D$ is simple and $s$ is $\mathfrak{o}$-increasing. 
\end{definition}

\begin{theorem}(\cite[Theorem $2.11$]{LeStatedSkein})\label{theorem_basis}  The set $\mathcal{B}^{\mathfrak{o}}$ is an $\mathcal{R}$-module  basis of $\mathcal{S}_{\omega}(\mathbf{\Sigma})$. 
\end{theorem}

\vspace{2mm}
\par 
\textbf{Height exchange moves}
\\ Important properties that we will use all along the paper are the following \textit{height exchange moves} \eqref{eq: height exch 1} and \eqref{eq: height exch corr} proved in 
\cite[Lemma $2.4$]{LeStatedSkein}. 
Note that the formula (20) of Lemma 2.4 of \textit{loc. cit.} contains a misprint. It is corrected here in \eqref{eq: height exch corr}. 
\begin{equation}\label{eq: height exch 1}
\heightexch{<-}{+}{+}=\omega^{2} \heightexch{->}{+}{+}, ~~~
\heightexch{<-}{+}{-}=\omega^{-2} \heightexch{->}{+}{-}, ~~~
\heightexch{<-}{-}{-}=\omega^{2} \heightexch{->}{-}{-}
\end{equation}
\begin{equation}\label{eq: height exch corr}
\omega^{-3}\heightexch{<-}{-}{+}-\omega^{3}\heightexch{->}{-}{+}=(\omega^{-4} -\omega^{4}) \heightcurve.  
\end{equation}

	\begin{remark}
	An important case that we will be led to consider is the stated skein algebra at parameter $\omega=+1$. As shown in \cite[Corollary 2.5]{LeStatedSkein} it is commutative; this is a direct consequence of \eqref{eq: skein 1} and the height exchange formulas \eqref{eq: height exch 1} and  \eqref{eq: height exch corr}. 
	\end{remark}

\vspace{2mm}
\par 
\textbf{Trivial arcs relations}
\\ We will also use the following \textit{trivial arcs relations}. Set 
$$ C= \begin{pmatrix} C_+^+ & C_-^+ \\ C_+^- & C_-^- \end{pmatrix} := \begin{pmatrix} 0 & \omega \\ -\omega^{5} & 0 \end{pmatrix} \mbox{ and }C^{-1}= -A^3 C =  \begin{pmatrix} 0 & -\omega^{-5} \\ \omega^{-1} & 0 \end{pmatrix}.$$

\begin{lemma}[\cite{LeStatedSkein} Lemma $2.3$] One has the following \textit{trivial arcs relations}:

\begin{equation}\label{trivial_arc_rel}
\begin{tikzpicture}[baseline=-0.4ex,scale=0.5,>=stealth]
\draw [fill=gray!45,gray!45] (-.7,-.75)  rectangle (.4,.75)   ;
\draw[->] (0.4,-0.75) to (.4,.75);
\draw[line width=1.2] (0.4,-0.3) to (0,-.3);
\draw[line width=1.2] (0.4,0.3) to (0,.3);
\draw[line width=1.1] (0,0) ++(90:.3) arc (90:270:.3);
\draw (0.65,0.3) node {\scriptsize{$i$}}; 
\draw (0.65,-0.3) node {\scriptsize{$j$}}; 
\end{tikzpicture}
= C^i_j 
\hspace{.2cm}
\begin{tikzpicture}[baseline=-0.4ex,scale=0.5,>=stealth]
\draw [fill=gray!45,gray!45] (-.7,-.75)  rectangle (.4,.75)   ;
\draw[-] (0.4,-0.75) to (.4,.75);
\end{tikzpicture}
\hspace{.4cm} \mbox{and}  \hspace{.4cm}
\begin{tikzpicture}[baseline=-0.4ex,scale=0.5,>=stealth]
\draw [fill=gray!45,gray!45] (-.7,-.75)  rectangle (.4,.75)   ;
\draw[->] (-0.7,-0.75) to (-.7,.75);
\draw[line width=1.2] (-0.7,-0.3) to (-0.3,-.3);
\draw[line width=1.2] (-0.7,0.3) to (-0.3,.3);
\draw[line width=1.15] (-.4,0) ++(-90:.3) arc (-90:90:.3);
\draw (-0.9,0.3) node {\scriptsize{$i$}}; 
\draw (-0.9,-0.3) node {\scriptsize{$j$}}; 
\end{tikzpicture}
=(C^{-1})^i_j 
\hspace{.2cm}
\begin{tikzpicture}[baseline=-0.4ex,scale=0.5,>=stealth]
\draw [fill=gray!45,gray!45] (-.7,-.75)  rectangle (.4,.75)   ;
\draw[-] (-0.7,-0.75) to (-0.7,.75);
\end{tikzpicture}.
\end{equation}

\end{lemma}

\vspace{2mm}
\par 
\textbf{Splitting morphisms}
\\ Suppose that $\mathbf{\Sigma}$ has two boundary arcs, say $a$ and $b$. Let   $\mathbf{\Sigma}_{|a\#b}$ be the punctured surface obtained from $\mathbf{\Sigma}$ by gluing $a$ and $b$. Denote by $\pi : \Sigma_{\mathcal{P}} \rightarrow (\Sigma_{|a\#b})_{\mathcal{P}_{|a\#b}}$ the projection and $c:=\pi(a)=\pi(b)$. Let $(T_0, s_0)$ be a stated framed tangle of ${\Sigma_{|a\#b}}_{\mathcal{P}_{|a\#b}} \times (0,1)$ transversed to $c\times (0,1)$ and such that the heights of the points of $T_0 \cap c\times (0,1)$ are pairwise distinct and such that framings of the points of $c\times (0,1)$ are vertical. Let $T\subset \Sigma_{\mathcal{P}}\times (0,1)$ be the framed tangle obtained by cutting $T_0$ along $c$. 
Using the partition $\partial T = \partial_a T \bigsqcup \pi^{-1}(\partial T_0) \bigsqcup \partial_b T$, a state on $T$ can be written $(s_a, s, s_b)$ where $s_a$, $s$ and $s_b$ are states on $\partial_a T$, $\partial T_0$ and $\partial_b T$ respectively.
Both the sets $\partial_a T$ and $\partial_b T$ are in canonical bijection with the set $T_0\cap c$ by the map $\pi$. Hence the two sets of states $s_a$ and $s_b$ are both in canonical bijection with the set $\mathrm{St}(c):=\{ s: c\cap T_0 \rightarrow \{-,+\} \}$. Let $i_{|a\#b}: \mathcal{S}_{\omega}(\mathbf{\Sigma}_{|a\#b}) \rightarrow \mathcal{S}_{\omega}(\mathbf{\Sigma})$ be the linear map given, for any $(T_0, s_0)$ as above, by: 
$$ i_{|a\#b} \left( [T_0,s_0] \right) := \sum_{s \in \mathrm{St}(c)} [T, (s, s_0 , s) ].$$

\begin{theorem}\label{theorem_gluingmap}\cite[Theorem $3.1$]{LeStatedSkein} 
The linear map $i_{|a\#b}$ is an injective morphism of algebras. Moreover the gluing operation is coassociative in the sense that if $a,b,c,d$ are four distinct boundary arcs, then we have $i_{|a\#b} \circ i_{|c\#d} = i_{|c\#d} \circ i_{|a\#b}$.
\end{theorem}
Note that the splitting morphism $i_{a\#b}$ does not depend on any choice of the boundary arcs.

\vspace{2mm}
\par 
\textbf{Triangulations}
\\ 

\begin{definition}\label{def small}
	A \emph{small} punctured surface is one of the following four connected punctured surfaces: the sphere with one or two punctures; the disc with only one puncture (on its boundary); the bigon (disc with two punctures on its boundary).   
\end{definition}

\begin{definition}
	A punctured surface is said to \emph{admit a triangulation} if each of its connected components has at least one puncture and is not small. 
\end{definition}
\begin{definition}\label{def_triangulation} Suppose $\mathbf{\Sigma}=(\Sigma, \mathcal{P})$ admits a triangulation.  
		A \textit{topological triangulation} $\Delta$ of $\mathbf{\Sigma}$ is a collection $\mathcal{E}(\Delta)$ of arcs in $\Sigma$ (named edges) which satisfy the following conditions: 
		the endpoints of the edges belong to $\mathcal{P}$;  
		the interior of the edges are pairwise disjoint and do not intersect $\mathcal{P}$; 
		the edges are not contractible and are pairwise non isotopic in $\Sigma_{\mathcal{P}}$,  if fixed their endpoints; 
		the boundary arcs of $\mathbf{\Sigma}$ belong to $\mathcal{E}(\Delta)$. 
	Moreover, the collection $\mathcal{E}(\Delta)$ is required to be maximal for these properties.  
\end{definition}

\par Each connected component of  $\Sigma \setminus \mathcal{E}(\Delta)$ is called a \emph{face} and the set of faces is denoted by $F(\Delta)$.
    Given a topological triangulation $\Delta$, the punctured surface is obtained from the disjoint union
    $\bigsqcup_{\mathbb{T}\in F(\Delta)} \mathbb{T}$ of triangles by gluing the triangles along the boundary arcs corresponding to the edges of the triangulation.
    Very often, we will let $\mathbb{T}$ be both a face (which is an open contractible space) and the triangle (which is a disc with exactly three punctures on its boundary). We hope that this abuse of notation is harmless.
    By composing the associated splitting morphisms, one obtains an injective morphism of algebras:
    $$i^{\Delta} : \mathcal{S}_{\omega}(\mathbf{\Sigma}) \hookrightarrow \otimes_{\mathbb{T} \in F(\Delta)} \mathcal{S}_{\omega}(\mathbb{T}). $$
\vspace{2mm}
\par 
\textbf{Filtrations}
\\ 

\par The stated skein algebra has natural filtrations defined as follows. Let $S=\{a_1, \ldots, a_n\}$ be a set of boundary arcs of $\mathbf{\Sigma}$ and fix an orientation $\mathfrak{o}$ of the boundary arcs of $\mathbf{\Sigma}$. For a basis element $[D,s]$ of $\mathcal{B}^{\mathfrak{o}}$, write $d([D,s]):= \sum_{a\in S} \big| \partial_a D \big|$. The map $d$ extends to a  map $d: \mathcal{S}_{\omega}(\mathbf{\Sigma}) \rightarrow \mathbb{Z}^{\geq 0}$ by the formula $d(\sum_i x_i [D_i, s_i]):= \max_{i | x_i\neq 0}d([D_i,s_i])$. It follows from the relations \eqref{eq: skein 1} and \eqref{eq: skein 2}  that for each $x,y\in \mathcal{S}_{\omega}(\mathbf{\Sigma})$, we have $d(xy)\leq d(x)+d(y)$. Given $m\geq 0$, denote by $\mathcal{F}_m\subset \mathcal{S}_{\omega}(\mathbf{\Sigma})$ the sub-vector space of those vectors $x$ satisfying $d(x)\leq m$. These sub-spaces satisfy $\mathcal{F}_{m}\subset \mathcal{F}_{m+1}$, $\mathcal{S}_{\omega}(\mathbf{\Sigma}) = \bigcup_{m\geq 0} \mathcal{F}_m$ and $\mathcal{F}_{m_1} \cdot \mathcal{F}_{m_2} \subset \mathcal{F}_{m_1+m_2}$, hence they form an algebra filtration of the stated skein algebra.

\begin{definition}\label{def_filtration}
The sequence $(\mathcal{F}_m)_{m\geq 0}$ is called the \textit{filtration} of $\mathcal{S}_{\omega}(\mathbf{\Sigma})$ associated to the orientation $\mathfrak{o}$ and the set $S$ of boundary arcs. For an element $X= \sum_{i\in I} x_i [D_i, s_i] \in \mathcal{S}_{\omega}(\mathbf{\Sigma})$, developed in the basis $\mathcal{B}^{\mathfrak{o}}$, we call \textit{leading term} of $X$ the element:
$$\lt (X) := \sum_{j\in I | d([D_j, s_j])= d(X)} x_j [D_j, s_j]. $$
\end{definition}

\subsection{Alternative bases}\label{sec_basecheloud}

In the next subsection, we will need alternative bases of $\mathcal{S}_{\omega}(\mathbf{\Sigma})$ which we now introduce. We fix an arbitrary orientation $\mathfrak{o}$ for each boundary arcs. Recall that $\mathfrak{o}$ induces a total order $\leq_{\mathfrak{o}}$ on each boundary arc that we use to associate a tangle to a diagram.

\begin{notations}
Let $\mathcal{D}(\mathbf{\Sigma})$ be the set of isotopy classes of simple diagrams and $\mathcal{CD}(\mathbf{\Sigma})$ be its subset of classes of connected diagrams. Fix an arbitrary total order $\prec$ on $\mathcal{CD}(\mathbf{\Sigma})$ and fix an orientation $\mathfrak{o}$ of the boundary arcs of $\mathbf{\Sigma}$ as before.
For $[D]\in \mathcal{CD}(\mathbf{\Sigma})$, we denote by $[T(D)]$ the isotopy class of the tangle $T(D)$ with vertical framing whose projection is $D$ and such that if $\partial T(D) = \{ v_1, v_2\}$ with $v_1, v_2$ in the same boundary arc $a$ with $v_1 \leq_{\mathfrak{o}} v_2$, then $h(v_1)<h(v_2)$. 
 For a general class of diagram $[D] \in \mathcal{D}(\mathbf{\Sigma})$ with connected components $D= \bigsqcup_{i=1}^n D_i$, where $[D_i] \preceq [D_{i+1}]$, we denote by $[T(D)]$ the class of the tangle $T(D):= \bigsqcup_{i=1}^n T(D_i)$ in $\Sigma_{\mathcal{P}}\times (0,1)$, where $T(D_{i+1})$ is on the top of $T(D_i)$ in the height direction. See Figure \ref{fig_TD} for an illustration.
 Let $\nu: \partial D \xrightarrow{\cong} \partial T(D)$ be the unique bijection such that, for $a$ a boundary arc, $\nu$ restricts to a bijection $\nu_{| a}: \partial_a D \to \partial_a T(D)$ which preserves the
 order $\leq_{\mathfrak{o}}$ on $\partial_a D$ and the 
  height order on $\partial_a T(D)$. Recall that $\partial_aD=D\cap a$ and that $\partial_aT(D)= T(D)\cap a\times (0,1)$.
 A state $s$ on $D$ defines a state $s\circ \nu^{-1}$ on $T(D)$ and we denote by $[T(D),s]$ the class of the stated tangle $(T(D),s\circ \nu^{-1})$.
\end{notations}

\begin{figure}[!h] 
\centerline{\includegraphics[width=5cm]{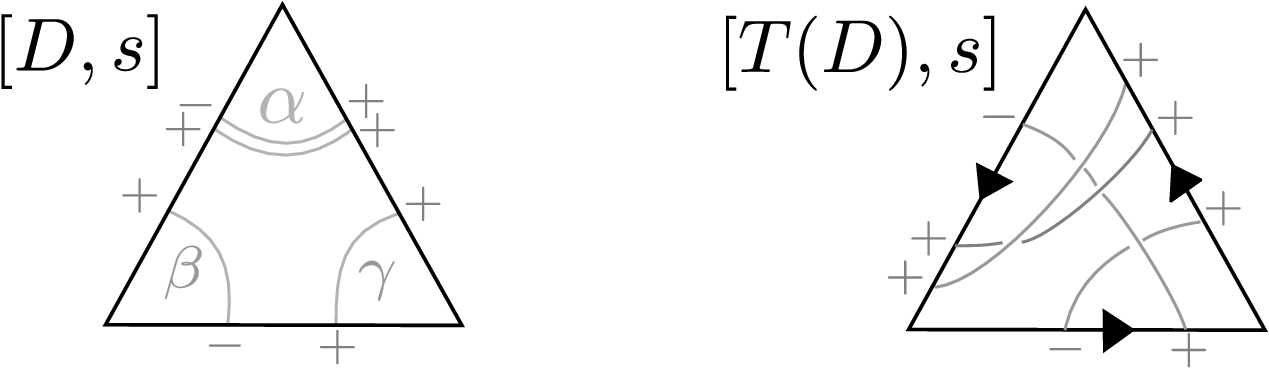} }
\caption{A stated diagram $[D,s]$ in the triangle and its associated stated tangle $[T(D), s]$. Here, we use the order $\gamma \prec \beta \prec \alpha$. Here $s$ is $\mathfrak{o}$-increasing so $[T(D),s]\in \mathcal{TB}^{\mathfrak{o}}$. } 
\label{fig_TD} 
\end{figure} 

\begin{definition}
We denote by $\mathcal{TB}^{\mathfrak{o}} \subset \mathcal{S}_{\omega}(\mathbf{\Sigma})$ the set of classes $[T(D),s]$ with $[D]\in \mathcal{D}(\mathbf{\Sigma})$ and $s$ an $\mathfrak{o}$-increasing state.
\end{definition}

Note that in our pictures the orientation $\mathfrak{o}$ is never represented, the arrows always refer to the height order and not to $\mathfrak{o}$.
The following lemma was proved in \cite{LeStatedSkein}, during the proof of Theorem $4.6$, in the particular case where $\mathbf{\Sigma}$ is a triangle. 

\begin{proposition}\label{prop_alternative_basis}
The set $\mathcal{TB}^{\mathfrak{o}} $ is a basis of $\mathcal{S}_{\omega}(\mathbf{\Sigma})$.
\end{proposition}

As an immediate consequence of Proposition \ref{prop_alternative_basis}, we get the:

\begin{corollary}
The stated skein algebra is algebraically generated by the classes of closed curves and stated arcs.
\end{corollary}

Here by closed curves and stated arcs we mean connected stated diagrams with no crossing which are closed and open respectively. 
Obviously, it is sufficient to prove Proposition \ref{prop_alternative_basis} in the case where $\Sigma$ is connected. If $\partial \Sigma=\emptyset$ or if $\mathbf{\Sigma}$ is a disc with one puncture on its boundary or a bigon whose boundary arcs points towards the same puncture, then $\mathcal{TB}^{\mathfrak{o}}=\mathcal{B}^{\mathfrak{o}}$ so the proposition follows from Theorem \ref{theorem_basis} in those cases. For the bigon whose boundary arcs points towards distinct punctures, Proposition \ref{prop_alternative_basis} was proved in the Step $1$ of the proof of Theorem $4.1$ in \cite{LeStatedSkein}.
So we now assume that $\mathbf{\Sigma}$ admits a topological triangulation $\Delta$ that we fix.
The proof of Proposition \ref{prop_alternative_basis}  is an easy adaption of L\^e's argument from the case of the triangle to the case of a triangulable punctured surface. The key feature is to consider a suitable filtration that we now introduce. 
\vspace{2mm}
\par
For a diagram $D$ and an edge $e\in \mathcal{E}(\Delta)$, we denote by $i(D, e) \in \mathbb{N}$ the geometric intersection of $D$ with $e$, that is the minimal number of intersection points when isotoping $D$ in such a way that it intersects $e$ transversally. We write
$$ |D| := \sum_{e \in \mathcal{E}(\Delta)} i(D,  e), $$
and, for $i\in \mathbb{N}$, we set

$$ \mathcal{F}_i := \mathrm{Span} \left\{ [D,s] \mbox{, such that }|D|\leq i  \right\}.$$

\begin{lemma}\label{lemma_filtration_s}
\begin{enumerate}
\item One has $\mathcal{F}_i \cdot \mathcal{F}_j \subset \mathcal{F}_{i+j}$.
\item The submodule $\mathcal{F}_i$ has basis the set $B_i$ of elements $[D,s]\in \mathcal{B}^{\mathfrak{o}}$ such that $|D|\leq i$.
\item For $[D,s] \in \mathcal{B}^{\mathfrak{o}}$, there exists $n\in \mathbb{Z}$ such that 
$$ [T(D), s] \equiv A^n[D,s] \pmod{\mathcal{F}_{|D|-2}}. $$
\end{enumerate}
\end{lemma}

\begin{proof}

$(1)$ Let $[D_1,s_1]$ and $[D_2,s_2]$ be two classes such that: 

\begin{itemize}
\item[(i)]  $D_1\cup D_2$ has only transversed double intersection points in the interior of $\Sigma_{\mathcal{P}}$ away from the edges of $\Delta$,  and
\item[(ii)] $D_1$ and $D_2$ are transversed to the edges of $\mathcal{E}(\Delta)$ with minimal intersection such that $|D_i| = |D_i \cap \mathcal{E}(\Delta)|$ for $i=1,2$.
\end{itemize}

 Let $D$ denote the diagram obtained by stacking $D_1$ on top of $D_2$ and $s$ the state corresponding to $s_1,s_2$ so that $[D,s]=[D_1,s_1][D_2,s_2]$. Then $|D| \leq |D \cap \mathcal{E}(\Delta)| = |D_1| + |D_2|$. Therefore $[D_1,s_1][D_2,s_2] \in \mathcal{F}_{|D_1|+|D_2|}$ and the first assertion is proved.
 \vspace{2mm}
 \par $(2)$  To prove the second assertion, first note that since $B_i$ is a subset of $\mathcal{B}^{\mathfrak{o}}$, it is free. We need to show that $B_i$ generates $\mathcal{F}_i$.  
 We proceed in two steps:
 
 \textit{Step 1:} We first prove that any class of stated diagram $[D,s]$ is a linear combination of elements $[D_i,s_i]$ with $|D_i|= |D|$ and such that $D_i$ has no crossing.

\textit{Step 2:} We then prove that any $[D,s]$ where $D$ has no crossing, is a linear combination of elements of $B_{|D|}$. 

The two steps imply that $B_i$ generates $\mathcal{F}_i$ and conclude the proof of the second assertion.

To prove the first step, fix an arbitrary stated diagram $(D,s)$. A \textit{resolution} of $D$ is a diagram obtained from $D$ by replacing each crossing 
$\begin{tikzpicture}[baseline=-0.4ex,scale=0.5,>=stealth]	
\draw [fill=gray!45,gray!45] (-.6,-.6)  rectangle (.6,.6)   ;
\draw[line width=1.2,-] (-0.4,-0.52) -- (.4,.53);
\draw[line width=1.2,-] (0.4,-0.52) -- (0.1,-0.12);
\draw[line width=1.2,-] (-0.1,0.12) -- (-.4,.53);
\end{tikzpicture}$
by either 
$\begin{tikzpicture}[baseline=-0.4ex,scale=0.5,>=stealth] 
\draw [fill=gray!45,gray!45] (-.6,-.6)  rectangle (.6,.6)   ;
\draw[line width=1.2] (-0.4,-0.52) ..controls +(.3,.5).. (-.4,.53);
\draw[line width=1.2] (0.4,-0.52) ..controls +(-.3,.5).. (.4,.53);
\end{tikzpicture}$
(positive resolution of the crossing )or $\begin{tikzpicture}[baseline=-0.4ex,scale=0.5,rotate=90]	
\draw [fill=gray!45,gray!45] (-.6,-.6)  rectangle (.6,.6)   ;
\draw[line width=1.2] (-0.4,-0.52) ..controls +(.3,.5).. (-.4,.53);
\draw[line width=1.2] (0.4,-0.52) ..controls +(-.3,.5).. (.4,.53);
\end{tikzpicture}$
(negative resolution).
Write $\mathrm{Res}(D)$ the set of resolutions and for $D_0 \in \mathrm{Res}(D)$, denote by $n(D_0)$ the difference between the numbers of positive and negative resolution crossings in $D_0$. Then, by the Kauffman-bracket skein relation \eqref{eq: skein 1}, one has 
$$
[D,s] = \sum_{D_i \in \mathrm{Res}(D)} A^{n(D_i)} [D_i, s], 
$$
where for each resolution $D_i$, one has $|D_i \cap \mathcal{E}(\Delta)|=|D\cap \mathcal{E}(\Delta)|=|D|$, so $|D_i| = |D|$ and Step $1$ is proved.
\vspace{2mm}
\par
 To prove the second step, consider a stated diagram $(D,s)$ where $D$ has no crossing. If $s$ is $\mathfrak{o}$-increasing, let $(D^{\prime}, s)$ be the stated diagram obtained from $(D,s)$ by removing its trivial components, so $|D^{\prime}| \leq |D|$. Then there exists a scalar $c$ such that $[D,s]=c [D^{\prime}, s]$ and $[D^{\prime}, s]\in B_{|D|}$.
 Else, we show by induction on the number $m(D,s)$ of pairs of points $v<_{\mathfrak{o}} w$ in $\partial D$ lying in the same boundary arc such that $(s(v),s(w))=(+,-)$, that $(D,s)$ is a linear combination of elements of $B_{|D|}$. Consider such a pair $(v,w)$ of points which are consecutive for $<_{\mathfrak{o}}$, and let $s'$ be the state on $D$ which agrees with $s$ on $\partial D \setminus \{v,w\}$ and such that $(s'(v), s'(w))=(-,+)$. The following skein relations:
\begin{equation*}
\heightcurve =
\omega^{-1}
\heightexch{->}{-}{+}
- \omega^{-5}
\heightexch{->}{+}{-}
, \hspace{.2cm} 
\heightcurveright
= \omega
\heightexchright{->}{+}{-}
-\omega^5 
\heightexchright{->}{-}{+}, 
\end{equation*}
shows that there exists $n\in \mathbb{Z}$ such that $[D,s]\equiv \omega^n [D, s'] \pmod{\mathcal{F}_{|D|-1}}$ (because the stated diagram representing either the term $\heightcurve$ or $\heightcurveright$ is in $\mathcal{F}_{|D|-1}$).
 Since $m(D,s')< m(D,s)$, we conclude by decreasing induction on $m$ that $[D,s]$ is a linear combination of elements $[D_i,s_i]$ where $D_i$ has no crossing and $s_i$ is $\mathfrak{o}$-increasing. Now writes $[D_i,s_i]= c_i [D_i^{\prime}, s_i]$, where $c_i$ is a scalar and $(D_i^{\prime},s_i)$ is obtained from $(D_i,s_i)$ by removing its trivial components so that $[D^{\prime}_i,s_i]\in B_{|D|}$. This concludes Step 2 and the proof of the second item.

\vspace{2mm}
\par 
$(3)$ Let us first make an obvious but useful remark. Let $D$ be a diagram transversed to $\mathcal{E}(\Delta)$. We say that $D$ \textit{contains a returning arc} if there exists a face $\mathbb{T}$ such that $D\cap \mathbb{T}$ contains a connected component that is an arc with both endpoints in the same edge. If $D$ contains a returning arc, then $D$ is not in minimal intersection position with respect to $\mathcal{E}(\Delta)$ so for all states $s$, $[D,s]\in \mathcal{F}_{|D|-2}$.

Now consider $[D,s]\in \mathcal{B}^{\mathfrak{o}}$ and denote by $TD$ the projection diagram of the tangle $T(D)$ so that $[T(D),s]=[TD,s]$ (think of Figure \ref{fig_TD}). We further suppose that $TD$ is transversed to $\mathcal{E}(\Delta)$ in minimal position and has its crossings outside $\mathcal{E}(\Delta)$. 
In the decomposition
$$ [TD,s] = \sum_{D_i \in \mathrm{Res}(TD)} A^{n(D_i)} [D_i, s], $$
we claim that there is exactly one resolution $D_0 \in \mathrm{Res}(TD)$ such that $D_0=D$ and that any other resolution $D_i\neq D_0$ contains a returning arc, so satisfies $[D_i, s_i] \in \mathcal{F}_{|D|-2}$. Since resolving a crossing is a local operation, it is sufficient to prove the claim in the case of the triangle; this was done by L\^e in \cite[Lemma $4.7$]{LeStatedSkein}. Recall that L\^e's proof consists noting that if $[T(D),s]$ has two connected components, it has $0$ or $1$ crossing (after eventually isotoping $TD$) and when there is one crossing in $TD$, exactly one of the two resolutions does not contain returning arc. The results then follows by induction on the number of components of $T(D)$ using the fact that the arcs in $T(D)$ are stacked on top of each other.

So we have $[T(D),s] \equiv A^{n(D)} [D,s] \pmod{\mathcal{F}_{|D|-2}}$ and the proof is completed.

\end{proof}

Obviously one has $\mathcal{F}_i \subset \mathcal{F}_{i+1}$ and $\cup_{i\geq 0} \mathcal{F}_i = \mathcal{S}_{\omega}(\mathbf{\Sigma})$. The first assertion of Lemma \ref{lemma_filtration_s} implies that $(\mathcal{F}_i)_{i\geq 0}$ forms an algebra filtration of $\mathcal{S}_{\omega}(\mathbf{\Sigma})$.
 Consider the graded algebra $\mathbf{Gr}_{\bullet}$ associated to the filtration. In other words,  we set $\mathbf{Gr}_0:= \mathcal{F}_0$,  $\mathbf{Gr}_i := \quotient{\mathcal{F}_i}{\mathcal{F}_{i-1}}$ for $i\geq 1$ and $\mathbf{Gr}_{\bullet} := \oplus_{i\geq 0} \mathbf{Gr}_i$. It follows from the second assertion of  Lemma \ref{lemma_filtration_s} that $\mathbf{Gr}_i$ has basis the set $\mathcal{B}_i^{\mathfrak{o}}$ of classes $[D,s] \in \mathcal{B}^{\mathfrak{o}}$ such that $|D|=i$. Since the set $\{ \mathcal{B}_i^{\mathfrak{o}} \}_{i\geq 0}$ forms a partition of $\mathcal{B}^{\mathfrak{o}}$, the natural graded morphism $\psi : \mathcal{S}_{\omega}(\mathbf{\Sigma}) \rightarrow \mathbf{Gr}_{\bullet}$ is an isomorphism.
To prove Proposition \ref{prop_alternative_basis}, we will derive from the third assertion of Lemma \ref{lemma_filtration_s} that the image of  $\mathcal{TB}^{\mathfrak{o}} $ through $\psi$ is a basis of $\mathbf{Gr}_{\bullet}$.

\begin{proof}[Proof of Proposition \ref{prop_alternative_basis}]
As noted previously,  if $\mathbf{\Sigma}$ is closed or if $\mathbf{\Sigma}$ is bigon or a disc with one puncture on its boundary, then $\mathcal{TB}^{\mathfrak{o}}=\mathcal{B}^{\mathfrak{o}}$ so the lemma follows from Theorem \ref{theorem_basis}. Else, we can consider a topological triangulation and consider the associated graded isomorphism $\psi :  \mathcal{S}_{\omega}(\mathbf{\Sigma}) \rightarrow \mathbf{Gr}_{\bullet}$. Let $\mathcal{TB}^{\mathfrak{o}}_i \subset \mathcal{TB}^{\mathfrak{o}}$ be the subset of elements $[T(D), s]$ such that $|D|=i$. Since $\psi(\mathcal{B}^{\mathfrak{o}}_i)$ is a basis of $\mathbf{Gr}_i$, the third assertion of Lemma \ref{lemma_filtration_s} implies that the image $\psi( \mathcal{TB}^{\mathfrak{o}}_i)$ is also a basis of $\mathbf{Gr}_i$. Therefore $\psi (\mathcal{TB}^{\mathfrak{o}})$ is a basis of $\mathbf{Gr}_{\bullet}$, so $\mathcal{TB}^{\mathfrak{o}}$ is a basis of $\mathcal{S}_{\omega}(\mathbf{\Sigma})$.

\end{proof}

 \subsection{Removing a puncture}\label{sec_removingpuncture}
Let $\mathbf{\Sigma}=(\Sigma, \mathcal{P})$ and consider a punctured surface $\mathbf{\Sigma}':=(\Sigma, \mathcal{P}\cup \{p_0\} )$ obtained from $\mathbf{\Sigma}$ by adding a puncture $p_0 \in \Sigma_{\mathcal{P}}$ to $\mathcal{P}$. The goal of this subsection is to define and study a map $\varphi : \mathcal{S}_{\omega}(\mathbf{\Sigma}') \rightarrow \mathcal{S}_{\omega}(\mathbf{\Sigma})$. Let $\mathcal{T}(\mathbf{\Sigma})$ denote the set of stated tangles in $\Sigma_{\mathcal{P}}\times (0,1)$ and denote by $\mathcal{J}(\mathbf{\Sigma}) \subset \mathcal{R}[\mathcal{T}(\mathbf{\Sigma})]$ the ideal generated the skein relations \eqref{eq: skein 1}, \eqref{eq: skein 2} and by the elements $(T,s)-(T',s)$, where $T, T'$ are isotopic; so by definition, one has $\mathcal{S}_{\omega}(\mathbf{\Sigma}) := \quotient{\mathcal{R}[\mathcal{T}(\mathbf{\Sigma})]}{\mathcal{J}(\mathbf{\Sigma})}$.
 The inclusion map $\iota : \Sigma_{\mathcal{P}\cup \{p_0\}} \times (0,1) \hookrightarrow \Sigma_{\mathcal{P}} \times (0,1)$ induces a linear map $\underline{\varphi} : \mathcal{R}[\mathcal{T}(\mathbf{\Sigma}')] \rightarrow \mathcal{R}[\mathcal{T}(\mathbf{\Sigma})]$ sending a stated tangle $(T,s)$ to $(\iota(T), s\circ \iota^{-1})$. 

\vspace{2mm}
\par First suppose that $p_0$ is in the interior of $\Sigma_{\mathcal{P}}$. In this case, $\underline{\varphi}$ obviously sends isotopic stated tangles to isotopic stated tangles and skein relations to skein relations, so it sends $\mathcal{J}(\mathbf{\Sigma}')$ to $\mathcal{J}(\mathbf{\Sigma})$ and it induces a linear map $\varphi : \mathcal{S}_{\omega}(\mathbf{\Sigma}') \rightarrow \mathcal{S}_{\omega}(\mathbf{\Sigma})$ by passing to the quotient. It is clear that $\varphi$ is a morphism of algebras. Moreover,  since any tangle in $\Sigma_{\mathcal{P}}\times (0,1)$ can be isotoped in $\Sigma_{\mathcal{P}\cup \{p_0\}} \times (0,1)$, the map $\varphi$ is surjective.
\vspace{2mm}
\par When $p_0$ lies in some boundary arc, say $a$, of $\mathbf{\Sigma}$, the construction is more subtle. Denote by $b$ and $c$ the two boundary arcs of $\mathbf{\Sigma}'$ which are the connected components of $a\setminus \{p_0\}$. The linear map $\varphi$ still sends skein relations to skein relations, however if $(T,s)$ and $(T',s')$ are two isotopic stated tangles, then $\underline{\varphi}(T,s)$ and $\underline{\varphi}(T',s')$ are no longer necessarily isotopic. Indeed, recall that in our definition of isotopy, for any boundary arc $d$,  the height order of $\partial_d T$ should be preserved. Now if we choose $T$ and $T'$ isotopic in $\Sigma_{\mathcal{P}\cup \{p_0\}}\times (0,1)$, the isotopy relating $T$ to $T'$ preserves the height orders of $\partial_b T$ and $\partial_cT$, but not necessarily the height order of $\partial_a T$, so $\underline{\varphi}(T,s)$ and $\underline{\varphi}(T',s')$ might not be isotopic.

Even worse: $T$ might have two endpoints in $\partial_b T$ and $\partial_c T$ with the same height, so that $\iota(T)$ is not a tangle in our sense since it would have two points in $\partial_a \iota(T)$ with the same height.

 To remedy this problem, we introduce the subset $\mathcal{T}^0(\mathbf{\Sigma}') \subset \mathcal{T}(\mathbf{\Sigma}')$ of stated tangles $(T,s)$ in $\Sigma_{\mathcal{P}\cup \{p_0\} }$ such that for any two points $v \in \partial_b(T)$ and $v' \in \partial_c(T)$, one has $h(v)<h(v')$ ($h$ is the height function). Since any stated tangle $(T,s) \in \mathcal{T}(\mathbf{\Sigma}')$ is isotopic to a stated tangle in $\mathcal{T}^0(\mathbf{\Sigma}')$, one has $\mathcal{S}_{\omega}(\mathbf{\Sigma}') = \quotient{\mathcal{R}[\mathcal{T}^0(\mathbf{\Sigma}')]}{\mathcal{J}(\mathbf{\Sigma}' )\cap \mathcal{R}[\mathcal{T}^0(\mathbf{\Sigma}')]}$. Now, the restriction $\underline{\varphi}^0 : \mathcal{R}[\mathcal{T}^0(\mathbf{\Sigma}')] \rightarrow \mathcal{R}[\mathcal{T}(\mathbf{\Sigma})]$ of $\underline{\varphi}^0$ preserves skein relations and $(T,s)$ is isotopic to $(T',s')$ implies that $\underline{\varphi}^0(T,s) $ is isotopic to $\underline{\varphi}^0(T',s')$; therefore $\varphi^0$ induces a linear map $\varphi : \mathcal{S}_{\omega}(\mathbf{\Sigma}') \rightarrow \mathcal{S}_{\omega}(\mathbf{\Sigma})$ which is obviously an algebra morphism and is surjective.

\begin{definition}\label{def_off_puncture} The \textit{off-puncture ideal} $\mathcal{I}_{p_0} \subset \mathcal{S}_{\omega}(\mathbf{\Sigma}')$ is the ideal generated by 
\begin{enumerate}
\item the elements $\gamma - \gamma'$, where $\gamma$ and $\gamma'$ are non-contractible simple closed curves in $\Sigma_{\mathcal{P}\cup \{p_0\}}$ which are isotopic in $\Sigma_{\mathcal{P}}$;
\item the elements $\alpha_{\varepsilon \varepsilon'} - \alpha'_{\varepsilon \varepsilon'}$, where $\alpha_{\varepsilon \varepsilon'}$ and $\alpha'_{\varepsilon \varepsilon'}$ are non-trivial simple stated arcs in $\Sigma_{\mathcal{P}\cup \{p_0\}}$ which are isotopic in $\Sigma_{\mathcal{P}}$;  
\item when $p_0$ is an inner puncture, the element $\gamma_{p_0} +q +q^{-1}$, where $\gamma_{p_0}$ is a peripheral curve encircling $p_0$ (recall that $q=\omega^{-4}$); 
\item when $p_0$ is on the boundary of $\Sigma_{\mathcal{P}}$, the elements ${\alpha_{p_0}}_{\mu \mu'} - C^{\mu}_{\mu'}$, where ${\alpha_{p_0}}$ is the trivial arc encircling $p_0$ as follows  ${\alpha_{p_0}}_{\mu \mu'} =\adjustbox{valign=c}{\includegraphics[width=1.2cm]{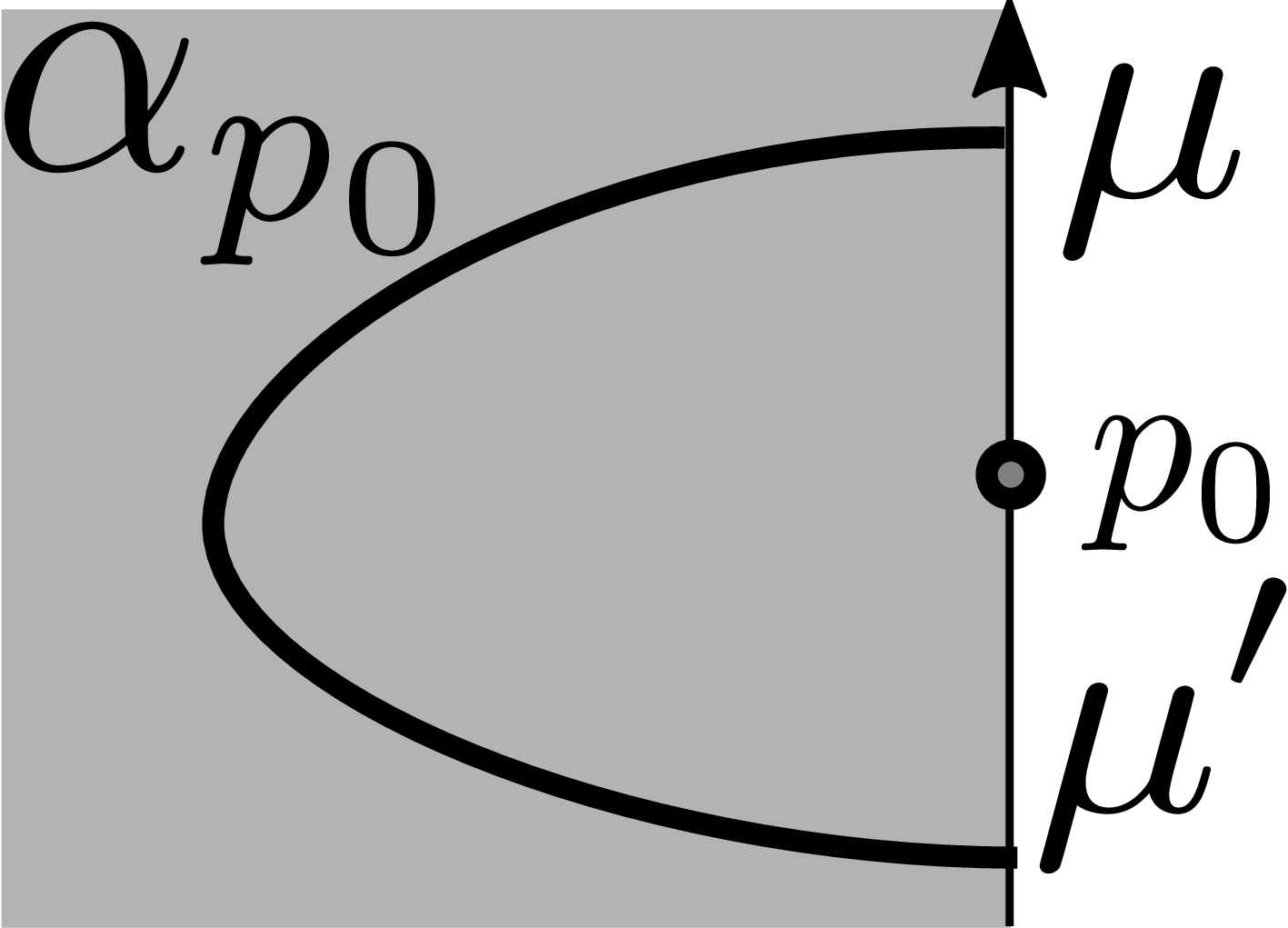}}$,  such that the endpoint with state $\mu$ has bigger height than the endpoint with state $\mu'$.

\end{enumerate}
\end{definition}

The purpose of this subsection it to prove the

\begin{proposition}\label{prop_off_puncture}
The following sequence is exact:
\begin{equation}\label{eq_off_puncture}
 0 \rightarrow \mathcal{I}_{p_0} \rightarrow \mathcal{S}_{\omega}(\mathbf{\Sigma}') \xrightarrow{\varphi} \mathcal{S}_{\omega}(\mathbf{\Sigma}) \rightarrow 0.
\end{equation}
\end{proposition}

The surjectivity of $\varphi$ follows from the preceding discussion and the inclusion $\mathcal{I}_{p_0} \subset \ker (\varphi)$ is an immediate consequence of the definitions and the trivial arcs relations \eqref{trivial_arc_rel} (where the equalities $\varphi({\alpha_{p_0}}_{\mu \mu'} ) = C^{\mu}_{\mu'}$ are proved), so we need to prove the inclusion $\ker(\varphi) \subset \mathcal{I}_{p_0}$.

\begin{notations}
\begin{itemize}
\item
Let $(D,s)$ be a connected simple stated diagram in $\Sigma_{\mathcal{P}\cup \{p_0\}}$ (so either a closed curve or a stated arc or the empty diagram) and define a scalar $c(D,s)\in \mathcal{R}$ as follows. If $\iota(D)$ is simple in $\Sigma_{\mathcal{P}}$, set $c(D,s)=1$. If $p_0$ is an inner puncture and $(D,s)=\gamma_{p_0}$ is a peripheral curve around $p_0$, write $c(\gamma_{p_0})=-q-q^{-1}$. If $p_0$ is on the boundary of $\Sigma_{\mathcal{P}}$ and $\iota(D)$ is a trivial arc encircling $p_0$, let $c(D,s)$ be the unique element $C^{\mu}_{\mu'}$ or $(C^{-1})^{\mu}_{\mu'}$ such that $\varphi(D,s)=c(D,s)$ (using the trivial arcs relations \eqref{trivial_arc_rel}). 
\item For a non-necessarily connected stated diagram $(D,s)= \bigsqcup_{i\in I} (D_i, s_i)$,  where the $(D_i, s_i)$ are its connected components, write $c(D,s)=\prod_{i\in I} c(D_i,s_i)$. Let $J\subset I$ be the subset of indices $j\in I$ such that $\iota(D_j)$ is simple. The \textit{reduction of} $D$ is the simple diagram $D^{red}:= \bigsqcup_{j\in J}D_j$. By definition, one has 
\begin{equation}\label{eq_reduction_sd}
\varphi([T(D),s])= c(D,s) \varphi([(T(D^{red}),s)]).
\end{equation}
 \end{itemize}

\end{notations}

\begin{lemma}\label{lemma_morphism_modules}
Let $M, M'$ be two free $\mathcal{R}$-modules  with respective bases $\mathcal{B}$ and $\mathcal{B}'$. Let $\pi : \mathcal{B}' \rightarrow \mathcal{B}$ and $c: \mathcal{B}' \rightarrow \mathcal{R}$ two maps and suppose that there exists ${\mathcal{B}'}^{red} \subset \mathcal{B}'$ such that the restriction $\pi_{| {\mathcal{B}'}^{red}} {\mathcal{B}'}^{red} \rightarrow \mathcal{B}$ is surjective and such that $c({b'}^{red})=1$ for all ${b'}^{red}\in {\mathcal{B}'}^{red}$. Consider the linear morphism  $\varphi : M' \rightarrow M$ defined by $\varphi(b'):=c(b')\pi(b')$, for $b'\in \mathcal{B}'$. Then 
$$\ker(\varphi) = \mathrm{Span} \left\{ b'- c(b'){b'}^{red},\mbox{ such that }\pi({b'}^{red})=\pi(b'), {b'}^{red}\in {\mathcal{B}'}^{red}, b'\in \mathcal{B}' \right\}.$$

\end{lemma}

\begin{proof}
Let $V\subset M'$ be the submodule linearly spanned by the elements $b'- c(b')b'^{red}$ with $\pi({b'}^{red})=\pi(b')$ and ${b'}^{red}\in {\mathcal{B}'}^{red}, b'\in \mathcal{B}'$.
By definition, $\varphi( b' - c(b') {b'}^{red}) = c(b') (\pi(b') - \pi({b'}^{red}))=0$ so the inclusion $V \subset \ker(\varphi)$ is obvious. Conversely, consider an arbitrary element $x= \sum_{b' \in \mathcal{B}'} \alpha_{b'} b' \in \ker(\varphi)$. Fix a right inverse $\iota :
\mathcal{B} \to {\mathcal{B}'}^{red}$ to $\pi$, that is a map such that $\pi \circ \iota=\id$. 
For $b\in \mathcal{B}$, write $x_b:= \sum_{b' \in \pi^{-1}(b)} \alpha_{b'} b'$ so that $x= \sum_{b\in \mathcal{B}} x_b$. Since $\mathcal{B}$ is a basis, the elements $\varphi(x_b)$ are linearly independent so $\varphi(x)=0$ implies that $\varphi(x_b)=0$ for all $b\in \mathcal{B}$. Let $b\in \mathcal{B}$ be such that $x_b\neq 0$ and let us prove that $x_b\in V$. Let ${b'}^{red}:= \iota(b) \in {\mathcal{B}'}^{red}$. Since $\varphi(x_b)=0$, one has $ \sum_{b' \in \pi^{-1}(b)} \alpha_{b'} c(b')=0$. Now
\begin{align*}
& x_b = \sum_{b' \in \pi^{-1}(b)} \alpha_{b'}b' = \sum_{b' \in \pi^{-1}(b)} \alpha_{b'}(b'-c(b'){b'}^{red}) + \left( \sum_{b' \in \pi^{-1}(b)}\alpha_{b'}c(b')\right){b'}^{red}
\\ &=  \sum_{b' \in \pi^{-1}(b)} \alpha_{b'}(b'-c(b'){b'}^{red}) \in V.\end{align*}

This concludes the proof.
\end{proof}

\begin{proof}[Proof of Proposition \ref{prop_off_puncture}]
Applying Lemma \ref{lemma_morphism_modules} to $M=\mathcal{S}_{\omega}(\mathbf{\Sigma})$, $M'=\mathcal{S}_{\omega}(\mathbf{\Sigma}')$, $\mathcal{B}=\mathcal{TB}^{\mathfrak{o}}(\mathbf{\Sigma})$, $\mathcal{B}'=\mathcal{TB}^{\mathfrak{o}}(\mathbf{\Sigma}')$, ${\mathcal{B}'}^{red}$ the subset of $\mathcal{B}'$ of diagrams $(T(D), s)$ such that $D^{red}=D$ and $\pi$ the reduction map, we obtain that $\ker(\varphi)$ is spanned by elements of the form $[T(D), s] - c(D,s)[T(D^{red}),s]$. By definition, the off-puncture ideal is the ideal generated by the elements $[T(D), s] - c(D,s)[T(D^{red}),s]$, where $D$ is connected. Let us prove by induction on the number of connected components of $D$ that $[T(D), s] - c(D,s)[T(D^{red}),s]\in \mathcal{I}_{p_0}$. If $D$ is connected or reduced, this is immediate. Else, $(D,s)$ contains a connected component $(D_0,s_0)$ so that $\iota(D_0)$ is either contractible or a trivial arc. Decompose $(D,s)= (D_1,s_1) \bigsqcup (D_0,s_0) \bigsqcup (D_2,s_2)$ so that for any connected component $C_1\subset D_1$, one has $C_1\preceq D_0$ and for any connected component $C_2\subset D_2$ one has $D_0 \preceq C_2$ (recall that $\preceq$ was defined in Section \ref{sec_basecheloud}). By definition, one has $[T(D),s]= [T(D_2), s_2] [T(D_0),s_0] [T(D_1), s_1]$ in $\mathcal{S}_{\omega}(\mathbf{\Sigma}')$ (this is where working with the basis $\mathcal{TB}^{\mathfrak{o}}$ is important), where $s_i$ are the restriction of $s$ to $D_i$. Therefore

\begin{align*}
 &[T(D), s] - c(D,s)[T(D^{red}), s] =  [T(D_2), s_2] \left( [T(D_0),s_0] - c(D_0,s_0) \right) [T(D_1), s_1] \\
 & + c(D_0,s_0) \left( [T(D_2 \cup D_1), s_2\cup s_1] - c(D_2\cup D_1, s_2\cup s_1) [T((D_2\cup D_1)^{red}), s]\right) \\
 & \equiv c ( [T(D', s')]- c(D',s')[T({D}^{\prime red}), s]) \pmod{\mathcal{I}_{p_0}}.
 \end{align*}

Where $c=c(D_0,s_0)$ and  $D'=D_2 \cup D_1$ has one connected component less than $D$, so we can apply the induction hypothesis to prove that $ [T(D), s] - c(D,s)[T(D^{red}), s]  \in \mathcal{I}_{p_0}$. This concludes the proof.

\end{proof}

 \subsection{Hopf comodule maps}

 \par Recall that the bigon $\mathbb{B}$ is a disc with two punctures on its boundary. It has two boundary arcs, say $b_L$ and $b_R$. Consider the simple diagram $\alpha$ made of a single arc joining $b_L$ and $b_R$. For $n\geq 0$, denote by $\alpha^{(n)}$ the diagram made of $n$ parallel copies of $\alpha$. Denote by $\alpha_{\varepsilon \varepsilon'}$ the class in $\mathcal{S}_{\omega}(\mathbb{B})$ of the stated diagram $(\alpha, s)$ where $s(\alpha\cap b_L)=\varepsilon$ and $s(\alpha\cap b_R)=\varepsilon'$. It is proved in \cite[Theorem $4.1$]{LeStatedSkein} that the  stated skein algebra $\mathcal{S}_{\omega}(\mathbb{B})$ is  presented by the four generators $\alpha_{\varepsilon \varepsilon'}$, with $\varepsilon, \varepsilon' = \pm$, and the following relations, where we put $q:= \omega^{-4}$:
\begin{align*}\label{relbigone}
\alpha_{++}\alpha_{+-} &= q^{-1}\alpha_{+-}\alpha_{++} & \alpha_{++}\alpha_{-+}&=q^{-1}\alpha_{-+}\alpha_{++}
\\ \alpha_{--}\alpha_{+-} &= q\alpha_{+-}\alpha_{--} & \alpha_{--}\alpha_{-+}&=q\alpha_{-+}\alpha_{--}
\\ \alpha_{++}\alpha_{--}&=1+q^{-1}\alpha_{+-}\alpha_{-+} &  \alpha_{--}\alpha_{++}&=1 + q\alpha_{+-}\alpha_{-+} 
\\ \alpha_{-+}\alpha_{+-}&=\alpha_{+-}\alpha_{-+} & &
\end{align*}
\par Consider a disjoint union $\mathbb{B}\bigsqcup \mathbb{B}'$ of two bigons. When gluing the boundary arcs $b_R$ with $b'_L$, we obtain another bigon. Denote by $\Delta : \mathcal{S}_{\omega}(\mathbb{B})\rightarrow \mathcal{S}_{\omega}(\mathbb{B}) \otimes \mathcal{S}_{\omega}(\mathbb{B})$ the composition: 
$$\Delta: \mathcal{S}_{\omega}(\mathbb{B}) \xrightarrow{i_{|b_R\#b_L'}} \mathcal{S}_{\omega}(\mathbb{B}\bigsqcup \mathbb{B}') \xrightarrow{\cong} \mathcal{S}_{\omega}(\mathbb{B}) \otimes \mathcal{S}_{\omega}(\mathbb{B}).$$
\par The map $\Delta$ is characterized by the formula $\Delta( \alpha_{\varepsilon \varepsilon'})= (\alpha_{\varepsilon +}\otimes \alpha_{+ \varepsilon'}) + (\alpha_{\varepsilon -} \otimes \alpha_{- \varepsilon'})$. Define an algebra morphism $\epsilon: \mathcal{S}_{\omega}(\mathbb{B}) \rightarrow \mathcal{R}$ and an anti-algebra morphism (that is $S$ is linear and $S(xy)=S(y)S(x)$) $S : \mathcal{S}_{\omega}(\mathbb{B}) \rightarrow \mathcal{S}_{\omega}(\mathbb{B})$ by the formulas $ \epsilon(\alpha_{\varepsilon \varepsilon'})= \delta_{\varepsilon \varepsilon'}$, $S(\alpha_{++}) = \alpha_{--}, S(\alpha_{--})=\alpha_{++}, S(\alpha_{+-})=-q \alpha_{+-}$ and $S(\alpha_{-+})=-q^{-1}\alpha_{-+}$. The coproduct  $\Delta$, the counit $\epsilon$ and the antipode $S$ endow $\mathcal{S}_{\omega}(\mathbb{B})$ with a structure of Hopf algebra. This Hopf algebra is canonically isomorphic to the so-called \textit{quantum }$\SL_2$ Hopf algebra $\mathcal{O}_q[\SL_2]$ 
 as defined in (\cite{Manin_QGroups}, \cite[Chapter $IV$ Section $6$]{Kassel}, \cite[Definition $7.1.1$]{ChariPressley}, \cite[Definition $I.1.10$]{BrownGoodearl}) where the generators $\alpha_{++}, \alpha_{-+}, \alpha_{+-}$ and $\alpha_{--}$ are denoted by $a,b,c$ and $d$. 
\\ For later use, let us write the coproduct, counit and antipode by the following more compact form: 
 
 \begin{equation*}
 \begin{pmatrix} \Delta (\alpha_{++}) & \Delta (\alpha_{+-}) \\ \Delta(\alpha_{-+}) & \Delta(\alpha_{--}) \end{pmatrix} 
 = 
 \begin{pmatrix} \alpha_{++} & \alpha_{+-} \\ \alpha_{-+} & \alpha_{--} \end{pmatrix} 
 \otimes 
 \begin{pmatrix} \alpha_{++} & \alpha_{+-} \\ \alpha_{-+} & \alpha_{--} \end{pmatrix} 
 \end{equation*}
 \begin{equation*}
 \begin{pmatrix} \epsilon(\alpha_{++}) & \epsilon(\alpha_{+-}) \\ \epsilon(\alpha_{-+}) & \epsilon(\alpha_{--}) \end{pmatrix} =
\begin{pmatrix} 1 &0 \\ 0& 1 \end{pmatrix}  
\text{ and }
\begin{pmatrix} S(\alpha_{++}) & S(\alpha_{+-}) \\ 	S(\alpha_{-+}) & S(\alpha_{--}) \end{pmatrix} 
	= 
	\begin{pmatrix} \alpha_{--} & -q\alpha_{+-} \\ -q^{-1}\alpha_{-+} & \alpha_{++} \end{pmatrix} .
 \end{equation*}
 Remark that when $q=+1$, we recover the Hopf algebra of regular functions of $\SL_2(\mathbb{C})$. 
 
 \vspace{2mm}
 \par Consider a punctured surface $\mathbf{\Sigma}$ with boundary arc $a$. When gluing the boundary $a$ of $\mathbf{\Sigma}$ with the boundary arc $b_L$ of $\mathbb{B}$ we obtain the same punctured surface $\mathbf{\Sigma}$. Define a left Hopf comodule map (see \textit{e.g.} \cite[Definition $III.7.1$]{Kassel}) $\Delta_a^L : \mathcal{S}_{\omega}(\mathbf{\Sigma})\rightarrow \mathcal{S}_{\omega}(\mathbb{B}) \otimes \mathcal{S}_{\omega}(\mathbf{\Sigma})$ as the composition:
 $$ \Delta_a^L : \mathcal{S}_{\omega}(\mathbf{\Sigma}) \xrightarrow{i_{|a\# b_L }} \mathcal{S}_{\omega}(\mathbb{B} \bigsqcup \mathbf{\Sigma}) \xrightarrow{\cong} \mathcal{S}_{\omega}(\mathbb{B}) \otimes \mathcal{S}_{\omega}(\mathbf{\Sigma}). $$
 \par Similarly, define a right Hopf comodule map $\Delta_a^R : \mathcal{S}_{\omega}(\mathbf{\Sigma})\rightarrow \mathcal{S}_{\omega}(\mathbf{\Sigma})\otimes  \mathcal{S}_{\omega}(\mathbb{B}) $ as the composition:
 $$ \Delta_a^R : \mathcal{S}_{\omega}(\mathbf{\Sigma}) \xrightarrow{i_{|b_R\# a }} \mathcal{S}_{\omega}(\mathbf{\Sigma} \bigsqcup \mathbb{B}) \xrightarrow{\cong} \mathcal{S}_{\omega}(\mathbf{\Sigma}) \otimes \mathcal{S}_{\omega}(\mathbb{B}). $$
\par The coassociativity of  $\Delta_a^L$ and $\Delta_a^R$  follows from the coassociativity of the splitting morphisms.

 Figure \ref{fig_comodule} illustrates the coproduct and the (left) comodule map.

\begin{figure}[!h] 
\centerline{\includegraphics[width=6cm]{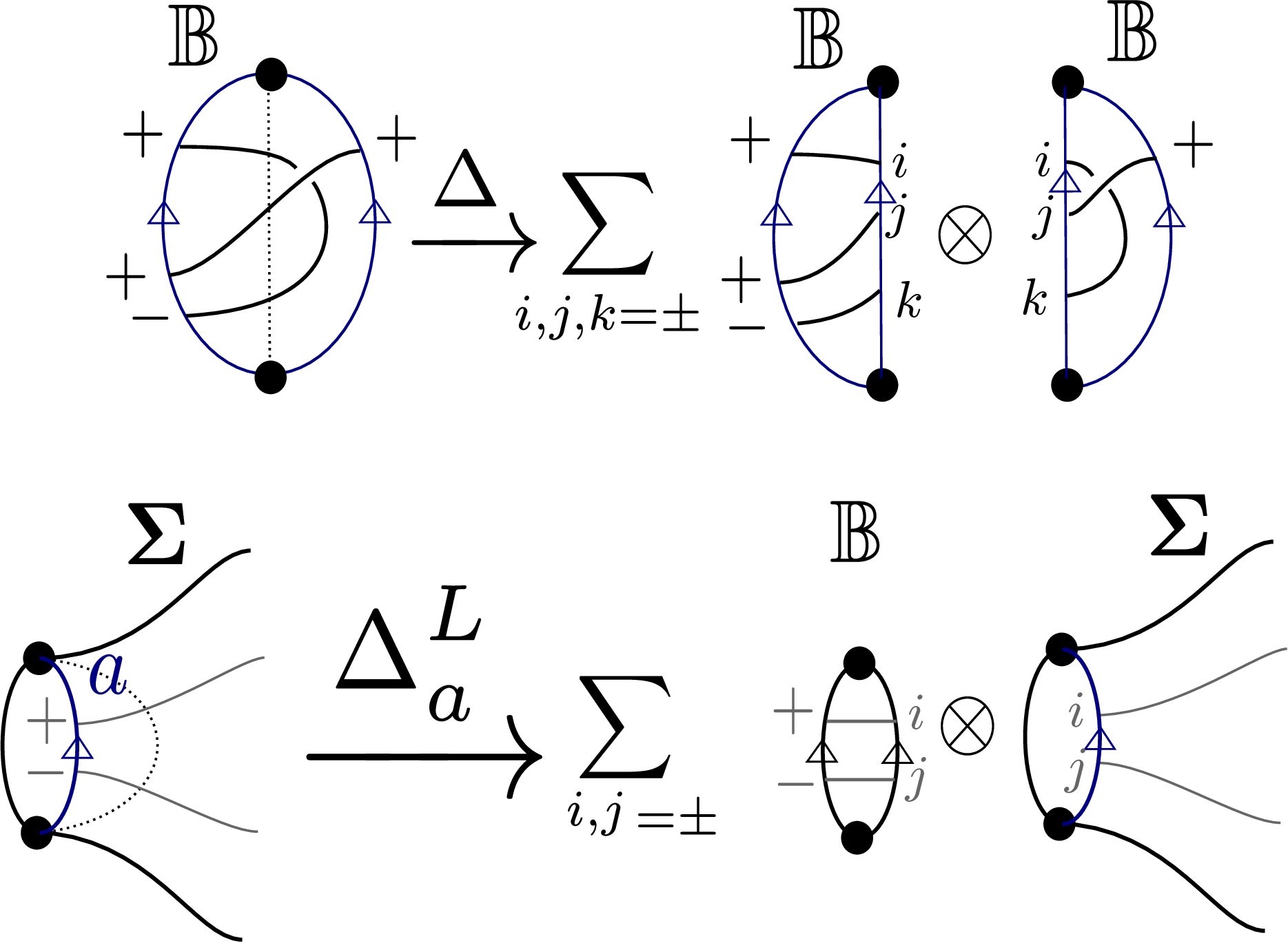} }
\caption{On the top: the coproduct in $\mathcal{S}_{\omega}(\mathbb{B})$. On the bottom: the comodule map. } 
\label{fig_comodule} 
\end{figure}

 \subsection{The image of the splitting morphism}
 
 The goal of this subsection is to prove Theorem \ref{theorem1} that we rewrite here for convenience of the reader:
 
 \begin{theorem}\label{theorem1prime} Let $\mathbf{\Sigma}$ be a punctured surface, $a,b$ two distinct boundary arcs. Then the following sequence is exact:
 $$ 0 \rightarrow \mathcal{S}_{\omega}(\mathbf{\Sigma}_{|a\#b}) \xrightarrow{i_{|a\#b}} \mathcal{S}_{\omega}(\mathbf{\Sigma}) \xrightarrow{\Delta_a^L - \sigma \circ \Delta_b^R} \mathcal{S}_{\omega}(\mathbb{B}) \otimes \mathcal{S}_{\omega}(\mathbf{\Sigma}),$$
	where $\sigma(x\otimes y)= y\otimes x$.
 \end{theorem}
 
 All along this subsection, we fix an orientation $\mathfrak{o}$ of its boundary arcs (though  Theorem \ref{theorem1prime} is obviously independent on this choice). 
\begin{notations} For  a boundary arc $a$ and a diagram $D$, we write $n_a(D):= |\partial_a D|$. Given $n\geq 1$, define the set $\St(n):= \{ -, + \}^n$ and the subset $\St^{\uparrow}(n)\subset \St(n)$ which consists of $n$-tuples $(\varepsilon_1, \ldots, \varepsilon_n)$ such that $i<j$ implies $\varepsilon_i \leq \varepsilon_j$. 
	 If $s=(\varepsilon_1, \ldots, \varepsilon_n)\in \St(n)$, denote by $s^{\uparrow}=(\varepsilon'_1, \ldots, \varepsilon'_n)\in \St^{\uparrow}(n)$ the unique element such that the number of indices $i$ such that $\varepsilon_i=+$ is equal to the number of indices $j$ such that $\varepsilon'_j=+$. 
Given $s=(\varepsilon_1, \ldots, \varepsilon_n)\in \St(n)$, denote by $k(s)$ the number of pairs $(i,j)$ such that $i<j$ and $\varepsilon_i > \varepsilon_j$. 
	For $s\in \St^{\uparrow}(n)$, let 
 $$ H_s(q) := \sum_{s'\in \St(n)| s'^{\uparrow}=s} q^{2 k(s')}.$$
 \end{notations}

 \par Let  $a$ and $b$ be two boundary arcs of $\mathbf{\Sigma}$ and consider the filtration associated to $S:=\{a, b\}$ and $\mathfrak{o}$ of Definition \ref{def_filtration}.
 
 \begin{lemma}\label{lemma_leading_term}
 Let $(D,s)$ be an $\mathfrak{o}$-oriented simple stated diagram and consider $v_1, v_2$ two points which both belong either to $\partial_a D$ or to $\partial_b D$. Suppose that $v_1<_{\mathfrak{o}} v_2$ and that there is no $v\in \partial D$ such that $v_1 <_{\mathfrak{o}} v <_{\mathfrak{o}} v_2$. Further assume that $s(v_1)=+$ and $s(v_2)=-$. Let $s'$ be the state of $D$ such that $s'(v_1)=-$, $s'(v_2)=+$ and $s'(v)=s(v)$ if $v\in \partial D \setminus \{v_1, v_2\}$. 
 Then one has $\lt ([D,s])= q \lt ([D, s'])$, where the leading term $\lt$ is defined in Definition \ref{def_filtration}.
 \end{lemma}
 
 \begin{proof} This is a straightforward consequence of the  boundary relations \eqref{eq: skein 2} and the height exchange formulas \eqref{eq: height exch 1} and \eqref{eq: height exch corr}.   \end{proof}
 
 \par Let $(D,s)$ be  an $\mathfrak{o}$-oriented simple stated diagram of $\mathbf{\Sigma}$ and write $s=(s_a, s_0, s_b)$ as in the definition of the gluing map before Theorem \ref{theorem_gluingmap}. It results from Lemma \ref{lemma_leading_term} that we have the equality:
 $$ \lt ([D,(s_a,s_0,s_b)]) = q^{k(s_a) +k(s_b)} \lt([D, (s_a^{\uparrow}, s_0, s_b^{\uparrow})]).$$
 
 \par Fix an orientation $\mathfrak{o}_{\mathbb{B}}$ of the left and right boundary arcs of the bigon. Consider the filtration of  $\mathcal{S}_{\omega}(\mathbb{B})\otimes \mathcal{S}_{\omega}(\mathbf{\Sigma})\cong \mathcal{S}_{\omega}(\mathbb{B} \bigsqcup \mathbf{\Sigma})$ associated to the set of boundary arcs $S':=\{b_L, b_R, a,b\}$ and the orientations $\mathfrak{o}$ and $\mathfrak{o}_{\mathbb{B}}$, as in Definition \ref{def_filtration}. Given $X' \in \mathcal{S}_{\omega}(\mathbb{B})\otimes \mathcal{S}_{\omega}(\mathbf{\Sigma})$, we denote by $\lt'(X')$ the associated leading term. By definition of the left comodule map, we have the formula:
 $$ \Delta_a^L ([D, (s_a,s_0,s_b)]) = \sum_{s\in \St(n_a(D))} [\alpha^{(n_a(D))}, (s_a,s)] \otimes [D, (s, s_0,s_b)] $$
 
 \begin{lemma}\label{lemmA}
  Let $[D, (s_a,s_0,s_b)]$ be an element of the basis $\mathcal{B}^{\mathfrak{o}}$. One has 
 \begin{equation*} 
	\lt'\left(\Delta_a^L ([D, (s_a, s_0,s_b)])\right)=\sum_{s\in \St^{\uparrow}(n_a(D))} H_s(q) [\alpha^{(|\partial_a(D)|)}, (s_a,s)] \otimes [D, (s,s_0,s_b)]
	\end{equation*}
	and 
	\begin{equation*} 
	\lt'\left(\sigma \circ \Delta_b^R([D, (s_a,s_0,s_b)]) \right)= \sum_{s\in \St^{\uparrow}(n_b(D)} H_{s}(q) [\alpha^{(|\partial_b(D)|)}, (s,s_b)] \otimes [D, (s_a, s_0,s)], 
\end{equation*}
where the summands are written in the basis associated to $(\mathfrak{o}, \mathfrak{o}_{\mathbb{B}})$ of $\mathcal{S}_{\omega}(\mathbb{B})\otimes \mathcal{S}_{\omega}(\mathbf{\Sigma})$. 
 \end{lemma}

 \begin{proof} This is a straightforward consequence of Lemma \ref{lemma_leading_term}. \end{proof}

 \begin{proof}[Proof of Theorems \ref{theorem1}, \ref{theorem1prime}]
 We want to show that the following sequence is exact:
	$$ 0 \rightarrow \mathcal{S}_{\omega}(\mathbf{\Sigma}_{|a\#b}) \xrightarrow{i_{|a\#b}} \mathcal{S}_{\omega}(\mathbf{\Sigma}) \xrightarrow{\Delta_a^L - \sigma \circ \Delta_b^R} \mathcal{S}_{\omega}(\mathbb{B}) \otimes \mathcal{S}_{\omega}(\mathbf{\Sigma}),$$
	where $\sigma(x\otimes y)= y\otimes x$.
 The injectivity of $i_{|a\# b}$ was proved in \cite{LeStatedSkein}.
 The inclusion $\mathrm{Im}(i_{|a\#b}) \subset \ker (\Delta_a^L - \sigma \circ \Delta_b^R)$ follows from the coassociativity of the comodule maps. To prove the reverse inclusion, consider an element $X:= \sum_{i\in I} x_i [D_i,s_i] \in \ker (\Delta_a^L - \sigma \circ \Delta_b^R)$ developed in the basis $\mathcal{B}^{\mathfrak{o}}$. If $\lt(X)=0$, then $X$ is a linear combination of diagrams which do not intersect $a$ and $b$, hence $X$ belongs to the image of $i_{|a\#b}$. Suppose that $\lt(X)>0$. We will find an element $Y\in \mathcal{S}_{\omega}(\mathbf{\Sigma}_{|a\#b})$ such that $\lt(i_{|a\#b}(Y)) = \lt(X)$. Now $X$ belongs to the image of $i_{|a\#b}$ if and only if $Z:= X- i_{|a\#b}(Y)$ belongs to this image. Since $d(Z)<d(X)$, the proof will follow by induction on $d(X)$.
 \vspace{2mm}
 \par Consider the set $\widetilde{\mathcal{D}}$ of pairs $(D,s_0)$ for which  there exists some states $s_a$ and $s_b$ such that the basis element $[D,(s_a,s_0,s_b)]$ appears in the expression of $X$. Given $\widetilde{D}=(D,s_0)\in \widetilde{\mathcal{D}}$, denote by $\St_X(\widetilde{D})$ the set of couples $(s_a,s_b)$ such that $[D, (s_a,s_0,s_b)]$ appears in the expression of $X$. We re-write the development of $X$ in the basis as:
 $$ X= \sum_{\widetilde{D}=(D,s_0)\in \widetilde{\mathcal{D}}} \sum_{(s_a,s_b)\in \St_X(\widetilde{D})} x_{[D,(s_a,s_0,s_b)]} [D, (s_a,s_0,s_b)]. $$
  Consider the subset $\widetilde{\mathcal{D}}_{\mathrm{max}} \subset \widetilde{\mathcal{D}}$ of pairs $(D,s_0)$ such that $d(X)=n_a(D)+n_b(D)$. By Lemma \ref{lemmA}, one has:
  \begin{multline*}
   \lt'(\Delta_a^L(X)) = \sum_{(D,s_0)\in \widetilde{\mathcal{D}}_{\mathrm{max}}} \sum_{(s_a,s_b)\in \St_X((D,s_0))} x_{[D,(s_a,s_0,s_b)]} 
   \\ \sum_{s\in \St^{\uparrow}(n_a(D))} H_s(q) [\alpha^{(n_a(D))}, (s_a,s)] \otimes [D, (s,s_0,s_b)].
   \end{multline*}
  Similarly, one has:
  \begin{multline*}
   \lt'(\sigma\circ\Delta_b^R(X)) = \sum_{(D,s_0)\in \widetilde{\mathcal{D}}_{\mathrm{max}}} \sum_{(s_a,s_b)\in \St_X((D,s_0))} x_{[D,(s_a,s_0,s_b)]} 
   \\ \sum_{s'\in \St^{\uparrow}(n_b(D))} H_{s'}(q) [\alpha^{(n_b(D))}, (s',s_b)] \otimes [D, (s_a,s_0,s')].
   \end{multline*}
 \par From the equality $\lt'(\Delta_a^L(X)) = \lt'(\sigma\circ \Delta_b^R(X))$, we find that for any pair $(D,s_0)\in \widetilde{\mathfrak{D}}_{\mathrm{max}}$,  for any pair $(s_a,s_b) \in \St_X((D,s_0))$ and for any state $s\in \St^{\uparrow}(n_a(D))$, there exists a unique pair $(s_a',s_b')\in \St_X((D,s_0))$ and a unique state $s'\in \St^{\uparrow}(n_b(D))$ such that:
 \begin{multline*}
  x_{[D,(s_a,s_0,s_b)]} H_s(q) [\alpha^{(n_a(D))}, (s_a,s)] \otimes [D, (s,s_0,s_b)] \\
  = x_{[D, (s_a',s_0,s'_b)]} H_{s'}(q) [\alpha^{(n_b(D))}, (s',s'_b)] \otimes [D, (s'_a,s_0,s')]. 
  \end{multline*}
 \par We deduce the following:
 \begin{itemize}
 \item For any $(D,s_0)\in \widetilde{\mathcal{D}}_{\mathrm{max}}$, we have $n_a(D)=n_b(D)= \frac{1}{2}d(X)$. We will denote by $n$ this integer.
 \item We have the equalities $s'=s_a=s_b$ and $s=s'_a=s'_b$. Hence for any $(D,s_0)\in \widetilde{\mathcal{D}}_{\mathrm{max}}$, we have $\St_X((D,s_0))= \{ (s,s), s\in \St^{\uparrow}(n) \}$.
 \item  For any $(D,s_0)\in \widetilde{\mathcal{D}}_{\mathrm{max}}$ and $s \in \St^{\uparrow}(n)$, the coefficient 
 $x_{[D,(s,s_0,s)]}H_s(q)$ is independent of $s$. We will denote this coefficient by $x_{(D,s_0)}$.
 \end{itemize}
 \par With the above notations, we re-write the leading term of $X$ as:
 $$ \lt(X) = \sum_{(D,s_0)\in \widetilde{\mathcal{D}}_{\mathrm{max}}} x_{(D,s_0)} \sum_{s\in \St^{\uparrow}(n)}[D,(s,s_0,s)].$$
 \par Given $(D,s_0)\in \widetilde{\mathcal{D}}_{\mathrm{max}}$, since $n_a(D)=n_b(D)=n$, there exists a diagram $D_0$ of $\mathbf{\Sigma}_{|a\#b}$ such that $D$ is obtained from $D_0$ by cutting along the common image in $\Sigma_{|a\#b}$ of $a$ and $b$ by the projection. Define the following element:
 $$ Y:= \sum _{(D,s_0)\in \widetilde{\mathcal{D}}_{\mathrm{max}}} x_{(D,s_0)} [D_0, s_0] \in \mathcal{S}_{\omega}(\mathbf{\Sigma}).$$
 \par By the above expression, we have the equality $\lt(X) = \lt(i_{|a\#b}(Y))$. This concludes the proof.
  \end{proof}
 
 \par Consider a topological triangulation $\Delta$ of $\mathbf{\Sigma}$. The punctured surface $\mathbf{\Sigma}$ is obtained from the disjoint union $\mathbf{\Sigma}_{\Delta}:=\bigsqcup_{\mathbb{T} \in F(\Delta)} \mathbb{T}$ by gluing the triangles along their common edges. Denote by $\mathring{\mathcal{E}}(\Delta) \subset \mathcal{E}(\Delta)$ the subset of edges which are not boundary arcs. Each edge $e\in \mathring{\mathcal{E}}(\Delta)$ lifts in $\mathbf{\Sigma}_{\Delta}$ to two boundary arcs $e_L$ and $e_R$. 
 By composing all the left comodule maps $\Delta_{e_L}^L$ together (the order does not matter thanks to the coassociativity property in Theorem \ref{theorem_gluingmap}) one gets a Hopf comodule map
 \begin{equation*}
 \Delta^L : \otimes_{\mathbb{T}\in F(\Delta)} \mathcal{S}_{\omega}(\mathbb{T}) \rightarrow \left( \otimes_{e\in \mathring{\mathcal{E}}(\Delta)} \mathcal{S}_{\omega}(\mathbb{B})\right) \otimes \left(  \otimes_{\mathbb{T}\in F(\Delta)} \mathcal{S}_{\omega}(\mathbb{T}) \right).
 \end{equation*}
 Similarly, composing all the right comodule maps $\Delta_{e_R}^R$ together gives 
 \begin{equation*}
 \Delta^R :  \otimes_{\mathbb{T}\in F(\Delta)} \mathcal{S}_{\omega}(\mathbb{T}) \rightarrow \left(  \otimes_{\mathbb{T}\in F(\Delta)} \mathcal{S}_{\omega}(\mathbb{T}) \right)\otimes \left( \otimes_{e\in \mathring{\mathcal{E}}(\Delta)} \mathcal{S}_{\omega}(\mathbb{B})\right). 
 \end{equation*} 
 Recall the definition of $i^{\Delta}$ in Section \ref{sec_definitions}.
 \begin{corollary}
 	The following sequence is exact. 
 \begin{equation*}
 0 \rightarrow \mathcal{S}_{\omega}(\mathbf{\Sigma}) \xrightarrow{i^{\Delta}} 
 \otimes_{\mathbb{T}\in F(\Delta)} \mathcal{S}_{\omega}(\mathbb{T}) \xrightarrow{\Delta^L -\sigma\circ\Delta^R}  \left( \otimes_{e\in \mathring{\mathcal{E}}(\Delta)} \mathcal{S}_{\omega}(\mathbb{B})\right) \otimes \left(  \otimes_{\mathbb{T}\in F(\Delta)} \mathcal{S}_{\omega}(\mathbb{T}) \right).
 \end{equation*} 
  \end{corollary}                         
\begin{proof}
 Theorem \ref{theorem1} applied to each inner edge provides  an isomorphism between $\mathcal{S}_{\omega}(\mathbf{\Sigma})$ and the intersection, over the inner edges $e$, of $\text{Ker} ( \Delta_{e_L}^L-\sigma\circ \Delta_{e_R}^R)$. We conclude by observing that the latter intersection is $\text{Ker} ( \Delta^L-\sigma\circ \Delta^R)$.
\end{proof}

 \par We can reformulate the above exact sequence in terms of coHochschild cohomology. 
  \begin{definition}\label{def_coHochschild} Given a coalgebra $C$ with a bi-comodule $M$, with comodules maps $\Delta^L :  M \rightarrow C\otimes M$ and $\Delta^R: M \rightarrow M \otimes C$, the $0$-th coHochschild cohomology group is $\mathrm{coHH}^0(C, M):= \ker \left(\Delta^L - \sigma \circ \Delta^R\right)$.
  \end{definition}
We refer to \cite{HPS_CohochschildHom} for a self-contained introduction to coHochschild (co)homology. The above triangular decomposition of skein algebra can be re-written as:
\begin{equation*}
 \mathcal{S}_{\omega}(\mathbf{\Sigma})\cong \mathrm{coHH}^0\left( \otimes_{e\in \mathring{\mathcal{E}}(\Delta)} \mathcal{O}_q[\SL_2], \otimes_{\mathbb{T}\in F(\Delta)}  \mathcal{S}_{\omega}(\mathbb{T})\right).
\end{equation*}

 \subsection{The center of stated skein algebras at odd roots of unity}
 
Here we prove Theorem \ref{theorem2}. We prove it for the bigon, then for the triangle, and we conclude with the general case. Let us start by the following classical result.  

 \begin{lemma}\label{lemma_qbinomial}
 Let $\mathcal{R}$ be a ring and $q\in \mathcal{R}^{\times}$ a root of unity of order $N>1$. Suppose that $\mathcal{A}$ is an $\mathcal{R}$-algebra and $x,y\in \mathcal{A}$ are such that $yx=qxy$. One has $(x+y)^N= x^N +y^N$.
 \end{lemma}
 
 \begin{proof} By \cite[Proposition $IV.2.2$]{Kassel}, one has:
 $$ (x+y)^N = \sum_{k=0}^N \begin{pmatrix} N \\ k \end{pmatrix}_q x^k y^{N-k}, $$
 where $\begin{pmatrix} N \\ k \end{pmatrix}_q := \prod_{i=0}^{k-1} \left( \frac{1-q^{N-i}}{1-q^{i+1}} \right)$. Since $q^N=1$, the coefficients $\begin{pmatrix} N \\ k \end{pmatrix}_q$ vanish for $1\leq k \leq N-1$, and we get the desired formula.
 \end{proof}
 
 \subsubsection{The case of the bigon}
 
 Recall from Section $2.2$ that the Hopf algebra $\mathcal{S}_{\omega}(\mathbb{B})$ is canonically isomorphic to $\mathcal{O}_q[\SL_2]$. In this case, Theorem \ref{theorem2} is a well known theorem of Lusztig. More precisely, it is proved in \cite{Lusztig_QGroupsRoots1} (see also \cite[Theorem $3.5.1$]{Lusztig_Book}) that there exists a morphism of braided Hopf algebras $Fr_*: \dot{U}_q \mathfrak{sl}_2 \to \dot{U}_{+1}\mathfrak{sl}_2$ which induces a braided functor $Fr: \Rep(\SL_2)\to \Rep_q(\SL_2)$ between the category of finite rank representations of $\SL_2$ and the category $\Rep_q(\SL_2)$ of finite rank $\dot{U}_q\mathfrak{sl}_2$ modules. Since $\mathcal{O}_q[\SL_2]$ (resp. $\mathcal{O}[\SL_2]$) is isomorphic to the coend of the forgetful functor $F: \Rep_q(\SL_2)\to \Mod_{\mathcal{R}}$ (resp. of the forgetful functor $\Rep(\SL_2) \to \Mod_{\mathcal{R}}$) the Frobenius functor $Fr$ induces a morphism $j: \mathcal{O}[\SL_2]\to \mathcal{O}_q[\SL_2]$.  Moreover, as noticed in \cite{Negron_Frobenius}, the image of $Fr$ lies in the M\"ugen center of $\Rep_q(\SL_2)$ so the image of $j$ is central.
 We refer to \cite[Section $5.1$]{Negron_Frobenius} for details on this approach. A down-to-earth construction of $j$, based on elementary computations using the definition of $\mathcal{O}_q[\SL_2]$ by generators and relations, was described by Brown-Goodearl and goes as follows:

 \begin{lemma}[\cite{BrownGoodearl} Proposition III.3.1] \label{lemma_center_bigone}
 Suppose that $q:=\omega^{-4}$ is a root of unity of odd order $N\geq 1$. There exists a injective morphism of Hopf algebras  $j_{\mathbb{B}}:\mathcal{S}_{+1}(\mathbb{B}) \rightarrow \mathcal{S}_{\omega}(\mathbb{B})$  characterized by $j_{\mathbb{B}}(\alpha_{\varepsilon \varepsilon'}):= (\alpha_{\varepsilon \varepsilon'})^N$  whose image lies in the center of $\mathcal{S}_{\omega}(\mathbb{B})$.
 \end{lemma}

 \subsubsection{The case of the triangle} Denote by $\alpha, \beta, \gamma$ the three arcs of Figure \ref{figtriangle} and $\tau$ the automorphism of $\mathcal{S}_{\omega}(\mathbb{T})$ induced by the rotation sending $\alpha, \beta, \gamma$ to $\beta, \gamma, \alpha$ respectively . In \cite[Theorem $4.6$]{LeStatedSkein}, it was proved that the stated skein algebra $\mathcal{S}_{\omega}(\mathbb{T})$ is presented by the generators $\alpha_{\varepsilon \varepsilon'}, \beta_{\varepsilon \varepsilon'}, \gamma_{\varepsilon \varepsilon'}$ and the following relations together with their images through $\tau$ and $\tau^2$:
\begin{eqnarray}
\alpha_{-\varepsilon}\alpha_{+\varepsilon'}&=& A^2 \alpha_{+\varepsilon}\alpha_{-\varepsilon'}-\omega^{-5}C_{\varepsilon'}^{\varepsilon} \label{eq1} \\
\alpha_{\varepsilon -}\alpha_{\varepsilon' +}&=& A^2 \alpha_{\varepsilon +}\alpha_{\varepsilon'-}-\omega^{-5}C_{\varepsilon'}^{\varepsilon} \label{eq2} \\
\beta_{\mu \varepsilon} \alpha_{\mu' \varepsilon'} &=& A\alpha_{\varepsilon \varepsilon'}\beta_{\mu\mu'} - A^2 C_{\mu'}^{\varepsilon} \gamma_{\varepsilon' \mu} \label{eq3}
\\ \alpha_{- \varepsilon}\beta_{\varepsilon' +} &=& A^2 \alpha_{+ \varepsilon} \beta_{\varepsilon' -} - \omega^{-5}\gamma_{\varepsilon \varepsilon'} \label{eq4}
\\ \alpha_{\varepsilon -} \gamma_{+ \varepsilon'} &=& A^2 \alpha_{\varepsilon +} \gamma_{- \varepsilon'} + \omega \beta_{\varepsilon' \varepsilon} \label{eq5}
\end{eqnarray}
\par Here we use the notation $A:= \omega^{-2}$, $C_-^- = C_+^+:=0$, $C_+^-:= -\omega^5$ and $C_-^+:=\omega$.
\begin{figure}[!h] 
\centerline{\includegraphics[width=7cm]{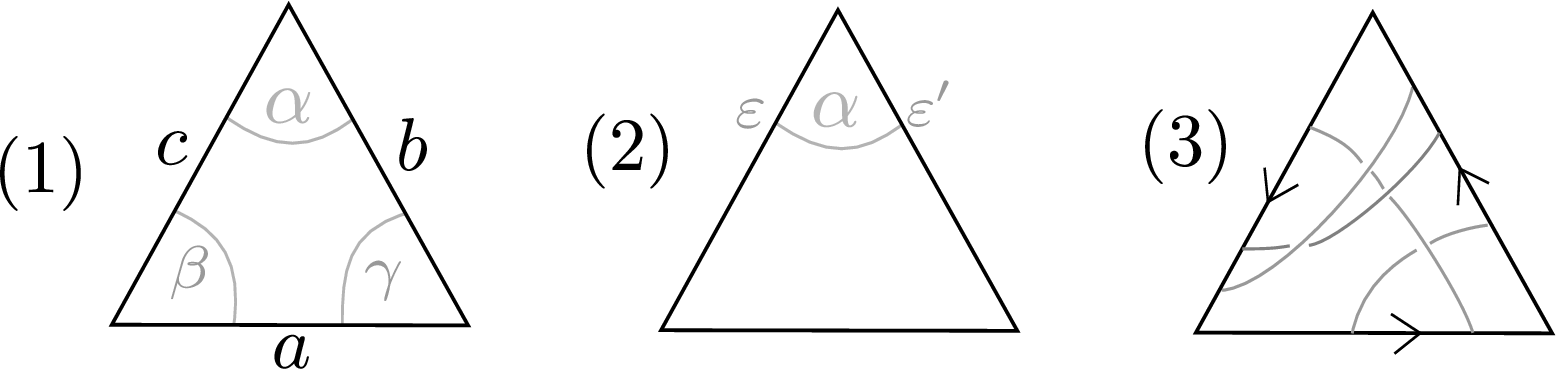} }
\caption{ $(1)$ The three diagrams $\alpha, \beta, \gamma$,  $(2)$  the stated diagram representing $\alpha_{\varepsilon \varepsilon'}$ and  $(3)$ the diagram $\theta^{(2,1,1)}$.}
\label{figtriangle} 
\end{figure} 

 \par When $\omega=+1$, the algebra $\mathcal{S}_{+1}(\mathbb{T})$ has the following simpler presentation. Consider the commutative unital polynomial algebra $\mathcal{A}:= \mathcal{R}[\alpha_{\varepsilon \varepsilon'}, \beta_{\varepsilon \varepsilon'}, \gamma_{\varepsilon \varepsilon'} | \varepsilon, \varepsilon'=\pm ]$. Given $\delta \in \{\alpha, \beta, \gamma \}$, denote by $M_{\delta}$ the $2\times 2$ matrix with coefficients in $\mathcal{A}$ defined by $M_{\delta}:= \begin{pmatrix} \delta_{++} & \delta_{+-} \\ \delta_{-+}& \delta_{--} \end{pmatrix}$ and write $C:= \begin{pmatrix} 0 & 1 \\ -1 & 0 \end{pmatrix}$ and $\mathds{1}:=\begin{pmatrix} 1&0 \\ 0&1 \end{pmatrix}$.
 
 \begin{lemma}\label{lemma_triangle+1}
 The algebra $\mathcal{S}_{+1}(\mathbb{T})$ is isomorphic to  
  \begin{equation*}
 \quotient{\mathcal{R}[\alpha_{\varepsilon \varepsilon'}, \beta_{\varepsilon \varepsilon'}, \gamma_{\varepsilon \varepsilon'} | \varepsilon, \varepsilon'=\pm ]}{\left( \begin{array}{l} \det(M_{\alpha})=\det(M_{\beta})=\det(M_{\gamma})=1, \\ M_{\gamma}CM_{\beta} C M_{\alpha} C = \mathds{1} \end{array} \right)}.
 \end{equation*} 
 \end{lemma}

 \begin{proof} 
		That $\mathcal{S}_{+1}(\mathbb{T})$ commutative is a particular case of \cite[Corollary $2.5$]{LeStatedSkein}. After setting $\omega=+1$ we see that Equations \eqref{eq1} and \eqref{eq2} coincide;  \eqref{eq5} is the image of \eqref{eq4} by rotation, and the latter is a particular case of \eqref{eq3}. 
		Moreover, a direct inspection shows that the other part of \eqref{eq1} and of \eqref{eq3} correspond to $\det(M_{\alpha})=1$ and $(M_{\gamma}C)^{-1}=M_{\beta} C M_{\alpha} C$, respectively. 
 \end{proof}
 
 \begin{lemma}\label{lemma_triangle_preliminar} Suppose that $\omega$ is a root of unity of odd order $N\geq 1$.
 For every $\varepsilon, \varepsilon', \mu, \mu' \in \{ -, +\}$ with $\varepsilon\neq \mu'$, one has
 \begin{equation*}
 \alpha_{\mu' \varepsilon'}^N \beta_{\mu \varepsilon}^N - \alpha_{\varepsilon \varepsilon'}^N \beta_{\mu \mu'}^N =  \gamma_{\varepsilon', \mu}^N.
\end{equation*} 
\end{lemma}

\begin{proof} We suppose that $(\varepsilon, \mu')=(-,+)$. The proof in the case where $(\varepsilon, \mu')=(+,-)$ is similar and left to the reader. For $n\geq 0$, let $D_n$ be the simple diagram made of $n$ parallel copies of $\alpha$ and $n$ parallel copies of $\beta$ and consider the orientation $\mathfrak{o}$ depicted in Figure \ref{fig_equationbidon}. For $\mathbf{\eta}=(\eta_1, \ldots, \eta_n) \in \{-, +\}^n $ let $\mathbf{\eta}^{\vee}:=\{-\eta_n, \ldots, -\eta_1\}$. For $\mathbf{\eta}, \mathbf{\eta}' \in \{ - , +\}^n$ let $s_{\mathbf{\eta}, \mathbf{\eta}'}$ be the state of $D_n$ sending all points of $\partial_b D_n$ to $\varepsilon'$, all points of $\partial_a D_n$ to $\mu$ and the points $(p_1, \ldots, p_n, p_1', \ldots, p_n')$ of $\partial_c D_n$ ordered by $\mathfrak{o}$, to the states $(\eta_1, \ldots, \eta_n, \eta_1', \ldots, \eta_n')$. Write $X_{\mathbf{\eta}, \mathbf{\eta}'}:= [D_n, s_{\mathbf{\eta}, \mathbf{\eta}'}]$. 

\begin{figure}[!h] 
\centerline{\includegraphics[width=10cm]{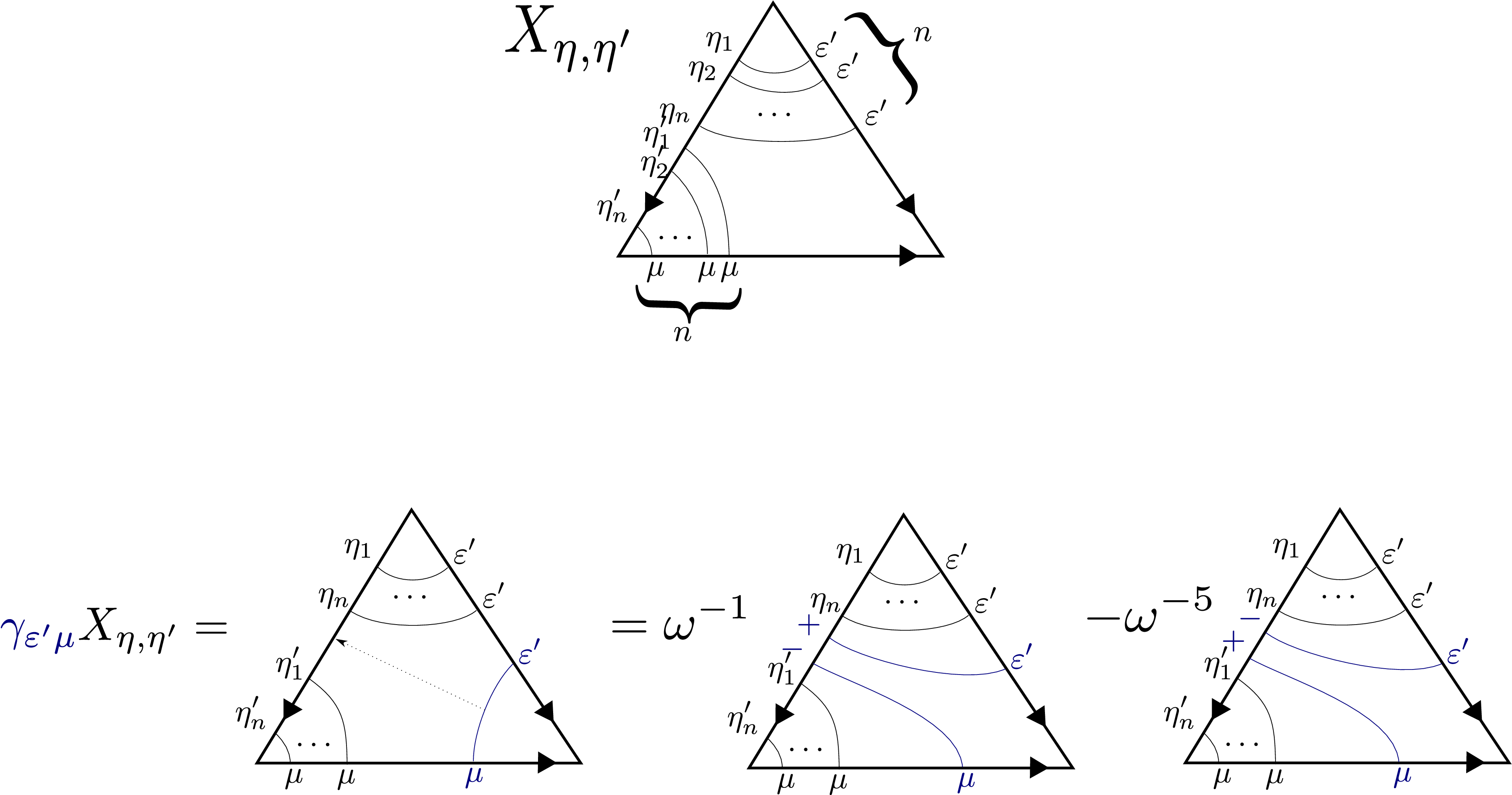} }
\caption{On the top: the element $X_{\eta, \eta'}$. On the bottom: an illustration of Equation \eqref{eq_bidon0}. } 
\label{fig_equationbidon} 
\end{figure} 

 \par Using the skein relations \eqref{eq: skein 2}, as illustrated in Figure \ref{fig_equationbidon}, we find that 
 \begin{equation}\label{eq_bidon0}
  X_{\mathbf{\eta}, \mathbf{\eta}'} \gamma_{\varepsilon', \mu} = \omega^{-1} X_{(\mathbf{\eta}, +), (-, \mathbf{\eta}')} - \omega^{-5} X_{(\mathbf{\eta}, -), (+, \mathbf{\eta}')}, 
  \end{equation}
 where $(\mathbf{\eta}, +) := (\eta_1, \ldots, \eta_n, +)$ and $(-, \mathbf{\eta}'):= (-, \eta'_1, \ldots, \eta_n')$. Let $n_+(\eta)$ be the number of indices $i\in \{1, \ldots, n\}$ such that $\eta_i=+$. 
 Using Equation \eqref{eq_bidon0}, we prove by induction of $n$ that 
 \begin{equation}\label{eq_bidon1}
 (\gamma_{\varepsilon' \mu})^n = \sum_{\eta \in \{-, +\}^n} (\omega^{-1})^{n_+(\eta)} (-\omega^{-5})^{n-n_+(\eta)} X_{\eta, \eta^{\vee}}.
 \end{equation}
 Let $m(\eta):= \# \{ 1\leq i <j \leq n | (\eta_i, \eta_j)=(+, -) \}$ and denote by $\eta_+$ the unique element of $\{-,+\}^n$ such that $n_+(\eta)=n_+(\eta_+)$ and  $m(\eta_+)=0$. Note that $m(\eta)=m(\eta^{\vee})$. Using the skein relation \eqref{eq: skein 2}, we find that for any $\eta, \eta' \in \{-,+\}^n$, one has 
 \begin{equation}\label{eq_bidon2}
 X_{\eta, \eta'} = q^{m(\eta)+m(\eta')} X_{\eta_+, \eta'_+}.
 \end{equation}
 
 For $1\leq k \leq N$, let $\eta_+^{(k)}\in \{-, +\}^N$ be the unique element such that $m(\eta_+^{(k)})=0$ and $n_+(\eta_+^{(k)})=k$, \textit{i.e.} ${\eta_+^{(k)}}_i = \left\{ \begin{array}{ll} - & \mbox{, for }1\leq i \leq N-k; \\ + & \mbox{, for }i>N-k.\end{array} \right.$
 Putting Equations \eqref{eq_bidon1} and \eqref{eq_bidon2} together, one finds that 
 $$ (\gamma_{\varepsilon' \mu})^N = \sum_{k=0}^N (\omega^{-1})^{k} (-\omega^{-5})^{N-k} \left(\sum_{\eta \in \{-,+\}^N, n_+(\eta)=k} q^{2m(\eta)} \right) X_{\eta_+^{(k)}, \eta_+^{(k)\vee}}.$$
Now, a simple computation shows that 
\begin{multline*}
 \left(\sum_{\eta \in \{-,+\}^N, n_+(\eta)=k} q^{2m(\eta)} \right) = q^{2nN-n(n-1)} \sum_{1\leq i_1 <i_2 < \ldots <i_n\leq N} q^{2(i_1+\ldots +i_n)} \\ = \left\{ \begin{array}{ll} 1 & \mbox{, if }k=0 \mbox{ or }k=N; \\ 0 & \mbox{, else.}\end{array}\right.\end{multline*}
Therefore, we find that 
$$ (\gamma_{\varepsilon' \mu})^N = X_{\eta_+^{(N)}, \eta_-^{(N)}} - X_{\eta_-^{(N)}, \eta_+^{(N)}} = \alpha_{+ \varepsilon'}^N \beta_{\mu -}^N  - \alpha_{- \varepsilon'}^N \beta_{\mu +}^N.$$
Note that we used that $(-1)^N=-1$, so that $N$ is odd. This concludes the proof.

\end{proof} 
 
 \begin{lemma}\label{lemma_center_triangle}
 Suppose that $\omega$ is a root of unity of odd order $N\geq 1$. There exists an injective morphism of algebras $j_{\mathbb{T}}:\mathcal{S}_{+1}(\mathbb{T}) \rightarrow \mathcal{S}_{\omega}(\mathbb{T})$, whose image lies in the center of $\mathcal{S}_{\omega}(\mathbb{T})$,  characterized by $j_{\mathbb{T}}(\delta_{\varepsilon \varepsilon'}):= (\delta_{\varepsilon \varepsilon'})^N$ for $\delta\in \{\alpha, \beta, \gamma\} $ and $\varepsilon, \varepsilon' = \pm$.  Moreover, if $a$ is a boundary arc of $\mathbb{T}$, the following diagrams commute:
 $$ \begin{array}{ll} 
 \begin{tikzcd}
 \mathcal{S}_{+1}(\mathbb{T}) \arrow[r, "\Delta_a^L"] \arrow[d, hook, "j_{\mathbb{T}}"] & \mathcal{S}_{+1}(\mathbb{B})\otimes \mathcal{S}_{+1}(\mathbb{T}) \arrow[d, hook, "j_{\mathbb{B}}\otimes j_{\mathbb{T}}"] \\
 \mathcal{S}_{\omega}(\mathbb{T}) \arrow[r, "\Delta_a^L"] & \mathcal{S}_{\omega}(\mathbb{B})\otimes \mathcal{S}_{\omega}(\mathbb{T})
 \end{tikzcd}
 &  
 \begin{tikzcd}
 \mathcal{S}_{+1}(\mathbb{T}) \arrow[r, "\Delta_a^R"] \arrow[d, hook, "j_{\mathbb{T}}"] &  \mathcal{S}_{+1}(\mathbb{T})\otimes \mathcal{S}_{+1}(\mathbb{B}) \arrow[d, hook, "j_{\mathbb{T}}\otimes j_{\mathbb{B}}"] \\
 \mathcal{S}_{\omega}(\mathbb{T}) \arrow[r, "\Delta_a^R"] &  \mathcal{S}_{\omega}(\mathbb{\mathbb{T}})\otimes \mathcal{S}_{\omega}(\mathbb{B})
 \end{tikzcd}
 \end{array}
 $$
 
 \end{lemma}
 
 \begin{proof}

	We proceed in a similar way than Lemma \ref{lemma_center_bigone}, by showing first that the extension of the assignment  $j_{\mathbb{T}}(\delta_{\varepsilon \varepsilon'}):= \delta_{\varepsilon \varepsilon'}^N$ to a morphism of algebras is well-defined. In virtue of Lemma \ref{lemma_triangle+1} and by the rotation automorphism, it is enough to show that  
	 $\alpha_{\varepsilon \varepsilon'}^N$ lies in the center of   $\mathcal{S}_{\omega}(\mathbb{T})$ and that $j_{\mathbb{T}}$ sends $\det( M_{\alpha})-1$ and $M_{\gamma}CM_{\beta} C M_{\alpha} C - \mathds{1}$ to zero.

	First remark that the relations \eqref{eq1} and \eqref{eq2} put together coincide with the defining relations of $\mathcal{S}_{\omega}(\mathbb{B})$, hence one has an inclusion of algebras $\phi: \mathcal{S}_{\omega}(\mathbb{B}) \hookrightarrow \mathcal{S}_\omega(\mathbb{T})$ defined by $\phi(\alpha_{\varepsilon\varepsilon'})=\alpha_{\varepsilon\varepsilon'}$. 
	By applying Lemma \ref{lemma_center_bigone}, one obtains an inclusion $\phi\circ j_{\mathbb{B}}: \mathcal{S}_{+1}(\mathbb{B}) \hookrightarrow \mathcal{S}_\omega(\mathbb{T})$ which coincides with $j_{\mathbb{T}}$ on the $\alpha_{\varepsilon \varepsilon'}$'s. It remains to show that the $\alpha_{\varepsilon \varepsilon'}^N$'s commute with the $\beta_{\mu \mu'}$'s and the $\gamma_{\mu \mu'}$'s, and that $j_{\mathbb{T}}$ vanishes on $M_{\gamma}CM_{\beta} C M_{\alpha} C - \mathds{1}$.  

We have  $\alpha_{\varepsilon \varepsilon'}^N \beta_{\mu \varepsilon} = A^{-N}\beta_{\mu \varepsilon}\alpha_{\varepsilon\varepsilon'}^N = \beta_{\mu \varepsilon}\alpha_{\varepsilon\varepsilon'}^N$. 
From 
\begin{align*}
 \alpha_{+ \varepsilon}^N \beta_{\varepsilon' -} &= \alpha_{+ \varepsilon}^{N-1}(A^{-2}\alpha_{- \varepsilon}\beta_{\varepsilon' +} +\omega^{-1} \gamma_{\varepsilon \varepsilon'})
\\ &= (A^{-3N+1}\alpha_{- \varepsilon} \beta_{\varepsilon' +} + \omega^{-1}A^{N-1} \gamma_{\varepsilon \varepsilon'})\alpha_{+ \varepsilon}^{N-1}
\end{align*}
and 
\begin{equation*}
\beta_{\varepsilon' -} \alpha_{+ \varepsilon}^N = (A\alpha_{- \varepsilon}\beta_{\varepsilon' +} + \omega \gamma_{\varepsilon \varepsilon'})\alpha_{+ \varepsilon}^{N-1},
\end{equation*}
one obtains 
$$ \alpha_{+ \varepsilon}^N\beta_{\varepsilon' -} - \beta_{\varepsilon' -}\alpha_{+ \varepsilon}^N = (A(A^{-3N}-1)\alpha_{- \varepsilon}\beta_{\varepsilon' +} + \omega(A^N-1)\gamma_{\varepsilon \varepsilon'})\alpha_{+ \varepsilon}^{N-1} = 0. $$
\par Similarly, we compute: 
\begin{eqnarray*}
 \alpha_{- \varepsilon}^N \beta_{\varepsilon' +} &=& \alpha_{- \varepsilon}^{N-1}(A^{2}\alpha_{+ \varepsilon}\beta_{\varepsilon' -} -\omega^{-5} \gamma_{\varepsilon \varepsilon'})
\\ &=& (A^{N+1}\alpha_{+ \varepsilon} \beta_{\varepsilon' -} - \omega^{-3}A^{N} \gamma_{\varepsilon \varepsilon'})\alpha_{- \varepsilon}^{N-1}; 
\\ \beta_{\varepsilon' +} \alpha_{- \varepsilon}^N &=& (A\alpha_{+ \varepsilon}\beta_{\varepsilon' -} - \omega^{-3} \gamma_{\varepsilon \varepsilon'})\alpha_{- \varepsilon}^{N-1}.
\end{eqnarray*}
Thus we find: 
$$ \alpha_{- \varepsilon}^N\beta_{\varepsilon' +} - \beta_{\varepsilon' +}\alpha_{- \varepsilon}^N = (A(A^{N}-1)\alpha_{+ \varepsilon}\beta_{\varepsilon' -} - \omega^{-3}(A^N-1)\gamma_{\varepsilon \varepsilon'})\alpha_{- \varepsilon}^{N-1} = 0. $$
 So we have proven that $\alpha_{\varepsilon \varepsilon'}^N$ commutes with every elements $\beta_{\mu \mu'}$. 
The commutativity of  $\alpha_{\varepsilon \varepsilon'}^N$ with each element $\gamma_{\mu \mu'}$ is shown in a very similar way. 

	Next, showing that $j_{\mathbb{T}}$ vanishes on $M_{\gamma}CM_{\beta} C M_{\alpha} C - \mathds{1}$ amounts to showing that 
	\begin{equation*}
	\beta_{\mu \varepsilon}^N \alpha_{\mu' \varepsilon'}^N - \alpha_{\varepsilon \varepsilon'}^N \beta_{\mu \mu'}^N =  \gamma_{\varepsilon', \mu}^N  \text{ for } \varepsilon\neq \mu'.
	\end{equation*} 
	This was proved in Lemma \ref{lemma_triangle_preliminar}.

	Now let us prove that $j_{\mathbb{T}}$ is injective.  
	To this end, let us consider the following basis of $\mathcal{S}_{\omega}(\mathbb{T})$. 
	
	Consider the counter-clockwise orientation $\mathfrak{o}$ of the boundary arcs of $\mathbb{T}$ as in Figure \ref{figtriangle}. 
	Given $\mathbf{k}=(k_{\alpha}, k_{\beta}, k_{\gamma}) \in (\mathbb{Z}^{\geq 0})^3$, denote by $\theta^{\mathbf{k}}$ the (not simple) diagram $\alpha^{k_{\alpha}}\beta^{k_{\beta}}\gamma^{k_{\gamma}}$; see Figure \ref{figtriangle} for an example. 
	By Proposition \ref{prop_alternative_basis} the set of classes $[\theta^{\mathbf{k}}, s]$, where $s$ is $\mathfrak{o}$-increasing, forms a basis of $\mathcal{S}_{\omega}(\mathbb{T})$.  
	
	By construction, $j_{\mathbb{T}}$  sends the elements  $[\theta^{\mathbf{k}}, s]$ of $\mathcal{S}_{+1}(\mathbb{T})$, where $s$ is $\mathfrak{o}$-increasing, to some basis elements $[\theta^{N \mathbf{k}}, s']$, where $s'$ is also $\mathfrak{o}$ increasing, therefore  $j_{\mathbb{T}}$ is injective. 
	
	It remains to prove that  $j_{\mathbb{T}}$ is a morphism of Hopf comodules. To avoid confusion, let us denote by $x_{\varepsilon \varepsilon'}$ the generators of $\mathcal{S}_{\omega}(\mathbb{B})$  and reserve the notation $\alpha_{\varepsilon \varepsilon'}$ for the element of $\mathcal{S}_{\omega}(\mathbb{T})$. By definition, we have $\Delta_c^L(\alpha_{\varepsilon \varepsilon'}) = x_{\varepsilon +} \otimes \alpha_{+ \varepsilon'} + x_{\varepsilon - } \otimes \alpha_{- \varepsilon'}$. Write $u:= x_{\varepsilon +} \otimes \alpha_{+ \varepsilon'}$ and $v:= x_{\varepsilon - } \otimes \alpha_{- \varepsilon'}$. Since $uv= q^{-2}vu$, by Lemma \ref{lemma_qbinomial} we have $(u+v)^N=u^N+v^N$, so 
	\begin{multline*} \Delta_c^L (j_{\mathbb{B}}(\alpha_{\varepsilon \varepsilon'})) = \left( \Delta_c^L(\alpha_{\varepsilon \varepsilon'}) \right)^N = (u+v)^N = u^N +v^N \\= x_{\varepsilon +}^N\otimes \alpha_{+ \varepsilon'}^N + x_{\varepsilon -}^N \otimes \alpha_{- \varepsilon'}^N = j_{\mathbb{B}}\otimes j_{\mathbb{T}} (\Delta_c^L (\alpha_{\varepsilon \varepsilon'})).\end{multline*}
	The proof that $\Delta_b^L (j_{\mathbb{B}}(\alpha_{\varepsilon \varepsilon'})) =  j_{\mathbb{B}}\otimes j_{\mathbb{T}} (\Delta_b^L (\alpha_{\varepsilon \varepsilon'}))$ is done using a similar computation and the equality $\Delta_a^L (j_{\mathbb{B}}(\alpha_{\varepsilon \varepsilon'})) =  j_{\mathbb{B}}\otimes j_{\mathbb{T}} (\Delta_a^L (\alpha_{\varepsilon \varepsilon'}))$ holds since both sides are equal to $1 \otimes \alpha_{\varepsilon \varepsilon'}^N$. By symmetry in the generators $\alpha, \beta, \gamma$, we have prove that $j_{\mathbb{B}}$ commutes with the left comodule maps. That it commutes with the right comodule maps is proved similarly. 
	 This concludes the proof.

\end{proof} 

\subsubsection{The general case: proof of Theorem \ref{theorem2}}

	In this section we prove Theorem \ref{theorem2} that we remind here for the convenience of the reader:
	
	 \begin{theorem}
	Suppose that  $\omega$ is a root of unity of odd order $N\geq 1$ and $\mathbf{\Sigma}$ a punctured surface. There exists an embedding 
	\begin{equation*}
	j_{\mathbf{\Sigma}} : \mathcal{S}_{+1}(\mathbf{\Sigma}) \hookrightarrow \mathcal{Z}\left( \mathcal{S}_{\omega}(\mathbf{\Sigma}) \right)
	\end{equation*}
	of the (commutative) stated skein algebra with parameter $+1$ into the center of the stated skein algebra with parameter $\omega$. Moreover, the morphism $j_{\mathbf{\Sigma}}$  is characterized by the property that it sends a closed curve $\gamma$ to $T_N(\gamma)$ and a stated arc $\alpha_{\varepsilon \varepsilon'}$ to $\alpha_{\varepsilon \varepsilon'}^{(N)}$, where $\alpha_{\varepsilon \varepsilon'}^{(N)}$ is the tangle made by stacking $N$ parallel copies of $\alpha_{\varepsilon \varepsilon'}$ on top of the others.
\end{theorem}
Recall from Section \ref{sec_basecheloud} that closed curves and arcs do not have self-intersection points by definition.
	We divide the proof in five steps. 
	\\
	\par In Step $1$, we show that the decomposition Theorem \ref{theorem1} together with the two previous sections provide an injective morphism of algebras 
	\begin{equation}\label{eq: morph j surf triang}
	j_{(\mathbf{\Sigma},\Delta)} : \mathcal{S}_{+1}(\mathbf{\Sigma}) \hookrightarrow \mathcal{S}_{\omega}(\mathbf{\Sigma}),  
	\end{equation}
	which is central. 
	We study further properties of $j_{(\mathbf{\Sigma},\Delta)}$ and we show that it is does \emph{not} depend  on a topological triangulation $\Delta$. 
	The other steps are devoted to making explicit the morphism $j_{(\mathbf{\Sigma},\Delta)}$ on arcs and loops. In Step $2$ to $4$, we suppose that the punctured surface has a non-degenerated triangulation (see below); in Step $5$ we treat the other punctured surfaces. 
	\\
	\par In Step $2$, we prove that  $j_{(\mathbf{\Sigma},\Delta)}$ sends the stated arcs that have their endpoints on \emph{two different} boundary arcs of $\Sigma$, to their $N$-th power. 
	\\
	\par In Step $3$, we prove that $j_{(\mathbf{\Sigma},\Delta)}$ sends some particular closed curves of $\Sigma_{\mathcal{P}}$ to their $N$-th Chebyshev polynomial of first kind.  
	\\
	\par Step $4$ is more involved. We first prove a structural result. Adding a puncture on a surface $\mathbf{\Sigma}$ gives rise to a surjective map $\varphi$ from the skein algebra of the new punctured surface to that of the initial one defined in Section $2.3$. We show that    $j_{(\mathbf{\Sigma},\Delta)}$ commutes with these surjections (see Lemma \ref{lemma_add_puncture}).  
	From this,  we deduce the image by $j_{(\mathbf{\Sigma},\Delta)}$ of  stated arcs that have both their endpoints on the \emph{same} boundary arc of $\Sigma$ and of \emph{any} closed curve of $\Sigma_{\mathcal{P}}$.  
	\\
	\par In Step $5$, we treat the remaining cases of connected punctured surfaces that do not admit a non-degenerate topological triangulation (including those with no puncture). The proof consists, again, in adding a puncture and using the previous study.
	\\
	\par These five steps prove Theorem \ref{theorem2}.
		\\
	
	All along this section, $\mathbf{\Sigma}$ is a punctured surface, $\Delta$  a topological triangulation $\mathbf{\Sigma}$ and $\omega$  a root of unity of odd order $N\geq 1$. 
	Except for Step $1$ and $5$, the triangulation $\Delta$ is required to be \emph{non-degenerate}, that is, such that each of its inner edges separates two distinct faces.

\paragraph{Step 1: formal definition.} 

Assume that $\mathbf{\Sigma}$ admits a (possibly degenerate) triangulation $\Delta$. 
Consider the following diagram, where both lines are exact by Theorem \ref{theorem1} and the vertical maps are given by Lemmas \ref{lemma_center_bigone} and \ref{lemma_center_triangle}. 
\begin{equation}\label{diag j}
\begin{tikzcd}
0 \arrow[r,""] & 
\mathcal{S}_{+1}(\mathbf{\Sigma}) \arrow[r, "i^{\Delta}"] \arrow[d, hook, dotted, "\exists!", "j_{(\mathbf{\Sigma}, \Delta)}"'] & 
\otimes_{\mathbb{T}\in F(\Delta)} \mathcal{S}_{+1}(\mathbb{T}) \arrow[r,"\Delta^L - \sigma \circ \Delta^R"] \arrow[d, hook,"\otimes_{\mathbb{T}}  j_{\mathbb{T}}"] &
\left( \otimes_{e\in \mathring{\mathcal{E}}(\Delta)} \mathcal{S}_{+1}(\mathbb{B}) \right) \otimes \left( \otimes_{\mathbb{T}\in F(\Delta)} \mathcal{S}_{+1}(\mathbb{T}) \right)
\arrow[d, hook,"(\otimes_e j_{\mathbb{B}})\otimes (\otimes_{\mathbb{T}} j_{\mathbb{T}})"] \\
0 \arrow[r,""] & 
\mathcal{S}_{\omega}(\mathbf{\Sigma}) \arrow[r, "i^{\Delta}"]  & 
\otimes_{\mathbb{T}\in F(\Delta)} \mathcal{S}_{\omega}(\mathbb{T}) \arrow[r,"\Delta^L - \sigma \circ \Delta^R"] &
\left( \otimes_{e\in \mathring{\mathcal{E}}(\Delta)} \mathcal{S}_{\omega}(\mathbb{B}) \right) \otimes \left( \otimes_{\mathbb{T}\in F(\Delta)} \mathcal{S}_{\omega}(\mathbb{T}) \right)
\end{tikzcd}
\end{equation}
The existence of an injective morphism $j_{(\mathbf{\Sigma},\Delta)} : \mathcal{S}_{+1}(\mathbf{\Sigma}) \hookrightarrow \mathcal{S}_{\omega}(\mathbf{\Sigma})$ follows from the exactness of the lines and the injectivity of $\otimes_{\mathbb{T}\in F(\Delta)}j_{\mathbb{T}}$ (and the fact that all maps involved in the diagram are algebra morphisms). Moreover, since $j_{\mathbb{T}}$ is central, so is $j_{(\mathbf{\Sigma},\Delta)}$. 

Let us show that $j_{(\mathbf{\Sigma},\Delta)}$ is compatible with the gluing maps. 
\begin{lemma}\label{lem: j commute with i}
If $a$ and $b$ are two boundary arcs of $\mathbf{\Sigma}$, the following diagram commutes:
$$ \begin{tikzcd}
\mathcal{S}_{+1}(\mathbf{\Sigma}_{|a\#b}) 
\arrow[r, hook, "j_{\mathbf{\Sigma}_{a\#b}}"] 
\arrow[d, hook, "i_{|a\#b}"] 
& \mathcal{S}_{\omega} (\mathbf{\Sigma}_{|a\#b}) \arrow[d, hook, "i_{|a\#b}"] \\
\mathcal{S}_{+1}(\mathbf{\Sigma})
\arrow[r, hook, "j_{\mathbf{\Sigma}}"] 
& \mathcal{S}_{\omega}(\mathbf{\Sigma})
\end{tikzcd} $$
\end{lemma}

\begin{proof}

 Let $\Delta_{a\#b}$ the topological triangulation of $\mathbf{\Sigma}_{|a\#b}$ that is induced by $\Delta$.  
Let us consider the following diagram. 
	$$ \begin{tikzcd}
	\mathcal{S}_{+1}(\mathbf{\Sigma}_{|a\#b}) \arrow[r, hook, "i_{|a\#b}"] \arrow[d, hook, "j_{(\mathbf{\Sigma}_{|a\#b}, \Delta_{a\#b})}"] 
	\arrow[rr, bend left=20, "i^{\Delta_{a\#b}}"]
	&
	\mathcal{S}_{+1}(\mathbf{\Sigma}) \arrow[r, hook, "i^{\Delta}"] \arrow[d, hook, "j_{(\mathbf{\Sigma}, \Delta)}"] &
	\otimes_{\mathbb{T}}\mathcal{S}_{+1}(\mathbb{T}) \arrow[d, hook, "\otimes_{\mathbb{T}}j_{\mathbb{T}}"] \\
	\mathcal{S}_{\omega}(\mathbf{\Sigma}_{|a\#b})  \arrow[r, hook, "i_{|a\#b}"]
	\arrow[rr, bend right=20, "i^{\Delta_{a\#b}}"]
	&
	\mathcal{S}_{\omega}(\mathbf{\Sigma}) \arrow[r, hook, "i^{\Delta}"] &
	\otimes_{\mathbb{T}}\mathcal{S}_{\omega}(\mathbb{T}) 
	\end{tikzcd}$$
	The outer triangles commute by coassociativity of the gluing maps. Two of the three squares commute by Diagram \eqref{diag j}. Since $i^{\Delta}$ is injective, the remaining (left-hand side) square commutes.     
\end{proof}

We now prove that the morphism $j_{(\mathbf{\Sigma}, \Delta)}$ does not depend on $\Delta$.  We first need a preliminary result.

\begin{lemma}\label{lemma_square}
	Let $Q$ be a square (\textit{i.e.} a  disc with four punctures on its boundary) 
	and  $\Delta_Q$ a topological triangulation of $Q$. 
	If $\alpha_{\varepsilon \varepsilon'} \in \mathcal{S}_{\omega}(Q)$ is the class of a stated arc, then $j_{(Q, \Delta_Q)}(\alpha_{\varepsilon \varepsilon'})= \alpha_{\varepsilon \varepsilon'}^N$. 
	In particular, $j_{(Q,\Delta_Q)}$ does not depend on $\Delta_Q$. 
\end{lemma}

\begin{proof}
	Let $e$ be the inner edge of $\Delta_Q$ which is a common boundary arc of two triangles $\mathbb{T}_1$ and $\mathbb{T}_2$. 
	Make the intersection  $\alpha\cap e$ transversal and minimal via an isotopy on $\alpha$.
	If the intersection is empty, then $\alpha$ is included 
	in one of the triangles and the lemma follows from Lemma \ref{lemma_center_triangle}. 
	If  $\alpha \cap e$ is not empty, then it has only one element.  
	Therefore, by letting $\alpha^{\mathbb{T}_i}:= \alpha \cap \mathbb{T}_i$ for $i=1,2$, one has  $i^{\Delta_Q}(\alpha_{\varepsilon \varepsilon'}) = \alpha^{\mathbb{T}_1}_{\varepsilon +} \otimes \alpha^{\mathbb{T}_2}_{+ \varepsilon'} + \alpha^{\mathbb{T}_1}_{\varepsilon -} \otimes \alpha^{\mathbb{T}_2}_{- \varepsilon'}$. 
	Write $x:= \alpha^{\mathbb{T}_1}_{\varepsilon +} \otimes \alpha^{\mathbb{T}_2}_{+ \varepsilon'}$ and $y:= \alpha^{\mathbb{T}_1}_{\varepsilon -} \otimes \alpha^{\mathbb{T}_2}_{- \varepsilon'}$ and remark that $xy=q^{-2}yx$. 
	By Lemma  \ref{lemma_qbinomial} one has 
	\begin{equation*}
	i^{\Delta_Q}(\alpha_{\varepsilon \varepsilon'}^N)= i^{\Delta_Q}(\alpha_{\varepsilon \varepsilon'})^N = (x+y)^N = x^N +y^N = (j_{\mathbb{T}_1} \otimes j_{\mathbb{T}_2}) \circ i^{\Delta_Q} (\alpha_{\varepsilon \varepsilon'}).
	\end{equation*}
	Hence one has the equality $j_{(Q, \Delta_Q)}(\alpha_{\varepsilon \varepsilon'})= \alpha_{\varepsilon \varepsilon'}^N$.
\end{proof}

\begin{lemma}\label{lemma_independance}
	The morphism $j_{(\mathbf{\Sigma}, \Delta)}$ does not depend on $\Delta$. 

\end{lemma}

\begin{proof}

	Every two triangulations can be related by a finite sequence of flips on the edges. Therefore, it is enough to prove that if $\Delta'$ differs from $\Delta$ by a flip of one edge, then $j_{(\mathbf{\Sigma}, \Delta)}=j_{(\mathbf{\Sigma}, \Delta')}$.  

Let $e$ be  an inner edge of $\Delta$ that bounds two distinct faces $\mathbb{T}_1$ and $\mathbb{T}_2$. 
Consider the topological triangulation $\Delta'$ obtained from $\Delta$ by flipping the edge $e$ inside the square $Q=\mathbb{T}_1\cup \mathbb{T}_2$. Let $i: \mathcal{S}_{\omega}(\mathbf{\Sigma}) \hookrightarrow \mathcal{S}_{\omega}(\mathbf{\Sigma}\setminus Q) \otimes \mathcal{S}_{\omega}(Q)$ be the gluing morphism. 
By Lemma \ref{lemma_square}, the morphism $j_Q : \mathcal{S}_{+1}(Q) \hookrightarrow \mathcal{S}_{\omega}(Q)$ does not depend on the triangulation of $Q$. Therefore, by Lemma \ref{lem: j commute with i}, both the morphisms $j_{(\mathbf{\Sigma}, \Delta)}$ and $j_{(\mathbf{\Sigma}, \Delta')}$ make the following diagram commutative:
	$$ \begin{tikzcd}
	\mathcal{S}_{+1}(\mathbf{\Sigma}) \arrow[r, hook, "i"] \arrow[d, hook, shift right=1.5ex, "j_{(\mathbf{\Sigma}, \Delta')}"'] \arrow[d, hook, "j_{(\mathbf{\Sigma}, \Delta)}"] &
	\mathcal{S}_{+1}(\mathbf{\Sigma}\setminus Q)\otimes \mathcal{S}_{+1}(Q) \arrow[d, hook, "j_{(\mathbf{\Sigma}\setminus Q, \Delta_{\mathbf{\Sigma}\setminus Q})}\otimes j_Q"] \\
	\mathcal{S}_{\omega}(\mathbf{\Sigma}) \arrow[r, hook, "i"] &
	\mathcal{S}_{\omega}(\mathbf{\Sigma}\setminus Q)\otimes \mathcal{S}_{\omega}(Q).
	\end{tikzcd}$$
	This proves that $j_{(\mathbf{\Sigma}, \Delta)}=j_{(\mathbf{\Sigma}, \Delta')}$ and concludes the proof.

\end{proof}

\paragraph{Step 2: arcs with endpoints in distinct boundary arcs.} 

We now assume that the triangulation $\Delta$ is non-degenerate.
\begin{lemma}\label{lemma_baleze} 
	If $\alpha_{\varepsilon \varepsilon'}\in \mathcal{S}_{\omega}(\mathbf{\Sigma})$ is  the class of a stated arc such that its  endpoints lie on two different boundary arcs, then $j_{\mathbf{\Sigma}} (\alpha_{\varepsilon \varepsilon'}) = \alpha_{\varepsilon \varepsilon'}^N$. 
\end{lemma}

\begin{proof}	 
By the defining property of  $j_{\mathbf{\Sigma}}$, as depicted in Diagram \eqref{diag j}, it is enough to prove that   
\begin{equation}\label{eq: coprod j is j coprod}
i^{\Delta}(\alpha_{\varepsilon\varepsilon'}^N)=(\otimes_{\mathbb{T}\in F(\Delta)}j_{\mathbb{T}}) i^{\Delta}(\alpha_{\varepsilon\varepsilon'}).
\end{equation}

Without lost of generality, we suppose that the arc $\alpha$ is in minimal and transverse position with the edges of $\Delta$. 
Let $T$ be a (vertical framed) tangle of $\Sigma_{\mathcal{P}}\times (0,1)$ that projects on $\alpha$ and such that its height projection is an injective map (this is possible since $\alpha$ is an arc). Note that for each  $\mathbb{T} \in F(\Delta)$, the tangle $T_{\mathbb{T}}:=T \cap (\mathbb{T}\times (0,1) )$ may have various connected components; since the height projection is injective, these components are ordered by height. 
Let $T^{(N)}$ be a tangle of $N$ parallel copies of $T$ obtained by stacking $N$ copies of $T$, but close enough to have the following property. 
For each $\mathbb{T} \in F(\Delta)$, if $T_1$ and $T_2$ are two connected components of $T_{\mathbb{T}}$ such that $T_1$ is below $T_2$, then, in $T^{(N)}_{\mathbb{T}}:=T^{(N)} \cap (\mathbb{T}\times (0,1))$, each copy of $T_1$ is below all the copies of $T_2$. See Figure \ref{figtriangle3d} for an illustration. 
Note that since $\alpha$ is an arc with boundary points at two distinct boundary arcs, the tangle $T^{(N)}$ is a representative of the $N$-th product of $\alpha_{\varepsilon\varepsilon'}$ in $\mathcal{S}_{\omega}(\mathbf{\Sigma})$; otherwise it may not be true. 

	\begin{figure}[!h] 
	\centerline{\includegraphics[width=8cm]{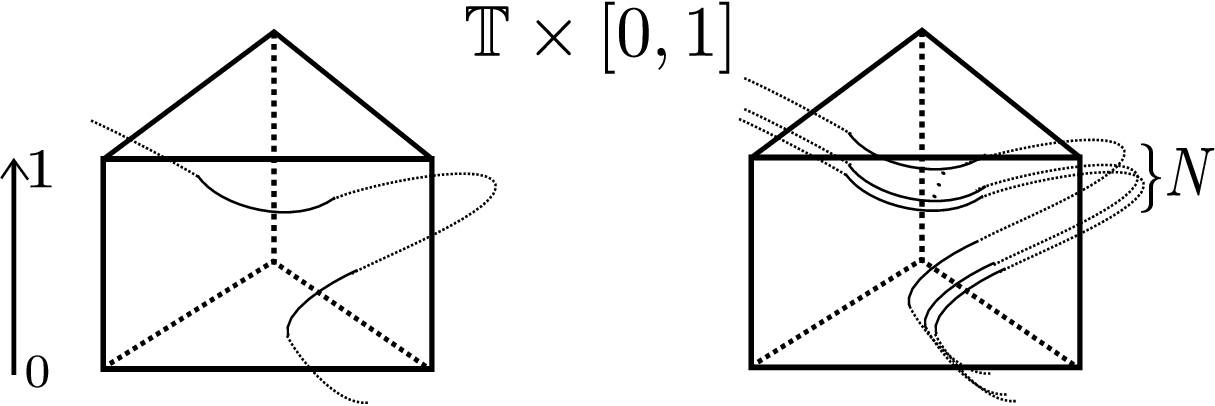} }
	\caption{Instance of tangles $T_{\mathbb{T}}$ and $T^{(N)}_{\mathbb{T}}$.  } 
	\label{figtriangle3d} 
\end{figure}

The left-hand term of \eqref{eq: coprod j is j coprod} can be described as the cutting of   $T^{(N)}$ along each edge of the triangulation, and summing the result over all possible states at each edge. More formally, it is described as follows.

Let $K$ be a subset of edges of $\Delta$ that intersect $\alpha$. 
We let $\St_K(\alpha)$ be the set of maps $s: T \cap (K\times (0,1)) \rightarrow \{-,+\}$. 
We identify $\St_K(\alpha)$ with $ \sqcup_{e\in K} \St_{\{e\}}(\alpha)$, which allows us to write $s\in \St_K(\alpha)$ as $\sqcup s_e$. 
We will only consider the two sets $K$: the set $E$ of all the \emph{internal} edges of $\Delta$ that intersect $\alpha$, and the set $K=\{e\}$ for an edge $e$. 
 
For $s\in \St_E(\alpha)$, write $s^{(N)}:=(s, \ldots, s)\in \St_E(\alpha)^{\times N}$. 
We denote by $s_0$ the state of $\alpha_{\varepsilon\varepsilon'}$ (so, one has $\alpha_{\varepsilon\varepsilon'}=[T,s_0]$). 

For $\mathbf{s}= (s_{1}, \ldots, s_{N})\in \St_E(\alpha)^{\times N}$, we let 
$$ \alpha(\mathbf{s}) := \otimes_{\mathbb{T}\in F(\Delta)} [T^{(N)}_{\mathbb{T}}, ( \mathbf{s}\sqcup s_0^{(N)})_{|\partial\mathbb{T}}] \in \otimes_{\mathbb{T}\in F(\Delta)}\mathcal{S}_{\omega}(\mathbb{T}), $$
where we associate, to the $k$-th copy of $T^{(N)}_{\mathbb{T}}$, the restriction of the state $s_{k}$. 
With this notation, the left-hand term of \eqref{eq: coprod j is j coprod} can be written as 
\begin{equation}\label{eq:j1}
  i^{\Delta}(\alpha_{\varepsilon \varepsilon'}^N)  = \sum_{\mathbf{s}\in \St_E(\alpha)^{\times N}} \alpha(\mathbf{s}). 
\end{equation}

Now, let us describe the right-hand term of  \eqref{eq: coprod j is j coprod}. 
Note that the construction of $T^{(N)}$ ensures that, for each triangle  $\mathbb{T}$ and each state $s$ of $T_{\mathbb{T}}$, one has $j_{\mathbb{T}}([T_{\mathbb{T}},s])=[T^{(N)}_{\mathbb{T}},s^{(N)}]$. 
Therefore, using that $j_{\mathbb{T}}$ is an algebra morphism, one has 
\begin{equation}\label{eq:j2} 
\left( \otimes_{\mathbb{T}\in F(\Delta)} j_{\mathbb{T}} \right) i^{\Delta} (\alpha_{\varepsilon \varepsilon'}) = \sum_{s\in \St_E(\alpha)} \alpha(s^{(N)}).
\end{equation}

Let $Y$ be the set of non-diagonal states $\St_E(\alpha)^{\times N}\setminus \{(s,...,s)| s\in \St_E(\alpha)\}$. 
The sum in \eqref{eq:j1} and in \eqref{eq:j2} differ by the sum of $\alpha(\mathbf{s})$ for $\mathbf{s}\in Y$. 

Let us fix an edge $e$ of $E$ and let us split  $Y$  into $J\sqcup Y_e$ where $Y_e$ is the set of $N$-tuples of states at $e$, that is,  $Y_e=\{\mathbf{s}\in Y~|~ \mathbf{s}: T^{(N)} \cap (e\times (0,1))\to \{-,+\}\}$. 
Therefore, showing  \eqref{eq: coprod j is j coprod} amounts to showing that 
\begin{equation*}
\sum_{\mathbf{s'}\in J} \sum_{\mathbf{s}\in Y_e} \alpha(\mathbf{s}'\sqcup \mathbf{s})=0. 
\end{equation*}
In fact, let us show that, for each $\mathbf{s}'\in J$, one has $\sum_{\mathbf{s}\in Y_e} \alpha(\mathbf{s}'\sqcup \mathbf{s})=0$.

Let $\mathbb{T}_1$ and $\mathbb{T}_2$ be the two triangles adjoining $e$ (they are distinct since $\Delta$ is assumed non-degenerate) and let $Q\subset  \Sigma_{\mathcal{P}}$ be the resulting square. 
Denote by $i_Q: \mathcal{S}_{\omega}(Q)\hookrightarrow \otimes_{\mathbb{T}\in F(\Delta)} \mathcal{S}_{\omega}(\mathbb{T})$ the corresponding embedding and write $T_Q:= T\cap (Q\times (0,1))$ . 
For each $\mathbf{s}'\in J$, one has 
\begin{multline*}
\sum_{\mathbf{s}\in Y_{e}} \alpha(\mathbf{s}'\sqcup \mathbf{s}) \\ = \left( \otimes_{\mathbb{T}\neq \mathbb{T}_1, \mathbb{T}_2} [T^{(N)}_{\mathbb{T}}, \mathbf{s}'_{|\partial \mathbb{T}} ]\right) 
\otimes \left( i_Q([T_Q^{(N)}, \mathbf{s}'_{|\partial Q} ]) - (j_{\mathbb{T}_1} \otimes j_{\mathbb{T}_2}) \circ i_Q ([T_Q^{(N)}, \mathbf{s}'_{|Q}])\right).
\end{multline*}
 The last term is zero by Lemma \ref{lemma_square} and the commutativity of the diagrams in Lemma \ref{lemma_center_triangle}. This concludes  the proof. 
\end{proof}

\paragraph{Step 3: closed curves that intersect $\Delta$ nicely.}

\begin{definition}  The $N$-th Chebyshev polynomial of first kind is the polynomial  $T_N(X) \in \mathbb{Z}[X]$ defined by the recursive formulas $T_0(X)=2$, $T_1(X)=X$ and $T_{n+2}(X)=XT_{n+1}(X) -T_n(X)$ for $n\geq 0$.
\end{definition}

\par  The following proposition is at the heart of (our proof of) the so-called "miraculous cancelations" from \cite{BonahonWong1}. We postpone its proof to the Appendix A.

\begin{proposition}\label{proptchebychev}
	If $\omega$ is a root of unity of odd order $N\geq 1$, then in  $\mathcal{S}_{\omega}(\mathbb{B})$, the following equality holds:
	$$ T_N(\alpha_{++}+\alpha_{--}) = \alpha_{++}^N + \alpha_{--}^N.$$
\end{proposition}

Recall that we suppose that the triangulation is non-degenerate.

\begin{lemma}\label{lemma_curve}
	Let  $\gamma \in \mathcal{S}_{\omega}(\mathbf{\Sigma})$ be the class of a closed curve. 
	If the closed curve can be chosen such that it intersects an edge of $\Delta$ once and only once, then 	$j_{\mathbf{\Sigma}} (\gamma) = T_N(\gamma)$.
\end{lemma}

\begin{proof}
	Consider the punctured surface $\mathbf{\Sigma}(e)$ obtained from $\mathbf{\Sigma}$ by replacing $e$ by two arcs $e'$ and $e''$ parallel to $e$ with the same endpoints and removing the bigone between $e'$ and $e''$. Consider the injective morphism $i_{|e'\#e''} : \mathcal{S}_{\omega}(\mathbf{\Sigma}) \hookrightarrow \mathcal{S}_{\omega}(\mathbf{\Sigma}(e))$. By Lemma \ref{lem: j commute with i}, the following diagram commutes: 
	$$\begin{tikzcd}
	\mathcal{S}_{+1}(\mathbf{\Sigma}) \arrow[r, hook, "j_{\mathbf{\Sigma}}"] \arrow[d, hook, "i_{|e'\#e''}"] & \mathcal{S}_{\omega}(\mathbf{\Sigma}) \arrow[d, hook, "i_{|e'\#e''}"] \\
	\mathcal{S}_{+1}(\mathbf{\Sigma}(e)) \arrow[r, hook, "j_{\mathbf{\Sigma}(e)}"] & \mathcal{S}_{\omega}(\mathbf{\Sigma}(e)) 
	\end{tikzcd} $$
	By cutting $\gamma$ along $e$, we get an arc $\beta\subset \Sigma(e)$ such that, by the hypothesis,  $i_{e'\#e''}(\gamma)=\beta_{++}+\beta_{--}$. 
	Consider the algebra morphism $\varphi : \mathcal{S}_{\omega}(\mathbb{B}) \rightarrow   \mathcal{S}_{\omega}(\mathbf{\Sigma}(e))$ sending $\alpha_{\varepsilon \varepsilon'}$ to $\beta_{\varepsilon \varepsilon'}$. One has: 
	\begin{align*}
	j_{\mathbf{\Sigma}(e)} \circ i_{|e'\#e''} (\gamma) &= j_{\mathbf{\Sigma}(e)} (\beta_{++} + \beta_{--})& \\
	&= \varphi(\alpha_{++}^N + \alpha_{--}^N) &\mbox{ by Lemma \ref{lemma_baleze}} \\
	&= \varphi(T_N(\alpha_{++} + \alpha_{--})) &\mbox{ by Proposition \ref{proptchebychev}} \\
	&= i_{|e'\#e''} (T_N(\gamma)). &
	\end{align*}
	Hence by the above diagram one has $j_{\mathbf{\Sigma}}(\gamma) = T_N(\gamma)$.
\end{proof}

\paragraph{Step 4: adding a puncture.}

Let $\mathbf{\Sigma}'=(\Sigma, \mathcal{P}\cup \{p_0\})$ be a punctured surface obtained from $\mathbf{\Sigma}=(\Sigma, \mathcal{P})$ by adding one puncture $p_0\in \Sigma_{\mathcal{P}}$ and consider the algebra morphism $\varphi : \mathcal{S}_{\omega}(\mathbf{\Sigma}') \rightarrow \mathcal{S}_{\omega}(\mathbf{\Sigma})$ of Section \ref{sec_removingpuncture}. We assume that $\mathbf{\Sigma}$ is equipped with a non-degenerated triangulation.

\begin{lemma}\label{lemma_add_puncture} The following diagram is commutative.  
	$$ \begin{tikzcd} 
	\mathcal{S}_{+1}(\mathbf{\Sigma}') \arrow[r, hook, "j_{\mathbf{\Sigma}'}"] \arrow[d, two heads, "\varphi"] &
	\mathcal{S}_{\omega}(\mathbf{\Sigma}') \arrow[d, two heads, "\varphi"] \\
	\mathcal{S}_{+1}(\mathbf{\Sigma}) \arrow[r, hook, "j_{\mathbf{\Sigma}}"] &
	\mathcal{S}_{\omega}(\mathbf{\Sigma}) 
	\end{tikzcd}$$
\end{lemma}

\begin{proof}
First consider the following diagram 

\begin{equation}\label{diagram_truc}
\begin{tikzcd}
0 \arrow[r] & \mathcal{I}^{+1}_{p_0} \arrow[r] \arrow[d, hook, "j_{\mathbf{\Sigma}'}"] & \mathcal{S}_{+1}(\mathbf{\Sigma}') \arrow[r, "\varphi"] \arrow[d, hook, "j_{\mathbf{\Sigma}'}"] & 
\mathcal{S}_{+1}(\mathbf{\Sigma}) \arrow[r] \arrow[d, hook, " j_{\mathbf{\Sigma}}"] & 0 \\
0 \arrow[r] & \mathcal{I}_{p_0} \arrow[r] & \mathcal{S}_{\omega}(\mathbf{\Sigma}') \arrow[r, "\varphi"] & \mathcal{S}_{\omega}(\mathbf{\Sigma})  \arrow[r] & 0
\end{tikzcd}
\end{equation}
 where $\mathcal{I}^{+1}_{p_0} \subset \mathcal{S}_{+1}(\mathbf{\Sigma}')$ and $\mathcal{I}_{p_0} \subset \mathcal{S}_{\omega}(\mathbf{\Sigma}')$ denote the off-puncture ideals in $ \mathcal{S}_{+1}(\mathbf{\Sigma}')$ and $ \mathcal{S}_{\omega}(\mathbf{\Sigma}')$ respectively (see Definition \ref{def_off_puncture}). 
By Proposition \ref{prop_off_puncture}, both lines are exact so we need to  prove the inclusion  $j_{\mathbf{\Sigma}'} (\mathcal{I}^{+1}_{p_0}) \subset \mathcal{I}_{p_0}$ to conclude. We divide the proof in two steps.

\vspace{2mm}
\par \textbf{Step 1:} We first suppose that $\mathbf{\Sigma}=\mathbb{T}_0$ is a triangle. In this case, $\mathbb{T}'_0$ is a punctured triangle and we have two possibilities depending whether $p_0$ is in the boundary or the interior of $\mathbb{T}_0$. Some non-degenerate triangulations $\Delta_0'$ of $\mathbb{T}_0'$ are drawn in Figure \ref{fig_puncturedTriangle}. 

\begin{figure}[!h] 
	\centerline{\includegraphics[width=6cm]{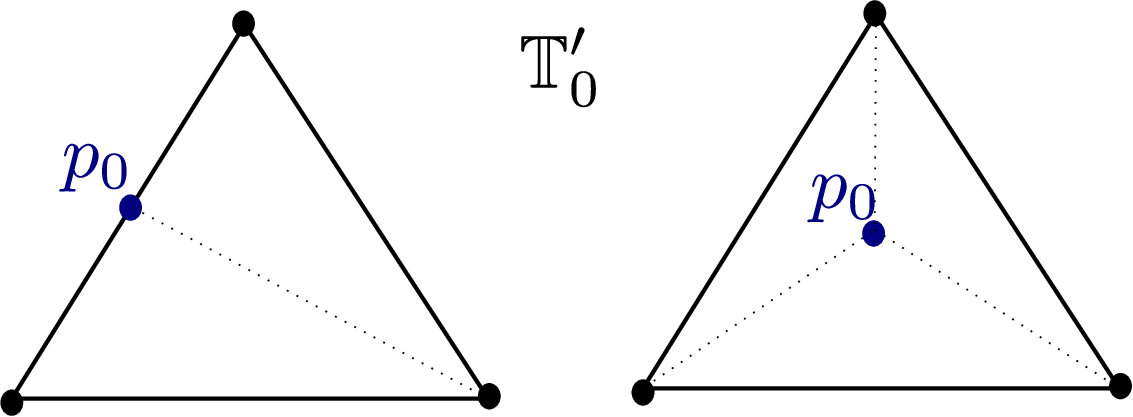} }
	\caption{Punctured triangles $\mathbb{T}'_0$ and their non-degenerated triangulations.  } 
	\label{fig_puncturedTriangle} 
\end{figure}

\par \textbf{Claim:} the off-kernel ideal $\mathcal{I}_{p_0}$ is generated by elements $\alpha_{\varepsilon \varepsilon'} - \alpha'_{\varepsilon \varepsilon'}$ and $\gamma- \gamma'$, where $\alpha$ and $\alpha'$ are arcs isotopic in $\mathbb{T}_0$ whose endpoints lye in distinct boundary arcs and $\gamma$, $\gamma'$ are curves isotopic in $\mathbb{T}_0$ which intersect each edge of $\Delta_0'$ once.

 \par If the claim is proved, then for $\alpha_{\varepsilon \varepsilon'} - \alpha'_{\varepsilon \varepsilon'}$ and $\gamma- \gamma'$ some generators of $\mathcal{I}_{p_0}$, then Lemma \ref{lemma_baleze} implies that $j_{\mathbb{T}_0'} \left( \alpha_{\varepsilon \varepsilon'} - \alpha'_{\varepsilon \varepsilon'}\right) \subset \mathcal{I}_{p_0}$ and Lemma \ref{lemma_curve} implies that $j_{\mathbb{T}_0'}(\gamma-\gamma') = T_N(\gamma) - T_N(\gamma') \in \mathcal{I}_{p_0}$. The the claim implies the inclusion $j_{\mathbf{\mathbb{T}_0}'} (\mathcal{I}^{+1}_{p_0}) \subset \mathcal{I}_{p_0}$ which concludes concludes the proof in the case of the triangle. To prove the claim, recall from Proposition \ref{prop_off_puncture} that $\mathcal{I}_{p_0}$ is generated by elements $\alpha_{\varepsilon \varepsilon'} - \alpha'_{\varepsilon \varepsilon'}$ and $\gamma- \gamma'$ with $\alpha$ and $\alpha'$ isotopic in $\mathbb{T}_0$ and $\gamma, \gamma'$ isotopic in $\mathbb{T}_0$.
 First note that when $p_0$ lies in the boundary of $\mathbb{T}_0$, then $\mathbb{T}_0'$ does not contain any non contractible simple closed curve and the non-trivial arcs of $\mathbb{T}_0'$ have endpoints in distinct boundary arcs, so the claim is immediate in this case. When $p_0$ lies in the interior of $\mathbb{T}_0$, there is only one non-trivial simple closed curve (which encircles $p_0$ once) and this curves intersects each edges of $\Delta_0'$ once. However $\mathbb{T}_0'$ contains $3$ non-trivial arcs with endpoints in the same boundary arcs which are related by a $\frac{2}{3}\pi$-radiant rotation. Let $\delta$ be one of these arcs and $\delta_{\varepsilon \varepsilon'}=  \adjustbox{valign=c}{\includegraphics[width=1.5cm]{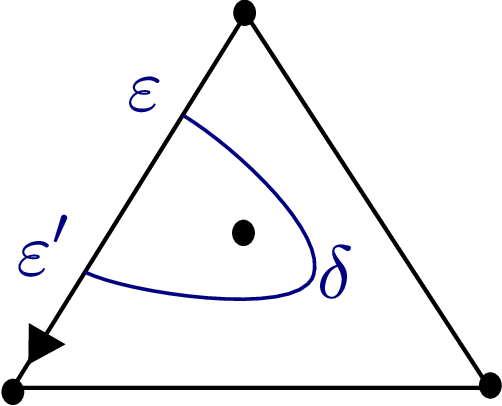}} $ . Since $x:=\delta_{\varepsilon \varepsilon'} - C^{\varepsilon'}_{\varepsilon} \in \mathcal{I}_{p_0}$, we need to show that $x$ belongs to the ideal $\mathcal{J}_{p_0}$ generated by elements  $\alpha_{\varepsilon \varepsilon'} - \alpha'_{\varepsilon \varepsilon'}$ with $\alpha, \alpha'$ isotopic in $\mathbb{T}_0$ with distinct endpoints. This is done by a simple application of the skein relation \eqref{eq: skein 2} as follows:
 $$x= \adjustbox{valign=c}{\includegraphics[width=1.5cm]{PTriangle1.eps}} - \adjustbox{valign=c}{\includegraphics[width=1.5cm]{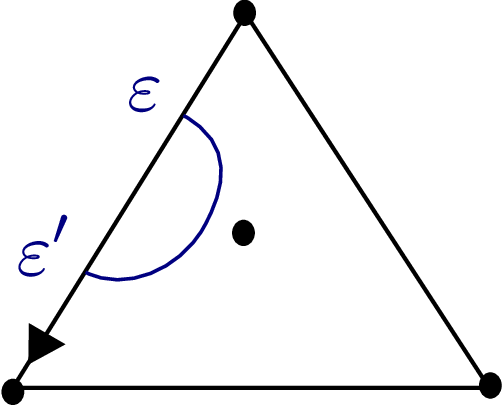}} = \sum_{\mu = +, -} C_{\mu}^{-\mu} \left( \adjustbox{valign=c}{\includegraphics[width=1.5cm]{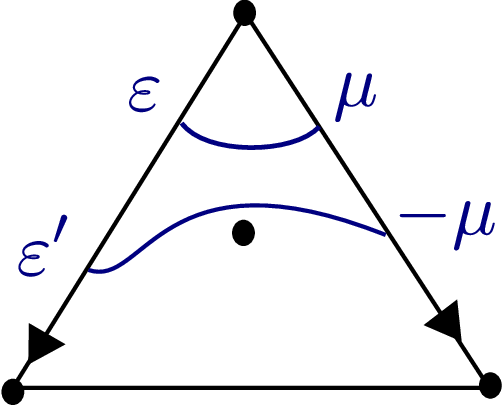}}- \adjustbox{valign=c}{\includegraphics[width=1.5cm]{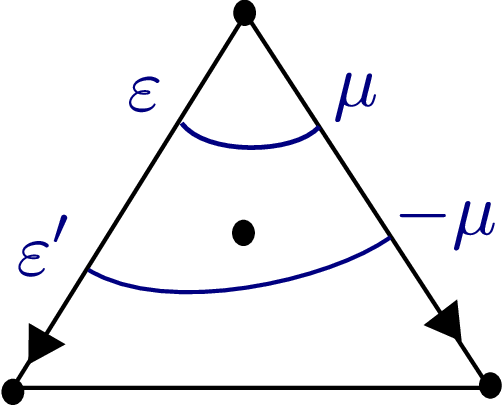}} \right).$$
 Therefore $x$ belongs to the ideal generated by elements $\adjustbox{valign=c}{\includegraphics[width=1.5cm]{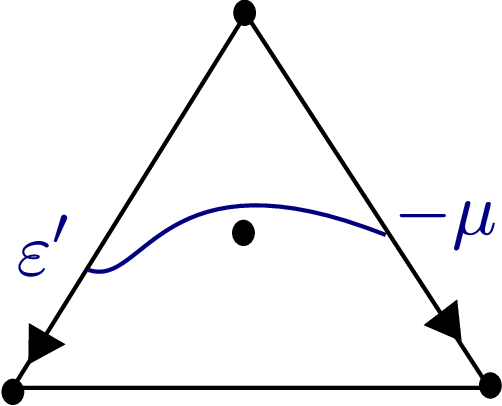}} - \adjustbox{valign=c}{\includegraphics[width=1.5cm]{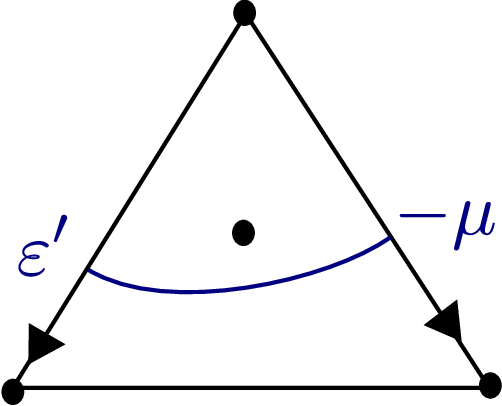}}$.
 This proves the claim and concludes the proof of the lemma in the case where $\mathbf{\Sigma}=\mathbb{T}_0$.
 \vspace{2mm}
\par \textbf{Step 2:} We consider the general case. Recall that $\mathbf{\Sigma}$ is equipped with a non-degenerate triangulation $\Delta$ and let $\mathbb{T}_0$ be the face containing the point $p_0$. Let $\mathbf{\Sigma}_0$ be the (possibly empty) punctured surface made of the faces of $\Delta$ distinct from $\mathbb{T}_0$ so that $\mathbf{\Sigma}$ is obtained from $\mathbb{T}_0 \bigsqcup \mathbf{\Sigma}_0$ by gluing some pairs of boundary arcs together and let $i : \mathcal{S}_{\omega}(\mathbf{\Sigma}) \hookrightarrow \mathcal{S}_{\omega}(\mathbb{T}_0) \otimes \mathcal{S}_{\omega}(\mathbf{\Sigma}_0)$ denote the gluing map. Similarly, let $i': \mathcal{S}_{\omega}(\mathbf{\Sigma}') \hookrightarrow \mathcal{S}_{\omega}(\mathbb{T}_0') \otimes \mathcal{S}_{\omega}(\mathbf{\Sigma}_0)$ be the gluing map of $\mathbf{\Sigma}'$. Consider the following diagram

$$
 \begin{tikzcd}
 \mathcal{S}_{+1}(\mathbb{T}_0') \otimes \mathcal{S}_{+1}(\mathbf{\Sigma}_0)
  \arrow[rrr, hook, "j_{\mathbb{T}_0'}\otimes j_{\mathbf{\Sigma}_0}" ]
   \arrow[ddd, twoheadrightarrow, "\varphi_0 \otimes \id" ] &{}&{}& 
 \mathcal{S}_{\omega}(\mathbb{T}_0') \otimes \mathcal{S}_{\omega}(\mathbf{\Sigma}_0) 
 \arrow[ddd, twoheadrightarrow, "\varphi_0 \otimes \id"] \\
 {} & 
 \mathcal{S}_{+1}(\mathbf{\Sigma}')
  \arrow[lu, hook, "i'"] \arrow[r, hook, "j_{\mathbf{\Sigma}'}"] \arrow[d, twoheadrightarrow, "\varphi"] & 
  \mathcal{S}_{\omega}(\mathbf{\Sigma}')
  \arrow[ru, hook, "i'"] \arrow[d, twoheadrightarrow, "\varphi"] &{} \\
 {} & 
  \mathcal{S}_{\omega}(\mathbf{\Sigma}')
  \arrow[ld, hook, "i"] \arrow[r, hook, "j_{\mathbf{\Sigma}}"] & 
  \mathcal{S}_{\omega}(\mathbf{\Sigma})
  \arrow[rd, hook, "i"] & {} \\
 \mathcal{S}_{+1}(\mathbb{T}_0) \otimes \mathcal{S}_{+1}(\mathbf{\Sigma}_0)
  \arrow[rrr, hook,  "j_{\mathbb{T}_0}\otimes j_{\mathbf{\Sigma}_0}"] &{}&{}&
 \mathcal{S}_{\omega}(\mathbb{T}_0) \otimes \mathcal{S}_{\omega}(\mathbf{\Sigma}_0)
\end{tikzcd}
$$

In this diagram: 
\begin{itemize}
\item the outer square commutes by Step $1$; 
\item the  squares on the top and bottom  commutes by Lemma \ref{lem: j commute with i}; 
\item the squares on the left and right sides commutes by definition of $\varphi$.
\end{itemize}
Therefore the innermost square commutes; this concludes the proof.
	
\end{proof}

\begin{notations}
For $\alpha_{\varepsilon \varepsilon'} \in \mathcal{S}_{\omega}(\mathbf{\Sigma})$ the class of a stated arc, we denote by  $\alpha_{\varepsilon \varepsilon'}^{(N)}$ be the class of the stated tangle made by stacking $N$ parallel copies of $\alpha_{\varepsilon \varepsilon'}$ on top of the others in the framing direction. More precisely,  if both endpoints of $\alpha$ lye in different boundary arcs, then $\alpha_{\varepsilon \varepsilon'}^{(N)}=(\alpha_{\varepsilon \varepsilon'})^{N}$. If $\alpha$ has its two endpoints, say $v$ and $w$, in the same boundary arc with $h(v)<h(w)$ such that $v$ has state $\varepsilon$ and $w$ has state $\varepsilon'$, then $ \alpha_{\varepsilon \varepsilon'}^{(N)}$ is the class of the stated tangle $(\alpha^{(N)}, s^{(N)})$ defined as follows. The tangle $\alpha^{(N)}$ is made of $N$ parallel copies $\alpha^{(N)}= \alpha_1 \cup \ldots \cup \alpha_N$ of $\alpha$ such that the height order is given by $h(v_1)<h(v_2)< \ldots < h(v_N)<h(w_1) < \ldots < h(w_N)$. The state $s^{(N)}$ sends the points $v_i$ to $\varepsilon$ and the points $w_j$ to $\varepsilon'$.
\end{notations}

\begin{lemma}\label{lemma_T1}
If $\alpha_{\varepsilon \varepsilon'}\in \mathcal{S}_{\omega}(\mathbf{\Sigma})$ is  the class of a stated arc such that its  endpoints lie on the same boundary arcs, then $j_{\mathbf{\Sigma}} (\alpha_{\varepsilon \varepsilon'}) = \alpha_{\varepsilon \varepsilon'}^{(N)}$. 
\end{lemma}
\begin{proof}
	Since the two endpoints of $\alpha$ lie on the same boundary arc $a$, we can pick a puncture $p_0\in a $  that 
	lies between these two endpoints. 
	Denote by $\mathbf{\Sigma}'=(\Sigma, \mathcal{P}\cup \{p_0\})$ the punctured surface obtained by adding this puncture and $\varphi : \mathcal{S}_{\omega}(\mathbf{\Sigma}') \rightarrow \mathcal{S}_{\omega}(\mathbf{\Sigma})$ the morphism of Section $2.3$. With the notations of Section $2.3$, the two components of $a\setminus \{p_0\}$ are two boundary arcs $b$ and $c$ of $\mathbf{\Sigma}'$ and we choose the convention such that $\alpha \in \mathcal{T}^{(0)}(\mathbf{\Sigma})$. Note that $\alpha^{(N)}$ is in $\mathcal{T}^{(0)}(\mathbf{\Sigma})$ as well. 
	To avoid confusion, we denote by $\alpha'$ the arc $\alpha$ seen as an arc in $\Sigma_{\mathcal{P}\cup \{p_0\}}$, so that $\iota (\alpha')=\alpha$.
	By Lemma \ref{lemma_baleze}, one has $j_{\mathbf{\Sigma}'} (\alpha'_{\varepsilon \varepsilon'})= (\alpha'_{\varepsilon \varepsilon'})^{N}= {\alpha'}_{\varepsilon \varepsilon'}^{(N)}$. By commutativity of the Diagram in Lemma \ref{lemma_add_puncture} and by definition of $\varphi$, the image $j_{\mathbf{\Sigma}}(\alpha_{\varepsilon \varepsilon'})$ is the class in $\mathcal{S}_{\omega}(\mathbf{\Sigma})$ of the unique stated tangle in $\mathcal{T}^{(0)}(\mathbf{\Sigma})$ which is isotopic to ${\alpha'}_{\varepsilon \varepsilon'}^{(N)}$: this is 
	$\alpha_{\varepsilon \varepsilon'}^{(N)}$.
	\end{proof}

\begin{lemma}\label{lemma_T2}
If  $\gamma \in \mathcal{S}_{\omega}(\mathbf{\Sigma})$ is the class of a closed curve, then $j_{\mathbf{\Sigma}} (\gamma) = T_N(\gamma)$. 
\end{lemma}

\begin{proof}
	If the closed curve can be chosen such that it intersects an edge of $\Delta$ once and only once, then this is Lemma \ref{lemma_curve}. 
	Otherwise, we can refine the triangulation by adding an inner puncture in order to have this property. Denote by $\mathbf{\Sigma}'$ the resulting punctured surface and let  $\gamma'\in \mathcal{S}_{+1}(\mathbf{\Sigma}')$ be such that  $\iota(\gamma')=\gamma$. 
	Lemma \ref{lemma_curve} implies that $j_{\mathbf{\Sigma}'}(\gamma')=T_N(\gamma')$ and Lemma \ref{lemma_add_puncture} implies that $j_{\mathbf{\Sigma}}(\gamma)=T_N(\gamma)$. 
\end{proof}

\paragraph{Step 5: punctured surfaces which do not admit non-degenerate triangulations.}

It remains to prove Theorem $1.2$ for connected punctured surfaces which do not admit non-degenerate topological triangulations, that is for the small punctured surfaces, for the disc with one inner puncture  and one puncture on its boundary and for the unpunctured surfaces $\mathbf{\Sigma}=(\Sigma, \emptyset)$ with empty set of puncture.

The disc with only one puncture (on its boundary) and the sphere with zero or one puncture have both trivial skein algebra, while the sphere with two punctures has a commutative skein algebra. Therefore,  Theorem \ref{theorem2} holds trivially for them.  It remains to prove the 

\begin{lemma}\label{lemma_degenerate}
Theorem \ref{theorem2} holds when $\mathbf{\Sigma}$ is either a disc with one inner puncture  and one puncture on its boundary or an unpunctured surface $\mathbf{\Sigma}=(\Sigma, \emptyset)$ of genus at least one.
\end{lemma}

\begin{proof}
Choose an inner puncture $p_0 \in \mathring{\Sigma}_{\mathcal{P}}$ and consider the punctured surface $\mathbf{\Sigma}':=(\Sigma, \mathcal{P}\cup \{p_0\})$. Since $\mathbf{\Sigma}'$ admits a non-degenerate triangulation, our previous study shows the existence of the Chebyshev morphism $j_{\mathbf{\Sigma}'} : \mathcal{S}_{+1}(\mathbf{\Sigma}') \hookrightarrow \mathcal{Z} \left(  \mathcal{S}_{\omega}(\mathbf{\Sigma}') \right)$.
Consider the off-puncture ideals $\mathcal{I}^{+1}_{p_0} \subset \mathcal{S}_{+1}(\mathbf{\Sigma}')$ and $\mathcal{I}_{p_0} \subset \mathcal{S}_{\omega}(\mathbf{\Sigma}')$. Exactly the same argument used in the proof of Lemma \ref{lemma_add_puncture}, shows the inclusion $j_{\mathbf{\Sigma}'} (\mathcal{I}^{+1}_{p_0}) \subset \mathcal{I}_{p_0}$. 
By Proposition \ref{prop_off_puncture}, both lines in the following diagram are exact, 
$$
\begin{tikzcd}
0 \arrow[r] & \mathcal{I}^{+1}_{p_0} \arrow[r] \arrow[d, hook, "j_{\mathbf{\Sigma}'}"] & \mathcal{S}_{+1}(\mathbf{\Sigma}') \arrow[r, "\varphi"] \arrow[d, hook, "j_{\mathbf{\Sigma}'}"] & 
\mathcal{S}_{+1}(\mathbf{\Sigma}) \arrow[r] \arrow[d, dotted, "\exists! j_{\mathbf{\Sigma}}"] & 0 \\
0 \arrow[r] & \mathcal{I}_{p_0} \arrow[r] & \mathcal{S}_{\omega}(\mathbf{\Sigma}') \arrow[r, "\varphi"] & \mathcal{S}_{\omega}(\mathbf{\Sigma})  \arrow[r] & 0
\end{tikzcd}
$$

, therefore there exists a unique algebra morphism $j_{\mathbf{\Sigma}} : \mathcal{S}_{+1}(\mathbf{\Sigma}) \rightarrow \mathcal{S}_{\omega}(\mathbf{\Sigma})$ which makes the diagram commuting. Since $j_{\mathbf{\Sigma}}$ is obtained from $j_{\mathbf{\Sigma}'}$ by passing to the quotient, its image is also central and one has the equalities $j_{\mathbf{\Sigma}}([\gamma])=T_N([\gamma])$ and $j_{\mathbf{\Sigma}}(\alpha_{\varepsilon \varepsilon'}) = \alpha^{(N)}_{\varepsilon \varepsilon'}$ for any closed curve $\gamma$ and any stated arc $\alpha_{\varepsilon \varepsilon'}$.
This concludes the proof.

\end{proof}

 \subsection{A Poisson bracket on $\mathcal{S}_{+1}(\mathbf{\Sigma})$}\label{sec_Poisson_skein}
 
 \par In this section, we define and make explicit a Poisson structure on $\mathcal{S}_{+1}(\mathbf{\Sigma})$.
 \vspace{2mm}
 
 \subsubsection{Preliminaries}
 
  We briefly recall some general facts concerning deformation quantization. 
  
  Let $\mathcal{A}$ be a complex commutative unital algebra, $\mathbb{C}[[\hbar]]$ be the ring of formal series in a parameter $\hbar$ and $\mathcal{A}[[\hbar]]:=\mathcal{A}\otimes_{\mathbb{C}}\mathbb{C}[[\hbar]]$. A \textit{star product} $\star$ on $\mathcal{A}$ is an associative product on $\mathcal{A}[[\hbar]]$ such that if $f=\sum_i f_i \hbar^i $ and $g=\sum_i g_i \hbar^i$ are elements of $\mathcal{A}[[h]]$, then:
  $$ f\star g = f_0g_0 \mod{\hbar},$$
 where $f_0g_0$ denotes the product of $f_0$ and $g_0$ in $\mathcal{A}$. 
  A star product induces a Poisson structure on $\mathcal{A}$ by the formula:
  \begin{equation}\label{eq_poisson} f\star g -g\star f = \hbar \{ f, g \} \mod{\hbar^2}, \end{equation}
  for all $f,g\in \mathcal{A}$.
  The algebra $(\mathcal{A}[[\hbar]], \star)$ is called a \textit{deformation quantization} of the commutative Poisson algebra $(\mathcal{A}, \{\cdot, \cdot \})$. We refer to  (\cite{KontsevichQuantizationPoisson}, \cite{GRS_QuantizationDeformation}  $II.2$)  for detailed discussions. A \textit{morphism of star products} between $(\mathcal{A}, \star_{\mathcal{A}})$ and $(\mathcal{B}, \star_{\mathcal{B}})$ is an algebra morphism $\psi : \mathcal{A}[[\hbar]] \rightarrow \mathcal{B}[[\hbar]]$ whose restriction to $\mathcal{A}\subset \mathcal{A}[[\hbar]]$ induces a morphism $\phi : \mathcal{A} \rightarrow \mathcal{B}$. 
Note that such a $\phi$ is, in fact, a morphism of Poisson algebras for the induced Poisson algebra structures.
   An isomorphism $\psi :  (\mathcal{A}[[\hbar]], \star_1) \xrightarrow{\cong} (\mathcal{A}[[\hbar]], \star_2)$ of star products is called a \textit{gauge equivalence} if $\psi(f)= f \pmod{\hbar}$. If two star products are gauge equivalent, they induce the same Poisson bracket on $\mathcal{A}$.
  
   To end this preamble, let us mention that deformation quantization is well-behaved relatively to the tensor product. 
  	Indeed, if $\mathcal{A}[[h]]$ and $\mathcal{B}[[h]]$ are deformation quantizations of $\mathcal{A}$ and $\mathcal{B}$ respectively, then $\mathcal{A}[[h]]\otimes \mathcal{B}[[h]]\cong (\mathcal{A}\otimes \mathcal{B})[[h]]$ is a deformation quantization of $\mathcal{A}\otimes \mathcal{B}$. Note also that the  Poisson structure on $\mathcal{A}\otimes \mathcal{B}$ given by \eqref{eq_poisson} is 
  	\begin{equation}\label{eq: tens poiss quant}
  	\ \{ f\otimes g, f'\otimes g'\} =   ff' \otimes \{g,g'\} + \{f,f'\}\otimes gg', 
  	\end{equation}
  for $f,f'\in \mathcal{A}$ and $g,g' \in \mathcal{B}$.

    \subsubsection{Formal definition}

  Let $\mathbf{\Sigma}$ be a punctured surface and $\mathfrak{o}$ an orientation of its boundary arc. Denote by $\mathcal{S}_{+1}(\mathbf{\Sigma})$ the stated skein algebra associated to the ring $\mathbb{C}$ with $\omega=+1$ and denote by $\mathcal{S}_{\omega_{\hbar}}(\mathbf{\Sigma})$ the stated skein algebra associated to the ring $\mathbb{C}[[\hbar]]$ with $\omega_{\hbar}:= \exp\left( -\hbar/4 \right)$. The convention is chosen so that $q=\exp(\hbar)$. 
Recall the basis $\mathcal{B}^{\mathfrak{o}}$ from Definition \ref{def_basis}. Since $\mathcal{B}^{\mathfrak{o}}$ is independent of $\omega$, one has an isomorphism of  $\mathbb{C}[[\hbar]]$-modules 
	\begin{equation}\label{eq: iso psi Bo}
	\psi^{\mathfrak{o}} : \mathcal{S}_{+1}(\mathbf{\Sigma})[[\hbar]] \xrightarrow{\cong}  \mathcal{S}_{\omega_{\hbar}}(\mathbf{\Sigma}).
	\end{equation}
Note that $\mathfrak{o}$ tells us how to lift the basis elements $[D,s]$ of $\mathcal{S}_{+1}(\mathbf{\Sigma})$ (which are independent on the height order) in $\mathcal{S}_{\omega_{\hbar}}(\mathbf{\Sigma})$. We emphasize that $\psi^{\mathfrak{o}}$ is not an algebra morphism.

  \begin{definition}\label{de:poiss brack}
   Pulling-back the product of $\mathcal{S}_{\omega_{\hbar}}(\mathbf{\Sigma})$ along $\psi^{\mathfrak{o}}$ gives a star product $\star_{\hbar}$ on  $\mathcal{S}_{+1}(\mathbf{\Sigma})$. We denote by   $\{\cdot, \cdot\}^s$ the resulting Poisson bracket on $\mathcal{S}_{+1}(\mathbf{\Sigma})$  given by Equation \eqref{eq_poisson}.
 \end{definition}
 
 Here the superscript $s$ stands for "skein".
 \begin{remark}\label{remark_independance} Note that for any two orientations $\mathfrak{o}_1$ and $\mathfrak{o}_2$ of the boundary arcs of $\mathbf{\Sigma}$, the automorphism $(\psi^{\mathfrak{o}_2})^{-1}\circ \psi^{\mathfrak{o}_1} : \mathcal{S}_{+1}(\mathbf{\Sigma})[[\hbar]] \xrightarrow{\cong} \mathcal{S}_{+1}(\mathbf{\Sigma})[[\hbar]]$ is a gauge equivalence, hence the Poisson bracket $\{\cdot, \cdot \}^s$ does not depend on $\mathfrak{o}$. 
 \end{remark}
 By definition, $(\mathcal{S}_{+1}(\mathbf{\Sigma})[[\hbar]], \star_{\hbar})$ is a quantization deformation of the Poisson algebra $(\mathcal{S}_{+1}(\mathbf{\Sigma}), \{\cdot, \cdot\}^s)$. Moreover, this structure of Poisson algebra  is compatible with decompositions of surfaces. More precisely, one has the following. 
 \begin{lemma}\label{lemma: poisson morph at +1 skein}
  The gluing maps $i_{|a\#b} : \mathcal{S}_{+1}(\mathbf{\Sigma}_{|a\#b})\hookrightarrow \mathcal{S}_{+1}(\mathbf{\Sigma})$, the maps $i^{\Delta}: \mathcal{S}_{+1}(\mathbf{\Sigma}) \hookrightarrow \otimes_{\mathbb{T}\in F(\Delta)}\mathcal{S}_{+1}(\mathbb{T})$ and the coproduct maps $\Delta^L, \Delta^R$  are Poisson morphisms.
  \end{lemma}
  
  \begin{proof}
   This results from the fact that each of these morphisms arises from a morphism of  star products.  
  \end{proof}

  \subsubsection{Explicit formula}\label{sec: explicit formula of bracket}

  This section is devoted to making explicit the Poisson bracket $\{\cdot, \cdot\}^s$ on stated diagrams. 
 It will be expressed in terms of \emph{resolutions} of stated diagrams, which are defined at crossings and at points on the boundary arcs.  
  
  All along this section, $\mathbf{\Sigma}$ is a punctured surface. 
  
  \paragraph{Resolution at a crossing.} 
  Let $(D,s)$ be a stated diagram and $c$ a crossing of $D$.  
 Denote by $D_+$ and $D_-$ the diagrams obtained from $D$ by replacing the crossing $c$ by its positive and negative resolution respectively:  
  	\begin{equation*}
 \text{ the crossing $c$  } 
 \begin{tikzpicture}[baseline=-0.4ex,scale=0.5,>=stealth]	
  \draw [fill=gray!45,gray!45] (-.6,-.6)  rectangle (.6,.6)   ;
  \draw[line width=1.2,-] (-0.4,-0.52) -- (.4,.53);
  \draw[line width=1.2,-] (0.4,-0.52) -- (0.1,-0.12);
  \draw[line width=1.2,-] (-0.1,0.12) -- (-.4,.53);
  \end{tikzpicture}
\text{ and its positive }
  \begin{tikzpicture}[baseline=-0.4ex,scale=0.5,>=stealth] 
  \draw [fill=gray!45,gray!45] (-.6,-.6)  rectangle (.6,.6)   ;
  \draw[line width=1.2] (-0.4,-0.52) ..controls +(.3,.5).. (-.4,.53);
  \draw[line width=1.2] (0.4,-0.52) ..controls +(-.3,.5).. (.4,.53);
  \end{tikzpicture}
\text{ and negative }
  \begin{tikzpicture}[baseline=-0.4ex,scale=0.5,rotate=90]	
  \draw [fill=gray!45,gray!45] (-.6,-.6)  rectangle (.6,.6)   ;
  \draw[line width=1.2] (-0.4,-0.52) ..controls +(.3,.5).. (-.4,.53);
  \draw[line width=1.2] (0.4,-0.52) ..controls +(-.3,.5).. (.4,.53);
  \end{tikzpicture}
  \text{ resolution.}
  \end{equation*}
    The resolution of $(D,s)$ at the crossing $c$ is defined by 
  \begin{equation*}
  \Res_c(D,s) := \left[D_+ , s\right] - \left[D_-, s\right] \in \mathcal{S}_{+1}(\mathbf{\Sigma}).
  \end{equation*}

  \paragraph{Resolution at boundary points.}

  Let $b_1,...,b_k$ be the boundary arcs of $\Sigma_{\mathcal{P}}$. 
  
  \begin{definition}
  	A \emph{height order} on a stated diagram $(D,s)$ of $\Sigma_{\mathcal{P}}$ is a $k$--tuple $\mathfrak{o}=(\mathfrak{o}_1,...,\mathfrak{o}_k)$ of bijections of sets $\mathfrak{o}_i: \partial_{b_i} D \to \{1, ..., |\partial_{b_i}D| \}$.   
  \end{definition} 
Note that the product of symmetric groups $\mathbb{S}_{n_1} \times \ldots \times \mathbb{S}_{n_k}$ acts freely and transitively on the set of height orders by left composition.

	To a height order $\mathfrak{o}$ on $(D,s)$ corresponds a stated tangle with same height order and which projects to $(D,s)$. Therefore, one can consider the class of  $(D,s,\mathfrak{o})$ in  $\mathcal{S}_{\omega}(\mathbf{\Sigma})$.   
  	If $\omega=+1$, the class $[D,s, \mathfrak{o}]\in \mathcal{S}_{+1}(\mathbf{\Sigma})$ is independent of $\mathfrak{o}$, and we denote it simply by $[D,s]$.
\\

 Let us choose a boundary arc $b_i$ and suppose there are two points $p_H$ and $p_L$ of $\partial_{b_i}D$ such that $\mathfrak{o}_i(p_H)=\mathfrak{o}_i(p_L)+1$ (\textit{i.e.} $p_H$ is the $\mathfrak{o}_i$--successor of $p_L$). 
Let $\widetilde{\mathfrak{o}}$ be the order on $b_i$ that is induced by the orientation of $\Sigma$. To alleviate notation, we write $p<_{\widetilde{\mathfrak{o}}}q$ for $\widetilde{\mathfrak{o}}(p)<\widetilde{\mathfrak{o}}(q)$. For instance, in the stated diagram $\heightexch{<-}{+}{-}$, if $p_L$ is the endpoint with $s(p_L)=+$ and $p_H$ the endpoint with $s(p_H)=-$ and if $\mathfrak{o}$ is the orientation given by the arrow, then $p_L>_{\widetilde{\mathfrak{o}}} p_H$ whereas $p_L<_{\mathfrak{o}} p_H$ (because the $\mathfrak{o}$ and $\widetilde{\mathfrak{o}}$ orientation of the boundary arc where live $p_L$ and $p_H$ are opposite).
\par 

Let $\tau \in \mathbb{S}_{n_i}$ be the transposition that exchanges the $\mathfrak{o}_i$ order of $p_H$ and $p_L$.  
The resolution of $(D,s)$ along $\tau$, denoted by $\Res_{\tau} (D,s, \mathfrak{o})\in \mathcal{S}_{+1}(\mathbf{\Sigma})$, is given by 
\begin{equation*}
\Res_{\tau} (D,s, \mathfrak{o}) = 
\begin{cases}
\frac{1}{2} [ D, s], 	& \text{if either } s(p_H)=s(p_L)\mbox{ and } p_L<_{\widetilde{\mathfrak{o}}} p_H, \\
&  \mbox{or }(s(p_H), s(p_L))=(-,+)\mbox{ and }p_H<_{\widetilde{\mathfrak{o}}} p_L; \\ 
-\frac{1}{2} [D,s],	& \mbox{if either }  s(p_H)=s(p_L)\mbox{ and }p_H<_{\widetilde{\mathfrak{o}}} p_L,  \\ 
& \mbox{or } (s(p_H), s(p_L))=(+,-)\mbox{ and } p_L<_{\widetilde{\mathfrak{o}}}p_H; \\ 
\frac{1}{2}[D,s] -2[{D}, {\tau s}], 	& \mbox{if } (s(p_H), s(p_L))=(+,-)\mbox{ and }p_H<_{\widetilde{\mathfrak{o}}}p_L;
\\ 
-\frac{1}{2}[D,s] +2[{D}, {\tau s}], & , \mbox{if } (s(p_H), s(p_L))=(-,+)\mbox{ and } p_L<_{\widetilde{\mathfrak{o}}} p_H
\end{cases}
\end{equation*}
where $\tau s$ is the state that differs from $s$ only by exchanging the states of $p_H$ and $p_L$. 
\\

Let us extend the resolution to several points, namely any permutation of the
boundary heights on a given boundary component. 
For two transpositions $\sigma_1$ and $\sigma_2$ of $\mathfrak{o}$--consecutive points, let 
\begin{equation}\label{eq: extension res}
\Res_{\sigma_1\circ\sigma_2}(D,s,\mathfrak{o}) = \Res_{\sigma_1}(D,s, \sigma_2\circ \mathfrak{o}) + \Res_{\sigma_2}(D,s,\mathfrak{o}).
\end{equation}

\begin{definition}
	For a permutation $\sigma \in \mathbb{S}_{n_{1}} \times \ldots \times \mathbb{S}_{n_{k}}$, the resolution $\Res_{\sigma}(D,s,\mathfrak{o})$ is defined via \eqref{eq: extension res}, by considering the decomposition of $\sigma$ into transpositions of $\mathfrak{o}$--consecutive points. This is clearly
independent of the choice of decomposition into transpositions.
\end{definition}
\begin{remark}
The resolution $\Res_{\sigma}(D,s,\mathfrak{o})$  is invariant under isotopy of $(D,s)$. Also, one has $\Res_{id}(D,s,\mathfrak{o}) =0$. 
\end{remark}

\begin{lemma}\label{lemmabracket}
In the skein algebra $\mathcal{S}_{\omega_{\hbar}}(\mathbf{\Sigma})$, the following two statements holds. 
\begin{enumerate}
\item Let $D\crosspos$ and $D\crossneg$  be two diagrams that differ from each other only by a change of a crossing $c$.  One has 
$$ [D\crosspos, s, \mathfrak{o}]-[D\crossneg, s, \mathfrak{o}] = \hbar \Res_c(D\crosspos, s) \mod{\hbar^2}. $$
\item Let $(D,s,\mathfrak{o})$ be an $\mathfrak{o}$-ordered stated diagram. For  $\pi\in \mathbb{S}_{n_1}\times \ldots \times \mathbb{S}_{n_k}$ one has 
$$ [D, s , \mathfrak{o}] - [D,s,\pi\circ\mathfrak{o}] = \hbar\Res_{\pi}(D,s,\mathfrak{o}) \mod{\hbar^2}.$$
\end{enumerate}
In the two statements, the resolutions $\Res$  are seen in $\mathcal{S}_{\omega_{\hbar}}(\mathbf{\Sigma})$ via the isomorphism $\psi^{\widetilde{\mathfrak{o}}}$ of \eqref{eq: iso psi Bo}. 
\end{lemma}

\begin{proof} Recall that $\omega_{\hbar}=\exp(-\hbar/4) \equiv 1-\frac{1}{4}\hbar \pmod{\hbar^2}$. 
The first equality follows from Equation \eqref{eq: skein 1} as follows: 
$$  \begin{tikzpicture}[baseline=-0.4ex,scale=0.5,>=stealth]	
  \draw [fill=gray!45,gray!45] (-.6,-.6)  rectangle (.6,.6)   ;
  \draw[line width=1.2,-] (-0.4,-0.52) -- (.4,.53);
  \draw[line width=1.2,-] (0.4,-0.52) -- (0.1,-0.12);
  \draw[line width=1.2,-] (-0.1,0.12) -- (-.4,.53);
  \end{tikzpicture}
   - 
  \begin{tikzpicture}[baseline=-0.4ex,scale=0.5,>=stealth]	
  \draw [fill=gray!45,gray!45] (-.6,-.6)  rectangle (.6,.6)   ;
  \draw[line width=1.2,-] (-0.4,0.53) -- (0.4,-0.52);
  \draw[line width=1.2,-] (-0.4,-0.52) -- (-0.1,-0.12);
  \draw[line width=1.2,-] (0.1,0.12) -- (.4,.53);
  \end{tikzpicture}
  = (\omega^{-2}-\omega^{2}) 
\begin{tikzpicture}[baseline=-0.4ex,scale=0.5,>=stealth] 
  \draw [fill=gray!45,gray!45] (-.6,-.6)  rectangle (.6,.6)   ;
  \draw[line width=1.2] (-0.4,-0.52) ..controls +(.3,.5).. (-.4,.53);
  \draw[line width=1.2] (0.4,-0.52) ..controls +(-.3,.5).. (.4,.53);
  \end{tikzpicture}
+(\omega^{2}-\omega^{-2}) 
  \begin{tikzpicture}[baseline=-0.4ex,scale=0.5,rotate=90]	
  \draw [fill=gray!45,gray!45] (-.6,-.6)  rectangle (.6,.6)   ;
  \draw[line width=1.2] (-0.4,-0.52) ..controls +(.3,.5).. (-.4,.53);
  \draw[line width=1.2] (0.4,-0.52) ..controls +(-.3,.5).. (.4,.53);
  \end{tikzpicture}
 \equiv \left( 
 \begin{tikzpicture}[baseline=-0.4ex,scale=0.5,>=stealth] 
  \draw [fill=gray!45,gray!45] (-.6,-.6)  rectangle (.6,.6)   ;
  \draw[line width=1.2] (-0.4,-0.52) ..controls +(.3,.5).. (-.4,.53);
  \draw[line width=1.2] (0.4,-0.52) ..controls +(-.3,.5).. (.4,.53);
  \end{tikzpicture}
  -
    \begin{tikzpicture}[baseline=-0.4ex,scale=0.5,rotate=90]	
  \draw [fill=gray!45,gray!45] (-.6,-.6)  rectangle (.6,.6)   ;
  \draw[line width=1.2] (-0.4,-0.52) ..controls +(.3,.5).. (-.4,.53);
  \draw[line width=1.2] (0.4,-0.52) ..controls +(-.3,.5).. (.4,.53);
  \end{tikzpicture}
  \right)\hbar \pmod{\hbar^2}.$$
Let us prove the second equality when $\pi$ a transposition of  two consecutive points $p_H, p_L$ with $p_H>_{\mathfrak{o}} p_L$. If $s(p_H)=s(p_L)= \varepsilon$,  Equation \eqref{eq: height exch 1} gives
$$ \heightexch{<-}{\varepsilon}{\varepsilon} = \omega^2  \heightexch{->}{\varepsilon}{\varepsilon} \quad \mbox{ and } \heightexch{->}{\varepsilon}{\varepsilon} = \omega^{-2} \heightexch{<-}{\varepsilon}{\varepsilon} $$ 
from which we deduce
$$  \heightexch{<-}{\varepsilon}{\varepsilon} -  \heightexch{->}{\varepsilon}{\varepsilon}  \equiv \left( -\frac{1}{2} \heightexch{->}{\varepsilon}{\varepsilon} \right)\hbar \pmod{\hbar^2} \quad , 
 \heightexch{->}{\varepsilon}{\varepsilon}  -  \heightexch{<-}{\varepsilon}{\varepsilon} \equiv \left( +\frac{1}{2} \heightexch{<-}{\varepsilon}{\varepsilon}\right)\hbar \pmod{\hbar^2} .$$
 Note that in the stated skein algebra at $\omega=+1$, the height order is irrelevant; said differently, at $\omega_{\hbar}=\exp(-\hbar/4)$, we have the skein relation 
 $$\heightexch{<-}{i}{j} \equiv \heightexch{->}{i}{j} \pmod{\hbar}. $$
 Now, if either $p_H <_{\widetilde{\mathfrak{o}}} p_L$ and $(s(p_H), s(p_L))= (-, +)$ or if $p_L<_{\widetilde{\mathfrak{o}}} p_H$ and $(s(p_H), s(p_L))= (+, -)$, using Equation \eqref{eq: height exch 1} we have: 
 $$ \heightexch{<-}{+}{-} = \omega^{-2}  \heightexch{->}{+}{-} \quad \mbox{ and } \heightexch{->}{+}{-} = \omega^{-2} \heightexch{<-}{+}{-} $$ 
from which we deduce
$$  \heightexch{<-}{+}{-}  -  \heightexch{->}{+}{-}   \equiv \left( +\frac{1}{2} \heightexch{->}{+}{-} \right)\hbar \pmod{\hbar^2} \quad , 
 \heightexch{->}{+}{-}   -  \heightexch{<-}{+}{-}  \equiv \left( -\frac{1}{2} \heightexch{<-}{+}{-} \right)\hbar \pmod{\hbar^2} .$$
If $p_H<_{\widetilde{\mathfrak{o}}}p_L$ and  $ (s(p_H), s(p_L))=(+,-)$, Equations 
 \eqref{eq: height exch corr} and \eqref{eq: skein 2} imply that 

$$
\heightexch{<-}{-}{+}
= \omega^{-2}
\heightexch{->}{-}{+}
+(\omega^{2}-\omega^{-6})
\heightexch{->}{+}{-};
$$
from which we deduce
$$
\heightexch{<-}{-}{+}
-
\heightexch{->}{-}{+}
=\left( \frac{1}{2} 
\heightexch{->}{-}{+}
-2
\heightexch{->}{+}{-}
\right) \hbar \text{ mod } \hbar^2. 
$$
Eventually the case where $p_L<_{\widetilde{\mathfrak{o}}} p_H$ and   $(s(p_H), s(p_L))=(-,+)$ is deduced from this case by taking the opposite of the preceding equality. This concludes the proof of the second equality of the lemma when $\tau$ is a transposition. The case of a general permutation $\pi$  follows by induction on the number of transpositions in a decomposition of $\pi$.

\end{proof}
\vspace{2mm}
\begin{proposition}\label{propbracket}
Let  $(D_1, s_2, \mathfrak{o}_1)$ and $(D_2,s_2, \mathfrak{o}_2)$ be two height ordered stated diagrams  such that $D_1$ and $D_2$ intersect transversally in the interior of $\Sigma_{\mathcal{P}}$. Let $(D_1D_2, s_1s_2)$ be the stated diagram obtained by staking $D_1$ on top of $D_2$, $\mathfrak{o}_1\mathfrak{o}_2$ the resulting height order and $\pi$ the permutation sending $\mathfrak{o}_2\mathfrak{o}_1$ to $\mathfrak{o}_1\mathfrak{o}_2$.  
In $\mathcal{S}_{+1}(\mathbf{\Sigma})$, the Poisson bracket from Definition \ref{de:poiss brack} satisfies  
$$ \left\{ [D_1, s_1], [D_2, s_2]\right\}^s = \sum_{c\in D_1\cap D_2} \Res_c(D_1D_2, s_1s_2) + \Res_{\pi}(D_1D_2, s_1s_2, \mathfrak{o_1}\mathfrak{o_2}).$$

\end{proposition}

\begin{proof} In the algebra $\mathcal{S}_{\omega_{\hbar}}(\mathbf{\Sigma})$, the products writes $[D_1, s_1, \mathfrak{o}_1] \cdot [D_2, s_2, \mathfrak{o}_2] = [D_1D_2, s_1s_2, \mathfrak{o}_1\mathfrak{o}_2]$ and  $[D_2, s_2, \mathfrak{o}_2] \cdot [D_1, s_1, \mathfrak{o}_1] = [D_2D_1, s_2s_1, \mathfrak{o}_2\mathfrak{o}_1]$. We pass from the diagram $D_1D_2$ to $D_2D_1$ by changing each crossing in the intersection of the diagrams and changing the height order using $\pi$, so the formula is a consequence of Lemma \ref{lemmabracket}.
\end{proof}

\vspace{2mm}
\begin{remark} Note that nor $\{\cdot, \cdot\}^s$ nor the formula in Proposition \ref{propbracket} depend on a choice of orientation of the boundary arcs by Remark \ref{remark_independance}.
When $\Sigma$ is a closed surface, we recover Goldman's formula from \cite{Goldman86}. When $\Sigma$ has non trivial boundary and no inner punctures, the sub-algebra of the stated skein algebra generated by tangles with states having only value $+$ is isomorphic to the M\"uller algebra defined in \cite{Muller} (see \cite[Section $6$]{LeStatedSkein} ). The Poisson bracket restricts to the corresponding sub-algebra of $\mathcal{S}_{+1}(\mathbf{\Sigma})$ and the resulting Poisson algebra is isomorphic to Yuasa' s Poisson algebra in \cite{Yuasa}.
\end{remark}

\vspace{3mm}
\begin{example}\label{exampleBigone} The Poisson bracket $\{-,-\}^s$ on the commutative algebra $\mathcal{S}_{+1}(\mathbb{B})$ is given by
	
\begin{align*}
\{ \alpha_{++}, \alpha_{+-}\}^s &= -\alpha_{+-}\alpha_{++} & \{\alpha_{++}, \alpha_{-+}\}^s &= -\alpha_{-+}\alpha_{++} 
\\ \{ \alpha_{--}, \alpha_{+-} \}^s &= \alpha_{+-}\alpha_{--} & \{\alpha_{--}, \alpha_{-+} \}^s &= \alpha_{-+}\alpha_{--}
\\ \{ \alpha_{+-}, \alpha_{-+} \}^s &= 0 & \{ \alpha_{++}, \alpha_{--}\}^s &= -2\alpha_{+-}\alpha_{-+}
\end{align*}
\end{example}

\vspace{2mm}
\begin{example}\label{exampleTriangle} For the triangle $\mathbb{T}$, the Poisson structure is described by the previous formulas in Example \ref{exampleBigone} by replacing $\alpha$ by each of the three arcs $\alpha, \beta$ and $\gamma$, together with the following relations and their images through the automorphisms $\tau$ and $\tau^2$:
\begin{align*}
\{ \gamma_{\varepsilon \mu}, \alpha_{ \mu' \varepsilon}\}^s &= -\frac{1}{2} \gamma_{\varepsilon \mu} \alpha_{\mu' \varepsilon} 
\\ \{ \gamma_{- \mu}, \alpha_{\mu' +} \}^s &= \frac{1}{2} \gamma_{- \mu}\alpha_{\mu' +}
\\ \{ \gamma_{+ \mu}, \alpha_{\mu' -}\}^s &= -\frac{3}{2} \gamma_{+ \mu}\alpha_{\mu' -} + 2 \beta_{\mu \mu'}.
\end{align*}
\end{example}

 \section{Relative character varieties}\label{sec 3}
 
 \subsection{Relative character varieties for open surfaces}
 
 \par In this subsection we briefly recall from \cite{KojuTriangularCharVar} the definition and main properties of character varieties for open surfaces.
 \vspace{2mm}
 
 \par The character variety of a closed punctured connected surface $\mathbf{\Sigma}$ is 
 the algebraic quotient (familiar in geometric invariant theory):
 \begin{equation*}
 \mathcal{X}_{\SL_2}(\mathbf{\Sigma}):= \Hom\left( \pi_1(\Sigma_{\mathcal{P}}) , \SL_2(\mathbb{C}) \right)  \sslash \SL_2(\mathbb{C}) 
 \end{equation*}  
 under the action  by conjugation of $\SL_2(\mathbb{C}) $. Recall that by "closed", we mean that $\Sigma$ is closed though  in this case  $\Sigma_{\mathcal{P}}$ is not closed when $\mathcal{P}\neq \emptyset$.  
 In \cite{Goldman86}, Goldman defined a Poisson structure on its algebra of regular functions. It follows from \cite{Bullock, PS00, Barett, Turaev91} that, given a spin  structure $S$ on $\Sigma$ with quadratic form $\omega_S$,  there is a Poisson isomorphism  
 $$\phi^S : (\mathcal{S}_{+1}(\mathbf{\Sigma}), \{\cdot, \cdot \}^s) \xrightarrow{\cong} (\mathbb{C}[\mathcal{X}_{\SL_2}(\mathbf{\Sigma})], \{\cdot, \cdot \}). $$
 For each  non-contractible closed curve $\gamma$,  it is given by $\phi^S(\gamma)= (-1)^{\omega_S([\gamma])+1} \tau_{\gamma} $, where $\tau_{\gamma}$ is the regular function  $\tau_{\gamma}([\rho]):= \tr (\rho(\gamma))$. 
\vspace{2mm}
 
 \par In \cite{KojuTriangularCharVar}, the first author introduced a generalization of the character varieties to punctured surfaces which are not necessarily closed and which is closely related to the construction of Fock-Rosly in \cite{FockRosly} and specifies to the constructions in \cite{AlekseevKosmannMeinrenken, AlekseevMalkinMeinrenken_LieGroupMomentMap, AlekseevMalkin_PoissonCharVar} and \cite{GHJW_ModSpacesParBd}
  when the marked surface is connected and has exactly one boundary arc (see \cite{KojuTriangularCharVar} for a precise comparison). We will also denote it by $ \mathcal{X}_{\SL_2}(\mathbf{\Sigma})$. 
 \begin{notations}
    For a topological space $X$, we let $\Pi_1(X)$ denote its fundamental groupoid: objects are the points in $X$ and morphisms are homotopy classes of oriented paths.  We let $s$ and $t$ denote the source and target maps, which for a morphism $\alpha: v_1\rightarrow v_2$ are given by $s(\alpha)=v_1$ and $t(\alpha)=v_2$.
  By convention, we compose the morphisms from left to right, \textit{i.e.} if $\alpha_1 : v_1 \rightarrow v_2$ and $\alpha_2 : v_2 \rightarrow v_3$ are two paths, their composition is a path $\alpha_1 \alpha_2 : v_1 \rightarrow v_3$. For $S\subset X$, we denote by $\Pi_1(X,S)$ the full subcategory of $\Pi_1(X)$ whose objects are points in $S$. 
  For a group $G$, the set  $\Hom(\Pi_1(X,S), G)$ denotes the set of functors $\rho : \Pi_1(X,S) \rightarrow G$, where $G$ is seen as a category with one element. With our conventions, if $t(\alpha_1)=s(\alpha_2)$, one has $\rho(\alpha_1\alpha_2)=\rho(\alpha_1)\rho(\alpha_2)$.
  \end{notations}

 Let $\mathcal{R}_{\SL_2}(\mathbf{\Sigma})$ be the set of functors $\rho: \Pi_1(\Sigma_{\mathcal{P}}) \to \SL_2$ whose restriction to $\Pi_1(\partial \Sigma_{\mathcal{P}}) \subset \Pi_1(\Sigma_{\mathcal{P}}) $ is the constant map with value the neutral element $\mathds{1}_2\in \SL_2$. Let $\mathcal{G}$ be the group of maps $g : \Sigma_{\mathcal{P}} \to \SL_2$ whose restriction to $\partial \Sigma_{\mathcal{P}}$ is constant with value $\mathds{1}_2$ and with finite support. It acts on $\mathcal{R}_{\SL_2}(\mathbf{\Sigma})$ by the formula 
 $$ g \cdot \rho (\alpha) := g(s(\alpha))^{-1} \rho(\alpha) g(t(\alpha)), \quad \rho \in \mathcal{R}_{\SL_2}(\mathbf{\Sigma}), g\in \mathcal{G}, \alpha \in \Pi_1(\Sigma_{\mathcal{P}}).$$
 Both $\mathcal{R}_{\SL_2}(\mathbf{\Sigma})$ and $\mathcal{G}$ have a structure of affine scheme and the action is algebraic so we can define the GIT quotient:
 	\begin{equation}\label{eq_def_charvar}
 	\mathcal{X}_{\SL_2}(\mathbf{\Sigma}):= \mathcal{R}_{\SL_2}(\mathbf{\Sigma}) \sslash \mathcal{G}.
 	\end{equation}
 	
 The character variety turns out to be an affine Poisson variety whose Poisson structure (given by a generalized Goldman formula) depends on a choice of orientation of the boundary arcs. It is
proved in \cite[Theorem $1.1$]{KojuTriangularCharVar} that its algebra  of regular functions $ \mathbb{C}[\mathcal{X}_{\SL_2}(\mathbf{\Sigma})]$ is well-behaved under triangular decompositions: for a topological triangulation $\Delta$, there are an injective Poisson morphism $i^{\Delta} : \mathbb{C}[\mathcal{X}_{\SL_2}(\mathbf{\Sigma})] \hookrightarrow \otimes_{\mathbb{T}\in F(\Delta)} \mathbb{C}[\mathcal{X}_{\SL_2}(\mathbb{T})]$ and Poisson Hopf comodule maps $\Delta^L$ and $\Delta^R$ such that the following sequence is exact: 
 \begin{multline}\label{eq_exactsequence}  0 \to \mathbb{C}[\mathcal{X}_{\SL_2}(\mathbf{\Sigma})] \xrightarrow{i^{\Delta}} \otimes_{\mathbb{T}\in F(\Delta)} \mathbb{C}[\mathcal{X}_{\SL_2}(\mathbb{T})]
 \\ \xrightarrow{\Delta^L - \sigma \circ \Delta^R} \left( \otimes_{e \in \mathring{\mathcal{E}}(\Delta)} \mathbb{C}[\SL_2] \right) \otimes \left( \otimes_{\mathbb{T}\in F(\Delta)} \mathbb{C}[\mathcal{X}_{\SL_2}(\mathbb{T})]\right). 
 \end{multline}

  In the present paper, we proceed by describing the character variety for the bigon and the triangle, together with the Hopf comodule maps $\Delta^L$ and $\Delta^R$. Then, in virtue of the above property, we characterize the Poisson structure of the relative character variety for any triangulated punctured surface as the kernel of $\Delta^L - \sigma \circ \Delta^R$. 
\\

 First, recall that   $\mathfrak{sl}_2$ denotes the Lie algebra of the  $2\times 2$ traceless matrices. It has a basis formed by 
 \begin{equation*}
 H:=\begin{pmatrix} 1 & 0 \\ 0& -1 \end{pmatrix}, 
 E:= \begin{pmatrix} 0 & 1 \\ 0&0 \end{pmatrix} \text{ and } 
 F:= \begin{pmatrix} 0&0 \\ 1 & 0 \end{pmatrix}. 
 \end{equation*}
 In order to define the Poisson structure, we will need the following. 
 
 \begin{definition}
 The \textit{classical }r\textit{-matrices} $r^{\pm}\in \mathfrak{sl}_2^{\otimes 2}$ are the bi-vectors $r^+ := \frac{1}{2} H\otimes H + 2 E\otimes F$ and $r^-:= \frac{1}{2} H\otimes H +2 F\otimes E$. Their symmetric part $\tau = \frac{1}{2} H\otimes H + E\otimes F +F \otimes E$ is the invariant bi-vector associated to the (suitably normalized) Killing form and we denote by $\overline{r}^+:= E\otimes F -F\otimes E=:-\overline{r}^-$ their skew-symmetric part.
 \end{definition}
 The classical $r$-matrices satisfy the classical Yang-Baxter equation (see \cite{DrinfeldrMatrix}, \cite[Section $2.1$]{ChariPressley} for details).
  
 \begin{notations} Given $a$ a boundary arc of $\mathbf{\Sigma}$, we write $\mathfrak{o}(a)=+$ if the $\mathfrak{o}$-orientation of $a$ coincides with the orientation induced by the orientation of $\Sigma_{\mathcal{P}}$ and write $\mathfrak{o}(a)=-$ if the orientation are opposite. 
\end{notations}
 
 \subsubsection{The bigon} 
 Consider the bigon $\mathbb{B}$ and write $\mathfrak{o}(b_L)=\varepsilon_1$ and $\mathfrak{o}(b_R)=\varepsilon_2$.
 
 \begin{definition}\label{def_bigone_poisson} The relative character variety of the bigon is $\mathcal{X}_{\SL_2}(\mathbb{B}):= \SL_2(\mathbb{C})$. Denote by $N= \begin{pmatrix} x_{++} & x_{+-} \\ x_{-+} & x_{--} \end{pmatrix}$ the $2\times 2$ matrix with coefficients in $\mathbb{C}[\mathcal{X}_{\SL_2}(\mathbb{B})]$. The Poisson bracket associated to $\mathfrak{o}$ is defined by:
  $$ \left\{ N \otimes N \right\}^{\varepsilon_1, \varepsilon_2} :=  \overline{r}^{\varepsilon_1} (N \otimes N) + (N\otimes N)\overline{r}^{\varepsilon_2}. $$
  \end{definition}
 \par Here we used the classical notation $\{N \otimes N\}$ to denote the matrix defined by $\{N\otimes N \}_{\varepsilon \varepsilon' \mu \mu'}= \{x_{\varepsilon \varepsilon'}, x_{\mu \mu'} \}$ (see for instance \cite[Section $2.2.A$]{ChariPressley} for details on this notation).
 \vspace{2mm}
 \par  Denote the Poisson variety $(\mathbb{C}[\SL_2], \{\cdot, \cdot \}^{\varepsilon_1, \varepsilon_2})$ by $\mathbb{C}[\SL_2]^{\varepsilon_1, \varepsilon_2}$. Remark that $\{\cdot, \cdot\}^{\varepsilon_1, \varepsilon_2} = - \{\cdot, \cdot\}^{-\varepsilon_1, -\varepsilon_2}$.  By \cite[Lemma $4.1$]{KojuTriangularCharVar}, the coproduct  $\Delta : \mathbb{C}[\SL_2]^{\varepsilon_1, \varepsilon_2} \rightarrow \mathbb{C}[\SL_2]^{\varepsilon_1, \varepsilon} \otimes \mathbb{C}[\SL_2]^{-\varepsilon, \varepsilon_2}$ and the antipode $S: \mathbb{C}[\SL_2]^{\varepsilon_1, \varepsilon_2} \rightarrow \mathbb{C}[\SL_2]^{-\varepsilon_1, -\varepsilon_2}$ are Poisson morphisms. In particular, the Poisson brackets $\{\cdot, \cdot \}^{-, +}$ and $\{\cdot, \cdot\}^{+,-}$ are the only ones which endow $\SL_2(\mathbb{C})$ with a Poisson-Lie structure.

 \subsubsection{The triangle} Consider the triangle $\mathbb{T}$ and fix an orientation $\mathfrak{o}$ of each of its three boundary arcs $a,b$ and $c$.  We will use the notation $s(\alpha)=t(\beta):=c$, $s(\gamma)=t(\alpha):=b$ and $s(\beta)=t(\gamma):=a$. Here, for instance, we think of $\alpha$ as an oriented path joining a point in $c=s(\alpha)$ (source) to a point in $b=t(\alpha)$ (target).
 
 \begin{definition}\label{def_triangle_poisson}
 The relative character variety of the triangle is the affine variety:
 $$ \mathcal{X}_{\SL_2}(\mathbb{T}):= \left\{ (M_{\alpha}, M_{\beta}, M_{\gamma}) \in \SL_2(\mathbb{C})^3 | M_{\gamma}M_{\beta}M_{\alpha}=\mathds{1} \right\} $$
 \par Given $\delta \in \{ \alpha, \beta, \gamma \}$,  denote by $N_{\delta}:=\begin{pmatrix} \delta(+,+) & \delta(+,-)\\ \delta(-,+) & \delta(-,-) \end{pmatrix}$ and the $2\times 2$ matrix with coefficients in $\mathbb{C}[\mathcal{X}_{\SL_2}(\mathbb{T})]$. The Poisson bracket $\{\cdot, \cdot\}^{\mathfrak{o}}$ is defined by the formulas:
 \begin{eqnarray*}
 \left\{ N_{\delta} \otimes N_{\delta} \right\}^{\mathfrak{o}} &:=&  \overline{r}^{\mathfrak{o}(s(\delta))} (N_{\delta} \otimes N_{\delta}) + (N_{\delta}\otimes N_{\delta})\overline{r}^{\mathfrak{o}(t(\delta))}, \delta \in \{\alpha, \beta, \gamma\}; \\
 \{ N_{\alpha} \otimes N_{\gamma} \}^{\mathfrak{o}} &:=& - (N_{\alpha} \otimes \mathds{1}) r^{\mathfrak{o}(b)} (\mathds{1} \otimes N_{\gamma}); \\
 \{ N_{\gamma}\otimes N_{\beta} \}^{\mathfrak{o}} &:=& - (N_{\gamma} \otimes \mathds{1}) r^{\mathfrak{o}(a)} (\mathds{1} \otimes N_{\beta} ); \\ 
  \{ N_{\beta}\otimes  N_{\alpha} \}^{\mathfrak{o}} &:=& - (N_{\beta} \otimes \mathds{1}) r^{\mathfrak{o}(c)} (\mathds{1} \otimes N_{\alpha} ). 
 \end{eqnarray*}
 \end{definition}

 \par Remark that, writing $S(N_{\delta}):= \begin{pmatrix} \delta(-,-) & -\delta(+,-) \\ -\delta(-,+) & \delta(+,+) \end{pmatrix}$,  the last expressions can be re-written in the form:
 \begin{eqnarray*}
  \{ N_{\alpha} \otimes S(N_{\gamma}) \}^{\mathfrak{o}} = (N_{\alpha}\otimes S(N_{\gamma}))r^{\mathfrak{o}(b)} \\
  \{ N_{\gamma} \otimes S(N_{\beta}) \}^{\mathfrak{o}} = (N_{\gamma}\otimes S(N_{\beta}))r^{\mathfrak{o}(a)} \\
    \{ N_{\beta} \otimes S(N_{\alpha}) \}^{\mathfrak{o}} = (N_{\beta}\otimes S(N_{\alpha}))r^{\mathfrak{o}(c)}
    \end{eqnarray*} 
 
 \par Given a boundary arc $d\in \{a,b,c\}$, we define a left Hopf-comodule $\Delta_d^L: \mathbb{C}[\mathcal{X}_{\SL_2}(\mathbb{T})] \rightarrow \mathbb{C}[\SL_2]^{(+\mathfrak{o}(d), - \mathfrak{o}(d))} \otimes \mathbb{C}[\mathcal{X}_{\SL_2}(\mathbb{T})]$ by:
 $$ \begin{pmatrix} \Delta_d^L (\delta(+,+)) & \Delta_d^L(\delta(+,-)) \\ \Delta_d^L(\delta(-,+)) & \Delta_d^L(\delta(-,-)) \end{pmatrix} := 
 \left\{ \begin{array}{ll}
 \begin{pmatrix} x_{++} & x_{+-} \\ x_{-+} & x_{--} \end{pmatrix} \otimes N_{\delta} &\mbox{, if }s(\delta)=d; \\
 \mathds{1}\otimes N_{\delta} &\mbox{, else.}
 \end{array} \right. $$
 \par Similarly, define a right Hopf-comodule $\Delta_d^R:  \mathbb{C}[\mathcal{X}_{\SL_2}(\mathbb{T})] \rightarrow  \mathbb{C}[\mathcal{X}_{\SL_2}(\mathbb{T})] \otimes \mathbb{C}[\SL_2]^{(-\mathfrak{o}(d), + \mathfrak{o}(d))} $ by:
  $$ \begin{pmatrix} \Delta_d^R (\delta(+,+)) & \Delta_d^R(\delta(+,-)) \\ \Delta_d^R(\delta(-,+)) & \Delta_d^R(\delta(-,-)) \end{pmatrix} := 
 \left\{ \begin{array}{ll}
N_{\delta} \otimes  \begin{pmatrix} x_{++} & x_{+-} \\ x_{-+} & x_{--} \end{pmatrix}  &\mbox{, if }t(\delta)=d; \\
 N_{\delta} \otimes  \mathds{1} &\mbox{, else.}
 \end{array} \right. $$
 \par By  \cite[Lemma $4.6$ ]{KojuTriangularCharVar}, both $\Delta_d^L$ and $\Delta_d^R$ are Poisson morphisms.

 \subsubsection{The general case} Let $\mathbf{\Sigma}$ be a punctured surface, $\Delta$ a topological triangulation of $\mathbf{\Sigma}$, and $\mathfrak{o}_{\Delta}$ an orientation of each edge of $\Delta$. For a face $\mathbb{T}\in F(\Delta)$, let $\mathfrak{o}_{\mathbb{T}}$ be the orientation of its boundary arcs given by  $\mathfrak{o}_{\Delta}$. Equip the algebra $\otimes_{\mathbb{T}\in F(\Delta)}\mathbb{C}[\mathcal{X}_{\SL_2}(\mathbb{T})]^{\mathfrak{o}_{\mathbb{T}}}$ with the  Poisson bracket defined in Definition \ref{def_triangle_poisson} . Each inner edge $e\in \mathring{\mathcal{E}}(\Delta)$ lifts to two oriented boundary arcs in $\mathbf{\Sigma}_{\Delta}:= \bigsqcup_{\mathbb{T}\in F(\Delta)} \mathbb{T}$. We denote by $e_L$ the lift of $e$ whose orientation coincides with the induced orientation of $\mathbf{\Sigma}_{\Delta}$ and by $e_R$ the other lift. The comodule maps $\Delta_{e_L}^L$ and $\Delta_{e_R}^R$ induce the comodule maps: 
 $$\Delta^L : \otimes_{\mathbb{T}\in F(\Delta)} \mathbb{C}[\mathcal{X}_{\SL_2}(\mathbb{T})]^{\mathfrak{o}_{\mathbb{T}}} \rightarrow \left( \otimes_{e \in \mathring{\mathcal{E}}(\Delta)} \mathbb{C}[\SL_2]^{-,+} \right) \otimes \left( \otimes_{\mathbb{T}\in F(\Delta)} \mathbb{C}[\mathcal{X}_{\SL_2}(\mathbb{T})]^{\mathfrak{o}_{\mathbb{T}}}\right);$$  
 $$\Delta^R : \otimes_{\mathbb{T}\in F(\Delta)} \mathbb{C}[\mathcal{X}_{\SL_2}(\mathbb{T})]^{\mathfrak{o}_{\mathbb{T}}} \rightarrow \left( \otimes_{\mathbb{T}\in F(\Delta)} \mathbb{C}[\mathcal{X}_{\SL_2}(\mathbb{T})]^{\mathfrak{o}_{\mathbb{T}}} \right) \otimes \left( \otimes_{e \in \mathring{\mathcal{E}}(\Delta)} \mathbb{C}[\SL_2]^{-,+} \right).$$

 \begin{definition}\label{def_charvar} The relative character variety $\mathcal{X}_{\SL_2}(\mathbf{\Sigma})$ is the affine variety whose algebra of regular functions is the kernel of $\Delta^L - \sigma \circ \Delta^R$. 
 \end{definition}
 
 \begin{lemma}(\cite[Theorem $1.4$]{KojuTriangularCharVar}) As a Poisson variety, $\mathcal{X}_{\SL_2}(\mathbf{\Sigma})$ only depends, up to canonical isomorphism, on the marked surface $\mathbf{\Sigma}$ and the orientation $\mathfrak{o}$ of the boundary arcs (so does not depends on the triangulation $\Delta$ nor on $\mathfrak{o}_{\Delta}$).
 \end{lemma}
 
 We denote by $\{.,.\}^{\mathfrak{o}}$ the Poisson bracket on $\mathbb{C}[\mathcal{X}_{\SL_2}(\mathbf{\Sigma})]$.
More precisely, in \cite{KojuTriangularCharVar}, we endow the variety $\mathcal{X}_{\SL_2}(\mathbf{\Sigma}):= \mathcal{R}_{\SL_2}(\mathbf{\Sigma}) \sslash \mathcal{G}$ (which only depends on $\mathbf{\Sigma}$) with a Poisson structure, given by a generalization of Goldman formula, which only depends on $\mathfrak{o}$. We then construct a splitting morphism $i^{\Delta}$ and prove in 
\cite[Theorem $1.4$]{KojuTriangularCharVar} that we have the exact sequence \eqref{eq_exactsequence}, thus $\mathcal{X}_{\SL_2}(\mathbf{\Sigma})$ can be alternatively defined using Definition \ref{def_charvar}.
\par
   Moreover when $\Sigma$ is closed, the Poisson variety $\mathcal{X}_{\SL_2}(\mathbf{\Sigma})$ is canonically isomorphic to the "classical" (Culler-Shalen) character variety with its Goldman Poisson structure (\cite[Theorem $1.1$]{KojuTriangularCharVar}).

 \subsection{Relation between relative character varieties and stated skein algebras}
 The goal of this subsection is to prove Theorem \ref{theorem3} which we recall here for the reader convenience: 
 
 \begin{theorem}\label{theorem30} 
	Suppose that  $\mathbf{\Sigma}$ has a topological triangulation $\Delta$. 
 Let $\mathfrak{o}_{\Delta}$ be an orientation of the edges of $\Delta$ and $\mathfrak{o}$ be the induced orientation of the boundary arcs of $\mathbf{\Sigma}$.
    There exists an isomorphism of Poisson algebras 
$$\Psi^{(\Delta, \mathfrak{o}_{\Delta})} : \left( \mathcal{S}_{+1}(\mathbf{\Sigma}), \{ \cdot, \cdot \}^s \right) \xrightarrow{\cong} \left( \mathbb{C}[\mathcal{X}_{\SL_2}(\mathbf{\Sigma})], \{\cdot, \cdot \}^{\mathfrak{o}} \right). $$
Moreover, the above isomorphism exists for small punctured surfaces (see Definition \ref{def small}), for which it only depends on  $\mathfrak{o}$. 
\end{theorem}
 
\par  We first prove this theorem for the bigon and the triangle, then we prove the general case using a topological triangulation.

\subsubsection{The case of the bigon}
 Let 
 	 \begin{equation*}
 	  M:=\begin{pmatrix} \alpha_{++} & \alpha_{+-} \\ \alpha_{-+} & \alpha_{--} \end{pmatrix}, 
 	  N:=\begin{pmatrix} x_{++} & x_{+-} \\ x_{-+} & x_{--}\end{pmatrix}  \text{  and } 
 	  C:=\begin{pmatrix} 0 & 1 \\ -1 & 0\end{pmatrix}   
 	 \end{equation*}
 	 be three  matrices with  coefficients in $\mathcal{S}_{+1}(\mathbb{B})$, $\mathbb{C}[\SL_2]$ and $\mathbb{C}$ respectively.  	
 
 \begin{lemma}\label{lemma_poisson_bigone} For $\varepsilon_1, \varepsilon_2 \in \{ -,+\}$, there is a Poisson isomorphism $\Psi^{\varepsilon_1,\varepsilon_2} : (\mathcal{S}_{+1}(\mathbb{B}), \{\cdot, \cdot \}^s) \xrightarrow{\cong} \mathbb{C}[\SL_2]^{\varepsilon_1, \varepsilon_2}$ defined by the formula: 
 $$ \Psi^{\varepsilon_1, \varepsilon_2} (M) := \left\{ 
 \begin{array}{ll}
 N &\mbox{, if }(\varepsilon_1,\varepsilon_2) = (-,+); \\
 CNC &\mbox{, if }(\varepsilon_1,\varepsilon_2) = (+,-);\\ 
  -CN &\mbox{, if }(\varepsilon_1,\varepsilon_2) = (+,+); \\
   -NC &\mbox{, if }(\varepsilon_1,\varepsilon_2) = (-,-). \\
 \end{array}\right.$$
   \end{lemma}
   
   \begin{proof}
   That $\Psi^{\varepsilon_1, \varepsilon_2}$ is an isomorphism of algebras follows from the fact that $\det(C)=1$.  Let us see the compatibility of $\Psi^{\varepsilon_1, \varepsilon_2}$ with the Poisson structures.  For $(\varepsilon_1, \varepsilon_2)=(-,+)$, this follows from a direct comparison of Definition \ref{def_bigone_poisson} and Example \ref{exampleBigone}. Indeed, one has: 
   \begin{eqnarray*}
   \left\{ N\otimes N\right\}^{-,+} &=& \overline{r}^-(N\otimes N) + (N\otimes N)\overline{r}^+ 
   \\  &=& (F\otimes E - E\otimes F)(N\otimes N) +(N\otimes N)(E\otimes F - F\otimes E) \\
   &=& \begin{pmatrix} 0 & x_{++} \\ 0 & x_{-+} \end{pmatrix} \otimes \begin{pmatrix} x_{+-} & 0 \\ x_{--} & 0\end{pmatrix} - 
   \begin{pmatrix} x_{+-} & 0 \\ x_{--} & 0\end{pmatrix} \otimes \begin{pmatrix} 0 & x_{++} \\ 0 & x_{-+} \end{pmatrix} + \\
   && 
   \begin{pmatrix} 0 & 0 \\ x_{++} & x_{+-} \end{pmatrix} \otimes \begin{pmatrix} x_{-+} & x_{--} \\ 0 & 0 \end{pmatrix} - 
   \begin{pmatrix} x_{-+} & x_{--} \\ 0 & 0 \end{pmatrix} \otimes \begin{pmatrix} 0 & 0 \\ x_{++} & x_{+-} \end{pmatrix}.
   \end{eqnarray*}
    \par We recover  the formulas computed in Example \ref{exampleBigone}. 
    For $(\varepsilon_1, \varepsilon_2)=(+,+)$,  we prove that the isomorphism $\varphi : \mathbb{C}[\SL_2]^{-,+} \xrightarrow{\cong} \mathbb{C}[\SL_2]^{+,+}$ given by $\varphi := \Psi^{+,+} \circ(\Psi^{-,+})^{-1}$, is a Poisson morphism. 
   Note that $\varphi(N)= -CN$ and that $ (C\otimes C) \overline{r}^{\varepsilon} = \overline{r}^{- \varepsilon} (C\otimes C)$. 
   It follows that  
    \begin{eqnarray*}
    \{ \varphi(N) \otimes \varphi(N) \}^{+, +} &=& \overline{r}^+ (CN \otimes CN) + (CN\otimes CN) \overline{r}^+ \\
    &=&  (C\otimes C) \left( \overline{r}^- (N\otimes N) + (N\otimes N) \overline{r}^+ \right) = \varphi^{\otimes 2} \left( \{ N\otimes N\}^{-, +} \right)
    \end{eqnarray*}
  which proves the claim. The two remaining cases for $(\varepsilon_1, \varepsilon_2)$ are proved similarly.   
   \end{proof}
   
   \subsubsection{The case of the triangle}  For $\delta \in \{ \alpha, \beta, \gamma \}$, let 
	\begin{equation*}
	 M_{\delta} := \begin{pmatrix} \delta_{++} & \delta_{+-} \\ \delta_{-+} & \delta_{--} \end{pmatrix} \text{  and } 
	 N_{\delta}:= \begin{pmatrix} \delta(+,+) & \delta(+,-) \\ \delta(-,+) & \delta(-,-) \end{pmatrix} 
	\end{equation*}
be  two matrices with coefficients in $\mathcal{S}_{+1}(\mathbb{T})$ and $\mathbb{C}[\mathcal{X}_{\SL_2}(\mathbb{T})]$ respectively.
      
   \begin{lemma}\label{lemma_poisson_triangle} There is a Poisson isomorphism $\Psi^{\mathfrak{o}} : (\mathcal{S}_{+1}(\mathbb{T}), \{\cdot, \cdot \}^s ) \xrightarrow{\cong} (\mathbb{C}[\mathcal{X}_{\SL_2}(\mathbb{T})], \{\cdot, \cdot \}^{\mathfrak{o}} )$ defined by the formula: 
    $$ \Psi^{\mathfrak{o}} (M_{\delta}) := \left\{ 
 \begin{array}{ll}
 N_{\delta} &\mbox{, if }(\mathfrak{o}(s(\alpha)),\mathfrak{o}(t(\alpha))) = (-,+); \\
 CN_{\delta}C &\mbox{, if }(\mathfrak{o}(s(\alpha)),\mathfrak{o}(t(\alpha))) = (+,-);\\ 
  -CN_{\delta} &\mbox{, if }(\mathfrak{o}(s(\alpha)),\mathfrak{o}(t(\alpha))) = (+,+); \\
   -N_{\delta}C &\mbox{, if }(\mathfrak{o}(s(\alpha)),\mathfrak{o}(t(\alpha))) = (-,-). \\
 \end{array}\right.$$
   for each $\delta \in \{ \alpha, \beta, \gamma \}$. Moreover, if $d\in \{a,b,c\}$ is a boundary arc of $\mathbb{T}$, the following diagrams commute: 
   
   \begin{eqnarray*}
   \begin{tikzcd}
   \mathcal{S}_{+1}(\mathbb{T}) \arrow[r, "\Delta_d^L"] \arrow[d, "\Psi^{\mathfrak{o}}", "\cong"'] &
   \mathcal{S}_{+1}(\mathbb{B}) \otimes \mathcal{S}_{+1}(\mathbb{T}) \arrow[d, "\Psi^{\mathfrak{o}(d), - \mathfrak{o}(d)} \otimes \Psi^{\mathfrak{o}}", "\cong"'] \\
   \mathbb{C}[\mathcal{X}_{\SL_2}(\mathbb{T})] \arrow[r, "\Delta_d^L"] & \mathbb{C}[\SL_2]\otimes \mathbb{C}[\mathcal{X}_{\SL_2}(\mathbb{T})] 
   \end{tikzcd} &
   \begin{tikzcd}
   \mathcal{S}_{+1}(\mathbb{T}) \arrow[r, "\Delta_d^R"] \arrow[d, "\Psi^{\mathfrak{o}}", "\cong"'] &
   \mathcal{S}_{+1}(\mathbb{T}) \otimes \mathcal{S}_{+1}(\mathbb{B})   \arrow[d, "\Psi^{\mathfrak{o}} \otimes \Psi^{-\mathfrak{o}(d),  \mathfrak{o}(d)} ", "\cong"'] \\
   \mathbb{C}[\mathcal{X}_{\SL_2}(\mathbb{T})] \arrow[r, "\Delta_d^R"] & \mathbb{C}[\mathcal{X}_{\SL_2}(\mathbb{T})] \otimes \mathbb{C}[\SL_2] 
   \end{tikzcd} 
   \end{eqnarray*}
   \end{lemma}
   
   \begin{proof} That $\Psi^{\mathfrak{o}}$ is an algebra morphism follows from Lemma \ref{lemma_triangle+1}.  For $\delta\in \{\alpha, \beta, \gamma\}$, the equality $(\Psi^{\mathfrak{o}})^{\otimes 2} (\{ \delta_{\varepsilon \varepsilon'} , \delta_{\mu \mu'} \}^{\mathfrak{o}}) = \{ \Psi^{\mathfrak{o}}(\delta_{\varepsilon \varepsilon'}) , \Psi^{\mathfrak{o}}(\delta_{\mu \mu'}) \}^s$ follows from the same computation that the proof of Lemma \ref{lemma_poisson_bigone}. For $\mathfrak{o}(a)=\mathfrak{o}(b)=\mathfrak{o}(c)=+$, one has: 
   \begin{eqnarray*}
   \{ N_{\alpha} \otimes N_{\gamma} \}^{\mathfrak{o}} &=& - (N_{\alpha} \otimes \mathds{1})(\frac{1}{2} H\otimes H +2E \otimes F)(\mathds{1}\otimes N_{\gamma}) \\
   &=& - \frac{1}{2} \begin{pmatrix} \alpha(+,+) & - \alpha(+,-) \\ \alpha(-,+) & -\alpha(-,-) \end{pmatrix} \otimes \begin{pmatrix} \gamma(+,+) & \gamma(+,-) \\ - \gamma(-,+) & -\gamma(-,-) \end{pmatrix} \\
   && -2 \begin{pmatrix} 0 & \alpha(+,+) \\ 0 & \alpha(-,+) \end{pmatrix} \otimes \begin{pmatrix} 0 & 0 \\ \gamma(+,+) & \gamma(+,-) \end{pmatrix}. 
   \end{eqnarray*}
   \par We recover the formulas of Example \ref{exampleTriangle}, hence one has  $(\Psi^{\mathfrak{o}})^{\otimes 2} (\{ \alpha_{\varepsilon \varepsilon'} , \gamma_{\mu \mu'} \}^{\mathfrak{o}}) = \{ \Psi^{\mathfrak{o}}(\alpha_{\varepsilon \varepsilon'}) , \Psi^{\mathfrak{o}}(\gamma_{\mu \mu'}) \}^s$. We get similar formulas by permuting cyclically the arcs $\gamma, \beta$ and $\alpha$. This proves that $\Psi^{\mathfrak{o}}$ is a Poisson morphism when $\mathfrak{o}(a)=\mathfrak{o}(b)=\mathfrak{o}(c)=+$. For another choice $\mathfrak{o}'$ of orientation of the boundary arcs, we prove that $\Psi^{\mathfrak{o}'}$ is Poisson by showing that the isomorphism $\Psi^{\mathfrak{o}'} \circ (\Psi^{\mathfrak{o}})^{-1}$ is Poisson. This follows from a similar computation than the one in the proof of Lemma \ref{lemma_poisson_bigone} by using the fact that $(C\otimes C)r^{\varepsilon} = r^{- \varepsilon}(C\otimes C)$. The fact that the two diagrams in the lemma commute follows from a straightforward computation.
   
   \end{proof}
   
   \subsubsection{The general case: proof of Theorem \ref{theorem3}}
   \par Consider a topological triangulation $\Delta$ of a punctured surface $\mathbf{\Sigma}$, together with a choice $\mathfrak{o}_{\Delta}$ of orientation of its edges. 
   Consider the following commutative diagram: 
   
   $$
\begin{tikzcd}
0 \arrow[r,""] & 
\mathcal{S}_{+1}(\mathbf{\Sigma}) \arrow[r, "i^{\Delta}"] \arrow[d, dotted, "\cong", "\exists! \Psi^{(\Delta, \mathfrak{o}_{\Delta})}"'] & 
\otimes_{\mathbb{T}} \mathcal{S}_{+1}(\mathbb{T}) \arrow[r,"\Delta^L - \sigma \circ \Delta^R"] \arrow[d, "\cong","\otimes_{\mathbb{T}}  \Psi^{\mathfrak{o}_{\mathbb{T}}}"'] &
\left( \otimes_{e} \mathcal{S}_{+1}(\mathbb{B}) \right) \otimes \left( \otimes_{\mathbb{T}} \mathcal{S}_{+1}(\mathbb{T}) \right)
\arrow[d, "\cong", "(\otimes_e \Psi^{-,+})\otimes (\otimes_{\mathbb{T}} \Psi^{\mathfrak{o}(\mathbb{T})})"'] \\
0 \arrow[r,""] & 
\mathbb{C}[\mathcal{X}_{\SL_2}(\mathbf{\Sigma})]  \arrow[r, "i^{\Delta}"]  & 
\otimes_{\mathbb{T}} \mathbb{C}[\mathcal{X}_{\SL_2}(\mathbb{T})] \arrow[r,"\Delta^L - \sigma \circ \Delta^R"] &
\left(\otimes_{e}\mathbb{C}[\SL_2]^{-,+}\right) \otimes  \left( \otimes_{\mathbb{T}} \mathbb{C}[\mathcal{X}_{\SL_2}(\mathbb{T})] \right)
\end{tikzcd}
$$
   In this diagram, both lines are exact and all morphisms are Poisson by Lemma \ref{lemma: poisson morph at +1 skein} and \cite{KojuTriangularCharVar}, hence there exists a unique Poisson isomorphism   $\Psi^{(\Delta, \mathfrak{o}_{\Delta})} : \left( \mathcal{S}_{+1}(\mathbf{\Sigma}), \{ \cdot, \cdot \}^s \right) \xrightarrow{\cong} \left( \mathbb{C}[\mathcal{X}_{\SL_2}(\mathbf{\Sigma})], \{\cdot, \cdot \}^{\mathfrak{o}} \right)$ induced by restriction of $\otimes_{\mathbb{T}} \Psi^{\mathfrak{o}(\mathbb{T})}$. This concludes the proof.

\subsection{Relative spin structures and explicit formulas}\label{sec_relativespin}
 The goal of this subsection is to give an explicit formula for the morphism $\Psi^{(\Delta, \mathfrak{o}_{\Delta})}$,  when evaluated on the generators of   $\mathcal{S}_{+1}(\mathbf{\Sigma})$. A key point is to have a global method to compute some signs that depend on the combinatorial data $(\Delta, \mathfrak{o}_{\Delta})$. We provide such a method by introducing the notion of relative spin structure, which gives a geometric interpretation these signs. We end the section by  relating the  $\Psi^{(\Delta, \mathfrak{o}_{\Delta})}$ with the morphism of \cite[Theorem 8.12]{CostantinoLe19}.
  
   \subsubsection{Relative spin structures}
  Since the classical identifications between skein algebras of closed punctured surfaces and character varieties are indexed by spin structures, it is natural to expect that the combinatorial data $(\Delta, \mathfrak{o}_{\Delta})$ indexing the isomorphism of Theorem \ref{theorem3} encode a generalization of the notion of spin structures  which would have a good behavior for the operation of gluing boundary arcs together. Before defining this notion, we introduce some notations.
  
  \begin{notations} \begin{enumerate}
  \item In this subsection, $\mathbf{\Sigma}=(\Sigma, \mathcal{P})$ will denote a triangulable punctured surface, $\mathfrak{o}$ an orientation of its boundary arcs and $(\Delta, \mathfrak{o}_{\Delta})$ a combinatorial data and we equip $\Sigma_{\mathcal{P}}$ with a Riemannian structure compatible with the orientation. For each boundary arc $a$, we fix a point $v_a\in a$. If $\partial \Sigma \neq \emptyset$, we write $\mathbb{V}:= \{v_a\}_a$ where $a$ runs through the set of boundary arcs. If $\Sigma$ is closed, we fix an arbitrarily point $v_a$ in each connected component $a$ of $\Sigma_{\mathcal{P}}$ and  write  $\mathbb{V}:= \{v_a\}_a$. 
  \item Let $\pi : U\Sigma_{\mathcal{P}} \rightarrow \Sigma_{\mathcal{P}}$ denote the unitary tangent bundle. For $\vec{v}=(v,u)\in U\Sigma_{\mathcal{P}}$, we denote by $-\vec{v}=(v,-u)$ the vector with opposite orientation. Denote by $\theta^{1/2}_{\vec{v}} : \vec{v} \rightarrow -\vec{v}$ the class in $\Pi_1(U\Sigma_{\mathcal{P}})$ of a path making a half-twist in the fiber over $\pi(\vec{v})$ in the direction given by the orientation and write $\theta_{\vec{v}}:= \theta^{1/2}_{\vec{v}}\theta^{1/2}_{-\vec{v}}$. For simplicity, for a path $\alpha : \vec{v}_1 \rightarrow \vec{v}_2$, we will write $\theta^{1/2}\alpha$ and $\alpha\theta^{1/2}$ instead of $\theta^{1/2}_{-\vec{v}_1}\alpha$ and $\alpha\theta^{1/2}_{\vec{v}_2}$ with no confusion possible.
  When $\partial \Sigma \neq \emptyset$, for each boundary arc $a$, we denote by $\vec{v}_a \in U\Sigma_{\mathcal{P}}$ the lift of $v_a$ pointing in the direction of $\mathfrak{o}$. When $\Sigma$ is closed, we fix an arbitrarily lift $\vec{v}_a$ of each $v_a$. We write $\widehat{\mathbb{V}}_+ := \{\vec{v}_a\}_a$ and $\widehat{\mathbb{V}} := \{ \vec{v}_a, - \vec{v}_a \}_a$.    
  \end{enumerate}
  
  \end{notations} 
   
  \begin{definition}\label{def_spin} A \textit{relative spin structure} on $\mathbf{\Sigma}$ is a functor 
  \\ $W \in \Hom (\Pi_1(U\Sigma_{\mathcal{P}}, \widehat{\mathbb{V}}_+), \mathbb{Z}/2\mathbb{Z})$ such that $W (\theta_{\vec{v}})=1$ for all $\vec{v} \in \widehat{\mathbb{V}}_+$. We denote by $\mathrm{Spin}(\mathbf{\Sigma})$ the set of relative spin structures on $\mathbf{\Sigma}$.
  \end{definition} 
  
  \begin{remark}\label{remark_spin1}
  When $\Sigma$ is closed and connected, an element $W\in \mathrm{Spin}(\mathbf{\Sigma})$ is a group morphism $W: \pi_1(U\Sigma_{\mathcal{P}}, \vec{v}_0^+)\rightarrow \mathbb{Z}/2\mathbb{Z}$ such that $W(\theta_{\vec{v}_0^+})=1$. Since $\mathbb{Z}/2\mathbb{Z}$ is abelian, such a morphism is equivalent to a group morphism $\underline{W}: \mathrm{H}_1(U\Sigma_{\mathcal{P}}, \mathbb{Z}/2\mathbb{Z}) \rightarrow \mathbb{Z}/2\mathbb{Z}$ satisfying $W([\theta])=1$. Such a morphism $\underline{W}$ defines a regular double covering $\widetilde{U}$ of $U\Sigma_{\mathcal{P}}$ such that the covering on each fiber is non trivial. Since $\mathrm{Spin}(2)$ is the only non-trivial double covering of $\mathrm{SO}(2)$, the space $\widetilde{U}$ is the total space of a $\mathrm{Spin}(2)$ fiber bundle over $\Sigma_{\mathcal{P}}$ lifting the bundle of orthogonal frames induced by the metric, hence it defines a spin structure. There is actually a $1$ to $1$ correspondence between isomorphism classes of spin structures and such morphism $\underline{W}$ (see \cite{MilnorSpinStructure} for details). Therefore a relative spin structure is the same than a  "standard" spin structure in the closed case. When the surface has non empty boundary, an element $W\in \mathrm{Spin}(\mathbf{\Sigma})$ still induces a group morphism $\underline{W}$ thus a spin structure. However the functor $W$ contains more information than $\underline{W}$ which permits to "glue" relative spin structures together.
  \end{remark}
  
  Let $a$ and $b$ be two distinct boundary arcs of $\mathbf{\Sigma}$ and denote by $p: \Sigma_{\mathcal{P}} \rightarrow {\Sigma_{\mathcal{P}}}_{|a\#b}$ the projection. Write $c:=p(a)=p(b)$. We assume that $(1)$ the restriction $p: \Sigma_{\mathcal{P}}\setminus (a\cup b) \to {\Sigma_{\mathcal{P}}}_{|a\#b} \setminus c$ is an isometry, $(2)$ the restriction $p: a\to c$ and $p:b\to c$ are isometries and $(3)$  that the orientations  $\mathfrak{o}$ of $a$ and $b$ coincide when gluing the arcs and that $p(v_a)=p(v_b)=:v_c$. The projection induces a lift $\vec{v}_c \in U {\Sigma_{\mathcal{P}}}_{|a\#b}$ of $v_c$ and a functor
  $$ p_* : \Pi_1(U\Sigma_{\mathcal{P}}, \widehat{\mathbb{V}}_+) \rightarrow \Pi_1(U{\Sigma_{\mathcal{P}}}_{|a\#b}, \widehat{\mathbb{V}}^{a\#b}_+ \cup \{\vec{v}_c\} ).$$   
   \begin{lemma} For $W\in \mathrm{Spin}(\Sigma)$, there exists a unique $W_{|a\#b} \in \mathrm{Spin}(\Sigma_{|a\#b})$ such that $W_{|a\#b}(p_*(\alpha))=W(\alpha)$ for all $\alpha \in \Pi_1(U\Sigma_{\mathcal{P}}, \widehat{\mathbb{V}}_+)$.
   \end{lemma}
   
   \begin{proof} Note that the image of $p_*$ generates the groupoid $\Pi_1(U{\Sigma_{\mathcal{P}}}_{|a\#b}, \widehat{\mathbb{V}}_+^{a\#b} \cup \{\vec{v}_c\})$ in the sense that any path $\alpha \in \Pi_1(U{\Sigma_{\mathcal{P}}}_{|a\#b}, \widehat{\mathbb{V}}_+^{a\#b} \cup \{\vec{v}_c\})$ can be written as a composition $\alpha = p_*(\alpha_1)\ldots p_*(\alpha_n)$ for some $\alpha_i \in \Pi_1(U\Sigma_{\mathcal{P}}, \widehat{\mathbb{V}}_+)$. Hence for $W\in \mathrm{Spin}(\mathbf{\Sigma})$, there exists a unique functor $\widetilde{W} : \Pi_1(U{\Sigma_{\mathcal{P}}}_{|a\#b}, \widehat{\mathbb{V}}_+^{a\#b} \cup \{\vec{v}_c\}) \rightarrow \mathbb{Z}/2\mathbb{Z}$ such that $\widetilde{W}(\pi_*(\alpha))=W(\alpha)$ for all $\alpha \in \Pi_1(U\Sigma_{\mathcal{P}}, \widehat{\mathbb{V}}_+)$, and $W_{|a\#b}$ has to be the restriction of $\widetilde{W}$ to the full subcategory $\Pi_1(U{\Sigma_{\mathcal{P}}}_{|a\#b}, \widehat{\mathbb{V}}_+^{a\#b})$.
   \end{proof}
   
    Note that the map $r_{a\#b}: \mathrm{Spin}(\mathbf{\Sigma}) \rightarrow \mathrm{Spin}(\mathbf{\Sigma}_{|a\#b})$ sending $W$ to $W_{|a\#b}$ is surjective but not injective. Indeed when lifting a functor in $\Hom(\Pi_1(U\Sigma_{\mathcal{P}}, \widehat{\mathbb{V}}_+), \mathbb{Z}/2\mathbb{Z})$ to a functor in $\Hom(\Pi_1(U\Sigma_{\mathcal{P}}, \widehat{\mathbb{V}}_+\cup \{ \overrightarrow{v}_c\}),  \mathbb{Z}/2\mathbb{Z})$ there is a $\mathbb{Z}/2\mathbb{Z}$ ambiguity.
     Note also that if $a,b,c,d$ are four distinct boundary arcs, one obviously has $r_{a\#b}\circ r_{c\#d} = r_{c\#d}\circ r_{a\#b}$. In particular, once some combinatorial data $(\Delta, \mathfrak{o}_{\Delta})$ of $\mathbf{\Sigma}$ are fixed, any relative spin structure on $\mathbf{\Sigma}$ can be obtained by gluing some relative spin structure on each face of the triangulation.
   
   \subsubsection{Lifts of embedded curves and the function $w$}
   
   Let us call  \textit{embedded arc} a smooth embedding $\alpha : [0,1] \rightarrow \Sigma_{\mathcal{P}}$ such that $\alpha(0), \alpha(1) \in \partial \Sigma_{\mathcal{P}}$. To any embedded arc and any simple closed curve, we associate two lifts in $U\Sigma_{\mathcal{P}}$ as follows.
   \par  For $\alpha$ an embedded arc oriented from the boundary arc $a$ to the boundary arc $b$, we isotope $\alpha$ (in the class of embedded arc)  such that $\alpha(0)=v_a, \alpha(1)=v_b$,  the vectors $\alpha'(0)$ and $\alpha'(1)$ are  tangent to $a$ and $b$ and such that $\alpha'(0)$ points in the direction of $a$ opposite to the orientation  induced by the orientation of $\Sigma_{\mathcal{P}}$ and such that $\alpha'(1)$ points in the direction of $b$ induced by the orientation of $\Sigma_{\mathcal{P}}$. The \textit{positive lift} of $\alpha$ is the homotopy class $\widehat{\alpha}^+\in \Pi_1(U\Sigma_{\mathcal{P}}, \widehat{\mathbb{V}})$ of the continuous map $t\mapsto \left(\alpha(t), \frac{\alpha'(t)}{\lVert \alpha'(t) \rVert}\right)$. 
   \par For $v$ a point in a boundary arc $a$, we write $\mathfrak{o}(v)=0$ if the orientation of $a$ agrees with the induced orientation of $\Sigma_{\mathcal{P}}$ and $\mathfrak{o}(v)=1$ else. The $\mathfrak{o}$-\textit{lift} $\widehat{\alpha}^{\mathfrak{o}}\in \Pi_1(U\Sigma_{\mathcal{P}}, \widehat{\mathbb{V}}_+)$ is defined by the formula
   \begin{equation}\label{eq_lifts}
   \widehat{\alpha}^+ = (\theta^{1/2})^{1-\mathfrak{o}(s(\alpha))} \widehat{\alpha}^{\mathfrak{o}} (\theta^{1/2})^{\mathfrak{o}(t(\alpha))}.
   \end{equation}
    \par Let $\gamma$ be a smooth embedded curve and $v\in \mathbb{V}$.  We define $\widehat{\gamma}_v^+ $ as the as the homotopy class of a map  $ t\mapsto \left(\beta(t), \frac{\beta'(t)}{\lVert \beta'(t) \rVert}\right)$  where $\beta$ is a smooth immersion  $\beta : [0,1] \rightarrow \Sigma_{\mathcal{P}}$  which is isotopic to $\gamma$ such that $\beta(0)=v=\beta(1)$ and $\beta'(0)$ points in the direction induced by the orientation of the surface for $\widehat{\gamma}_v^+ $. Similarly, we define $\widehat{\gamma}_v^{\mathfrak{o}}$ as the homotopy class of a map  $ t\mapsto \left(\beta(t), \frac{\beta'(t)}{\lVert \beta'(t) \rVert}\right)$ where this time $\beta'(0)$ points in the direction of $\mathfrak{o}$ for $\widehat{\gamma}_v^{\mathfrak{o}}$.
     If $\Sigma$ is closed and $\gamma$ is in a connected component $b$, we impose that $\widehat{\gamma}_v^+ =\widehat{\gamma}_v^{\mathfrak{o}}$ is defined from an immersion $\beta$ such that $(\beta(0), \beta'(0))=v_b$.
   
   \begin{notations}\label{def_w} For $W\in \mathrm{Spin}(\mathbf{\Sigma})$ and $\alpha$ an embedded arc, we write $w(\alpha):= W(\widehat{\alpha}^{\mathfrak{o}})\in \mathbb{Z}/2\mathbb{Z}$. For $\gamma$ a closed curve we write $w(\gamma):= W(\widehat{\gamma}_v^{\mathfrak{o}})$.
   \end{notations}
   
   \begin{remark}\label{remark_Johnson}
   The value $w(\gamma)$ associated to a closed curve is obviously independent on the choice of the point $v$. Moreover, as noted in Remark \ref{remark_spin1}, the value $W(\widehat{\gamma})$ only depends on the homology class $[\widehat{\gamma}^{\mathfrak{o}}] \in \mathrm{H}_1(U\Sigma_{\mathcal{P}}; \mathbb{Z}/2\mathbb{Z})$ and is closely related to the Johnson quadratic form as follows. Let $\{\gamma_i\}_{i=1, \ldots, n}$ be a collection of simple closed curves. Johnson proved in \cite[Theorem 1.A]{Johnson_SpinStructures} that the class $y:= \sum_{i=1}^n [\widehat{\gamma}_i^{\mathfrak{o}}] +n[\theta] \in \mathrm{H}_1(U\Sigma_{\mathcal{P}}; \mathbb{Z}/2\mathbb{Z})$ only depends on the homology class of $x:=\sum_{i=1}^n[\gamma_i] \in \mathrm{H}_1(\Sigma_{\mathcal{P}}; \mathbb{Z}/2\mathbb{Z})$ hence that the assignation $x\mapsto y$ defines a map (not a morphism) $\mathrm{H}_1(\Sigma_{\mathcal{P}}; \mathbb{Z}/2\mathbb{Z})\rightarrow \mathrm{H}_1(U\Sigma_{\mathcal{P}}; \mathbb{Z}/2\mathbb{Z})$. Moreover, for a (relative) spin structure $W$, Johnson proved in \cite[Theorem 1.B]{Johnson_SpinStructures} that the map $\omega :  \mathrm{H}_1(\Sigma_{\mathcal{P}}; \mathbb{Z}/2\mathbb{Z}) \rightarrow \mathbb{Z}/2\mathbb{Z}$ defined by $\omega(\sum_{i=1}^n [\gamma_i]):= n + \sum_{i=1}^n w([\gamma_i]) \pmod{2}$ satisfies the relation
   $$ \omega([\alpha +\beta]) = \omega([\alpha]) + \omega([\beta]) + \left< [\alpha], [\beta]\right>, $$
   hence that $\omega$ is a quadratic form for $(\mathrm{H}_1(\Sigma_{\mathcal{P}}; \mathbb{Z}/2\mathbb{Z}), \left<\cdot, \cdot \right>)$, where $\left<\cdot, \cdot \right>$ represents the intersection form. Hence the values $w(\gamma)$ in Definition \ref{def_w} are related to the Johnson quadratic form of the underlying spin structure by $\omega([\gamma])= w(\gamma)+1 \pmod{2}$.
   \end{remark}

\subsubsection{Relative spin structures associated to combinatorial data}

In order to  assign a relative spin structure to some combinatorial data $(\Delta, \mathfrak{o}_{\Delta})$ in a canonical way, we need to assign to each triangle $\mathbb{T}$, equipped with an orientation $\mathfrak{o}_{\mathbb{T}}$ of its boundary arcs, a canonical relative spin structure and then gluing the triangles along their faces. Let $\alpha, \beta, \gamma$ be the three paths in Figure \ref{figtriangle} which generate the groupoid $\Pi_1(\mathbb{T}, \mathbb{V})$ with relation $\gamma\beta\alpha=1$. Note that for any choice of $\mathfrak{o}_{\mathbb{T}}$, one has the relation $\widehat{\gamma}^{\mathfrak{o}_{\mathbb{T}}}\widehat{\beta}^{\mathfrak{o}_{\mathbb{T}}}\widehat{\alpha}^{\mathfrak{o}_{\mathbb{T}}}=\theta^{-2}$. Hence a relative spin structure $W$ on $\mathbb{T}$ is described by three elements $W(\widehat{\alpha}^{\mathfrak{o}_{\mathbb{T}}}), W(\widehat{\beta}^{\mathfrak{o}_{\mathbb{T}}}), W(\widehat{\gamma}^{\mathfrak{o}_{\mathbb{T}}}) \in \mathbb{Z}/2\mathbb{Z}$ such that $W(\widehat{\alpha}^{\mathfrak{o}_{\mathbb{T}}}) +W(\widehat{\beta}^{\mathfrak{o}_{\mathbb{T}}})+W(\widehat{\gamma}^{\mathfrak{o}_{\mathbb{T}}}) =0$. Therefore there exist four different relative spin structures on $\mathbb{T}$.

\begin{definition}\label{def_combinatoric}
The \textit{distinguished} relative spin structure on $\mathbb{T}$ is the relative spin structure $W$ such that $W(\widehat{\alpha}^{\mathfrak{o}_{\mathbb{T}}}) =W(\widehat{\beta}^{\mathfrak{o}_{\mathbb{T}}})=W(\widehat{\gamma}^{\mathfrak{o}_{\mathbb{T}}}) =0$. For $\mathbf{\Sigma}$ a punctured surface with combinatorial data $(\Delta, \mathfrak{o}_{\Delta})$, we associate a relative spin structure $W^{(\Delta, \mathfrak{o}_{\Delta})} \in \mathrm{Spin}(\mathbf{\Sigma})$ by gluing together the distinguished spin structures on the faces of the triangulation.
\end{definition}

Note that the distinguished relative spin structure $W$ on $\mathbb{T}$ satisfies $w(\alpha)=w(\beta)=w(\gamma)=0$ and $w(\alpha^{-1})=w(\beta^{-1})=w(\gamma^{-1})=1$.

\begin{remark} Since we associate to each face a specific (named distinguished) relative spin structure, there is no reason to believe that every spin structure on $\Sigma_{\mathcal{P}}$ can be associated to some combinatorial data. Moreover we will not investigate under which condition two combinatorial data induce the same relative spin structure. In \cite{NovakRunkel}, Novak and Runkel showed that any spin structure on a surface can be encoded by the combinatorial data consisting in a triangulation (with no degenerate face), an orientation of the edges and a choice of distinguished vertex for each face. Moreover they proved that two such combinatorial data induce the same spin structure if and only if they can be related by a sequence of elementary moves. It would be interesting to compare their approach to Definition \ref{def_combinatoric}.

\end{remark}

We now state an explicit formula for the values $w(\alpha)$ associated to a relative spin structure $W^{(\Delta, \mathfrak{o}_{\Delta})}$. For each edge $e\in \mathcal{E}(\Delta)$, fix a point $v_e\in e$ and let $\mathbb{V}^{\Delta}=\{v_e\}_{e\in \mathcal{E}(\Delta)}$. When $\partial \Sigma \neq \emptyset$, we assume that $\mathbb{V}^{\Delta}\cap \partial \Sigma_{\mathcal{P}}= \mathbb{V}$. When $\Sigma$ is closed, we assume that $\mathbb{V} \subset \mathbb{V}^{\Delta}$. Let $\vec{v}_e \in U\Sigma_{\mathcal{P}}$ be the lift of $v_e$ oriented in the direction of $\mathfrak{o}_{\Delta}$ and set $\widehat{\mathbb{V}}^{\Delta}_+ := \{\vec{v}_e; e\in \mathcal{E}(\Delta)\}$ and $\widehat{\mathbb{V}}^{\Delta} := \{\vec{v}_e, -\vec{v}_e; e\in \mathcal{E}(\Delta)\}$. Note that the set $\widehat{\mathbb{G}}^{\Delta}:= \{ (\widehat{\alpha}^{\mathfrak{o}}_{\mathbb{T}})^{\pm 1}, (\widehat{\beta}^{\mathfrak{o}}_{\mathbb{T}})^{\pm 1}, (\widehat{\gamma}^{\mathfrak{o}}_{\mathbb{T}})^{\pm 1}; \mathbb{T}\in F(\Delta)\}$ generates the groupoid $\Pi_1(U\Sigma_{\mathcal{P}}, \widehat{\mathbb{V}}_+^{\Delta})$. By definition of the gluing operation, the functor $W^{(\Delta, \mathfrak{o}_{\Delta})}$ is the restriction of the functor $\widetilde{W} \in \Hom(\Pi_1(U\Sigma_{\mathcal{P}}, \widehat{\mathbb{V}}_+^{\Delta}), \mathbb{Z}/2\mathbb{Z})$ characterized by $\widetilde{W}(\widehat{\alpha}_{\mathbb{T}}^{\mathfrak{o}_{\mathbb{T}}})= \widetilde{W}(\widehat{\beta}_{\mathbb{T}}^{\mathfrak{o}_{\mathbb{T}}})=\widetilde{W}(\widehat{\gamma}_{\mathbb{T}}^{\mathfrak{o}_{\mathbb{T}}})=0$ for every face $\mathbb{T}$ and $\widetilde{W}(\theta_{\vec{v}})=1$ for any $\vec{v}\in \widehat{\mathbb{V}}_+^{\Delta}$. Set $\mathbb{G}^{\Delta}:= \pi (\widehat{\mathbb{G}}^{\Delta}_+)=\{ \alpha_{\mathbb{T}}^{\pm 1}, \beta_{\mathbb{T}}^{\pm 1}, \gamma_{\mathbb{T}}^{\pm 1}; \mathbb{T}\in F(\Delta)\}$ and for $\delta \in \mathbb{G}^{\Delta}$ a path in $\mathbb{T}$, write $w(\delta):=\widetilde{W}(\widehat{\delta}^{\mathfrak{o}_{\mathbb{T}}})$. Hence $w(\delta)=0$ if $\delta=\alpha_{\mathbb{T}}, \beta_{\mathbb{T}}$ or $\gamma_{\mathbb{T}}$ and $w(\delta)=1$ if $\delta=\alpha_{\mathbb{T}}^{-1}, \beta_{\mathbb{T}}^{-1}$ or $\gamma_{\mathbb{T}}^{-1}$.
\vspace{2mm}
\par Let $\alpha$ be either an embedded arc or a closed curve and choose a decomposition
\begin{equation}\label{eq_decomp_arc}
\alpha = \alpha_1 \ldots \alpha_n \quad , \alpha_i \in \mathbb{G}^{\Delta},
\end{equation}
such that either $\alpha_i$ and $\alpha_{i+1}$ lies in different faces $\mathbb{T}_i \neq \mathbb{T}_{i+1}$  of $\Delta$, or $\mathbb{T}_i=\mathbb{T}_{i+1}$ is a degenerate triangle, with two boundary arcs glued together to give an arc $c$ in $\Sigma_{\mathcal{P}}$, and $\alpha_i\alpha_{i+1}$ crosses $c=t(\alpha_i)=s(\alpha_{i+1})$ transversally. In the above statement, the indices $i$ are taken in $\mathbb{Z}/n\mathbb{Z}$ when $\alpha$ is a closed curve. Note that such a decomposition is obtained by isotoping $\alpha$ transversally with minimal intersection to the edges of the triangulation, and then cutting $\alpha$ along the edges. For $(\mathbb{T}, \mathfrak{o}_{\mathbb{T}})$ a triangle with oriented edges,  $a$ an edge  and $v_a \in a$, recall that we write $\mathfrak{o}_{\mathbb{T}}(v_a)=0$ if the orientation of $a$ corresponds to the orientation induced by the orientation of $\mathbb{T}$ and write $\mathfrak{o}_{\mathbb{T}}(v_a)=+1$ else.

\begin{lemma}\label{lemma_spinw}
The function $w$ associated to the relative spin structure $W^{(\Delta, \mathfrak{o}_{\Delta})}$ is characterized by the formula
$$ w(\alpha) = \left\{ \begin{array}{ll} \sum_{i=1}^n w(\alpha_i) + \sum_{i=1}^{n-1} \mathfrak{o}_{\mathbb{T}_i}(t(\alpha_i)) \pmod{2} & \mbox{, if } \alpha \mbox{ is an embedded arc;} \\
\sum_{i=1}^n w(\alpha_i) + \sum_{i=1}^{n} \mathfrak{o}_{\mathbb{T}_i}(t(\alpha_i)) \pmod{2} & \mbox{, if } \alpha \mbox{ is a closed curve.} \end{array} \right.$$
\end{lemma}

\begin{proof}
First note that for the positive lifts, one has the equality
$$ \widehat{\alpha}^+ = \widehat{\alpha}_1^+ \ldots \widehat{\alpha}_n^+.$$
This equality follows from the fact that the embedded curve chosen to represent $ \widehat{\alpha}^+ $ can be isotoped such that it crosses tangentially the edges of $\Delta$ in such a way that, when cutting along the edges, one obtains the composition $\widehat{\alpha}_1^+ \ldots \widehat{\alpha}_n^+.$ Note also that this equality is essentially \cite[Proposition $8.11$]{CostantinoLe19}. Recall from Equation \eqref{eq_lifts} that $\widehat{\alpha_i}^+ = (\theta^{1/2})^{1-\mathfrak{o}(s(\alpha_i))} \widehat{\alpha}_i^{\mathfrak{o}} (\theta^{1/2})^{\mathfrak{o}(t(\alpha_i))}$ and note that, since we assume that the faces $\mathbb{T}_i$ and $\mathbb{T}_{i+1}$ are distinct, one has 
$$(1-\mathfrak{o}_{\mathbb{T}_i}(t(\alpha_i))) + \mathfrak{o}_{\mathbb{T}_{i+1}}(s(\alpha_{i+1}))= 2 \mathfrak{o}_{\mathbb{T}_{i+1}}(s(\alpha_i)),$$
 (where indices are understood modulo $n$ when $\alpha$ is a closed curve). When $\alpha$ is an arc,  we thus obtain the equality
 $$ \widehat{\alpha}_1^{\mathfrak{o}_{\mathbb{T}_1}} \ldots  \widehat{\alpha}_n^{\mathfrak{o}_{\mathbb{T}_n}} = \theta^{\sum_{i=1}^{n-1}\mathfrak{o}_{\mathbb{T}_i}(t(\alpha_i))} (\theta^{1/2})^{1-\mathfrak{o}(s(\alpha))} \widehat{\alpha}^+ (\theta^{1/2})^{\mathfrak{o}(t(\alpha))},$$
from which we deduce that
\begin{align*}
w(\alpha) := W(\widehat{\alpha}^{\mathfrak{o}}) &=& W \left( (\theta^{-1/2})^{1-\mathfrak{o}(s(\alpha))} \widehat{\alpha}^+ (\theta^{-1/2})^{\mathfrak{o}(t(\alpha))} \right) \\
&=& W\left( \theta^{-\sum_{i=1}^{n-1}\mathfrak{o}_{\mathbb{T}_i}(t(\alpha_i))} \widehat{\alpha}_1^{\mathfrak{o}_{\mathbb{T}_1}}\ldots \widehat{\alpha}_n^{\mathfrak{o}_{\mathbb{T}_n}} \right) \\
&=& \sum_{i=1}^{n-1} \mathfrak{o}_{\mathbb{T}_i}(t(\alpha_i)) + \sum_{i=1}^n w(\alpha_i) \pmod{2}.
\end{align*}
The computation when $\alpha$ is a closed curve is done in the same manner.

\end{proof}

\subsubsection{Explicit formulas for the isomorphism}
   
 In order to describe the isomorphism $\Psi^{(\Delta, \mathfrak{o}_{\Delta})}$ of Theorem \ref{theorem3} more explicitly, let us recall from \cite{KojuTriangularCharVar} a set of generators for the ring of regular functions of the relative character varieties.
    For $\alpha$ an embedded arc, seen as a path in the fundamental groupoid, and $\varepsilon, \varepsilon' = \pm$, the regular function $F_{\alpha_{\varepsilon \varepsilon'}} \in \mathbb{C}[\mathcal{X}_{\SL_2}(\mathbf{\Sigma})]$ is defined on the class $[\rho]$ of a functor $\rho \in \mathcal{R}_{\SL_2}(\mathbf{\Sigma}_{\mathcal{P}})$ by $\rho(\alpha)= \begin{pmatrix} F_{\alpha_{++}}(\rho) & F_{\alpha_{+-}}(\rho) \\ F_{\alpha_{-+}}(\rho) & F_{\alpha_{--}}(\rho) \end{pmatrix}$. For $\gamma$ a closed curve, represented by an arbitrary path $\gamma_v\in \Pi_1(\Sigma_{\mathcal{P}}, \mathbb{V})$, one defines $F_{\gamma}\in \mathbb{C}[\mathcal{X}_{\SL_2}(\mathbf{\Sigma})]$ by $F_{\gamma}([\rho]):= \tr (\rho(\gamma_v))$. Since the trace is invariant by conjugacy, the value $\tr (\rho(\gamma_v))$ does not depend on the choice of base point $v$ nor on the representative $\rho$ in the class $[\rho]$. The functions $F_{\alpha_{\varepsilon \varepsilon'}}$ and $F_{\gamma}$ generate the algebra $\mathbb{C}[ \mathcal{X}_{\SL_2}(\mathbf{\Sigma})]$. For $\alpha$ an arc, we set $N_{\alpha}:= \begin{pmatrix} F_{\alpha_{++}} & F_{\alpha_{+-}} \\ F_{\alpha_{-+}} & F_{\alpha_{--}} \end{pmatrix}$ the $2\times 2$ matrix with coefficients in $\mathbb{C}[\mathcal{X}_{\SL_2}(\mathbf{\Sigma})]$. Note that $N_{\alpha^{-1}}=  \begin{pmatrix} F_{\alpha_{--}} & -F_{\alpha_{+-}} \\ -F_{\alpha_{-+}} & F_{\alpha_{++}} \end{pmatrix}$.

   \vspace{2mm}
   \par For $\alpha$ an embedded arc and $\varepsilon, \varepsilon' = \pm$, we denote by $\alpha_{\varepsilon \varepsilon'}\in \mathcal{S}_{+1}(\mathbf{\Sigma})$ the class of the arc $\alpha$ with state $\varepsilon$ at $s(\alpha)$ and $\varepsilon'$ at $t(\alpha)$. We write $M_{\alpha}:= \begin{pmatrix} \alpha_{++} & \alpha_{+-} \\ \alpha_{-+} & \alpha_{--}\end{pmatrix}$ the $2\times 2$ matrix with coefficients in $\mathcal{S}_{+1}(\mathbf{\Sigma})$. Note that $M_{\alpha^{-1}}= (M_{\alpha})^{\intercal}=\begin{pmatrix} \alpha_{++} & \alpha_{-+} \\ \alpha_{+-} & \alpha_{--}\end{pmatrix}$. Recall the isomorphism $\Psi^{(\Delta, \mathfrak{o}_{\Delta})}$ of Theorem \ref{theorem3} and recall that $C^{-1}=\begin{pmatrix} 0 & -1 \\ 1 & 0 \end{pmatrix}$. 
   
   \begin{theorem}\label{theorem5}
    For each  embedded arc $\alpha$, one has   
   		\begin{equation}\label{eq_F1}
   		\Psi^{(\Delta, \mathfrak{o}_{\Delta})} (M_{\alpha}) = (-1)^{w(\alpha)} (C^{-1})^{1-\mathfrak{o}(\alpha(0))} N_{\alpha} (C^{-1})^{\mathfrak{o}(\alpha(1))}.  
   		\end{equation}
   		For each closed curve   $\gamma$, one has 
   		\begin{equation}\label{eq_F2}
   		\Psi^{(\Delta, \mathfrak{o}_{\Delta})} (\gamma) = (-1)^{w(\gamma)} F_{\gamma}. 
   		\end{equation}   
   \end{theorem}
   
   \begin{remark} When $\Sigma$ is closed, recall from Remarks \ref{remark_spin1}, \ref{remark_Johnson} that $W^{(\Delta, \mathfrak{o}_{\Delta})}$ is a standard spin structure associated to a quadratic form $\omega$ such that $w(\gamma)=\omega([\gamma])+1$. Hence in the closed case, the isomorphism $\Psi^{(\Delta, \mathfrak{o}_{\Delta})}$ coincides with the "standard" isomorphisms described at the beginning of Section $3.1$.
   \end{remark}
   
   Recall that $\Psi^{(\Delta, \mathfrak{o}_{\Delta})}$ is defined by the following diagram
   
 \begin{equation}\label{eq_Diag}
\begin{tikzcd}
\mathcal{S}_{+1}(\mathbf{\Sigma}) \arrow[r, hook, "i^{\Delta}"] \arrow[d, "\cong", "\Psi^{(\Delta, \mathfrak{o}_{\Delta})}"'] & 
\otimes_{\mathbb{T}} \mathcal{S}_{+1}(\mathbb{T})\arrow[d, "\cong","\otimes_{\mathbb{T}}  \Psi^{\mathfrak{o}_{\mathbb{T}}}"']  \\
\mathbb{C}[\mathcal{X}_{\SL_2}(\mathbf{\Sigma})]  \arrow[r, "i^{\Delta}"]  & 
\otimes_{\mathbb{T}} \mathbb{C}[\mathcal{X}_{\SL_2}(\mathbb{T})]
\end{tikzcd}
\end{equation}
   For $x\in  \mathcal{S}_{+1}(\mathbb{T})$, we still denote by $x$ the element in $\otimes_{\mathbb{T}} \mathcal{S}_{+1}(\mathbb{T})$ having $1$ in the factors $\mathcal{S}_{+1}(\mathbb{T}')$ for $\mathbb{T}'\neq \mathbb{T}$ and $x$ in the factor $\mathcal{S}_{+1}(\mathbb{T})$. Hence for $\delta\in \mathbb{G}^{\Delta}$ a path in $\mathbb{T}$, the matrix $M_{\delta}$ is considered as a $2\times 2$ matrix with coefficients in $\otimes_{\mathbb{T}} \mathcal{S}_{+1}(\mathbb{T})$. Similarly,  the matrix $N_{\delta}$ is considered as a $2\times 2$ matrix with coefficients in $\otimes_{\mathbb{T}} \mathbb{C}[\mathcal{X}_{\SL_2}(\mathbb{T})]$.
   
   \begin{proof}
   We first show that if Equation \eqref{eq_F1} holds for an arc $\alpha$, then it holds for $\alpha^{-1}$. This fact follows from the facts that $w(\alpha^{-1})= w(\alpha)+1$, from the equalities $(C^{-1})^{\intercal}=C$ and  $A^{-1}=-C^{-1}A^{\intercal}C^{-1}$ for $A \in \SL_2(\mathbb{C})$ and the following computation:
   \begin{eqnarray*}
   \Psi (M_{\alpha^{-1}}) &=& \Psi(M_{\alpha}^{\intercal}) = (-1)^{w(\alpha)} C^{\mathfrak{o}(t(\alpha))} (N_{\alpha})^{\intercal} C^{1-\mathfrak{o}(s(\alpha))} \\
   &=& (-1)^{w(\alpha)+1} (C^{-1})^{1-\mathfrak{o}(s(\alpha^{-1}))} \left( -C^{-1}N_{\alpha}^{\intercal} C^{-1}\right) (C^{-1})^{\mathfrak{o}(t(\alpha^{-1}))} \\
    &=& (-1)^{w(\alpha^{-1})} (C^{-1})^{1-\mathfrak{o}(s(\alpha^{-1}))} N_{\alpha^{-1}} (C^{-1})^{\mathfrak{o}(t(\alpha^{-1}))}.
    \end{eqnarray*}
  Next let us prove the theorem for the triangle $\mathbb{T}$. The fact that Equation  \eqref{eq_F1} holds for the arcs  $\alpha_{\mathbb{T}}, \beta_{\mathbb{T}}$ and $\gamma_{\mathbb{T}}$ is an immediate consequence of the definition of $\Psi^{\mathfrak{o}_{\mathbb{T}}}$ in Lemma  \ref{lemma_poisson_triangle} and from the definition of the canonical spin structure in $\mathbb{T}$. By the preceding arguments, Equation \eqref{eq_F1} also holds for the arcs  $\alpha_{\mathbb{T}}^{-1}, \beta_{\mathbb{T}}^{-1}, \gamma_{\mathbb{T}}^{-1}$ and the theorem is proved for $\mathbb{T}$.
  \vspace{2mm}
  \par In the general case, consider an arc $\alpha$ and choose a decomposition
  $$ \alpha = \alpha_1 \ldots, \alpha_n \quad, \alpha_i \in \mathbb{G}^{\Delta}$$
   as before. By the gluing formula for stated skein algebras (\cite[Theorem $3.1$]{LeStatedSkein}), one has $i^{\Delta}(M_{\alpha})=M_{\alpha_1}\ldots M_{\alpha_{n}}$. By definition of the morphism $i$ in Equation \eqref{eqKK}, one has $i^{\Delta}(N_{\alpha})= N_{\alpha_1}\ldots N_{\alpha_n}$. By the preceding case of the triangle, one has 
   $$(\otimes_{\mathbb{T}}\Psi^{\mathfrak{o}_{\mathbb{T}}})(M_{\alpha_i})= (-1)^{w(\alpha_i)} (C^{-1})^{1-\mathfrak{o}_{\mathbb{T}_i}(s(\alpha_i))} N_{\alpha_i} (C^{-1})^{\mathfrak{o}_{\mathbb{T}_i}(t(\alpha)}.$$
   Hence by Lemma \ref{lemma_spinw} one has
   $$(\otimes_{\mathbb{T}}\Psi^{\mathfrak{o}_{\mathbb{T}}})\circ i^{\Delta} (M_{\alpha})= i^{\Delta}\left( (-1)^{w(\alpha)} (C^{-1})^{1-\mathfrak{o}(s(\alpha))} N_{\alpha} (C^{-1})^{\mathfrak{o}(t(\alpha))} \right), $$
   and Equation \eqref{eq_F1} follows from the commutativity of the diagram \eqref{eq_Diag}. The proof for a closed curved is done similarly by taking the trace of the above equality.
   \end{proof}
 
 \subsubsection{Comparison with Costantino-L\^e's isomorphism}\label{sec_comparaisonCL}
   
   Let $\mathbf{\Sigma}$ be a connected punctured surface with non-trivial boundary. Costantino and L\^e defined in \cite{CostantinoLe19} the twisted character variety $\chi(\mathbf{\Sigma})$ as the space of functors $\widehat{\rho} \in \Hom(\Pi_1(U\Sigma_{\mathcal{P}}, \widehat{\mathbb{V}}), \SL_2(\mathbb{C}))$ such that $\widehat{\rho}(\theta^{1/2}_{\vec{v}})=C^{-1}$ for any $\vec{v} \in \widehat{\mathbb{V}}$. Let $\mathscr{S}$ denote the maximal spectrum of $\mathcal{S}_{+1}(\mathbf{\Sigma})$. For $\chi \in \mathscr{S}$, seen as a character $\chi : \mathcal{S}_{+1}(\mathbf{\Sigma}) \rightarrow \mathbb{C}^*$, and for $\alpha$ an oriented arc, write $\chi(\alpha):=\begin{pmatrix} \chi(\alpha_{++}) & \chi(\alpha_{+-}) \\ \chi(\alpha_{+-}) & \chi(\alpha_{--}) \end{pmatrix}$. In \cite[Theorem $8.12$]{CostantinoLe19}, the authors defined an affine isomorphism $\Theta : \mathscr{S} \xrightarrow{\cong} \chi(\mathbf{\Sigma})$ sending a character $\chi$ to a functor $\widehat{\rho}$ such that $\chi(\alpha)=\widehat{\rho}(\widehat{\alpha}^+)$ for any embedded (even immersed) arc and such that $\chi(\gamma)=\tr (\widehat{\rho}(\widehat{\gamma}^+))$ for any closed curve. Composing $\Theta$ with the isomorphism induced by $\Psi^{(\Delta, \mathfrak{o}_{\Delta})}$, one obtains an isomorphism $\mathcal{X}_{\SL_2}(\mathbf{\Sigma}) \cong \chi(\mathbf{\Sigma})$. By Theorem \ref{theorem5}, this isomorphism sends a functor $\rho \in \Hom(\Pi_1(\Sigma_{\mathcal{P}}, \mathbb{V}), \SL_2(\mathbb{C}))$ to a functor $\widehat{\rho} \in \Hom(\Pi_1(U\Sigma_{\mathcal{P}}, \widehat{\mathbb{V}}), \SL_2(\mathbb{C}))$ characterized by the formulas $\widehat{\rho}(\widehat{\alpha}^{\mathfrak{o}})= (-1)^{w(\alpha)} \rho(\alpha)$ for any arc $\alpha$, $\tr(\widehat{\rho}(\widehat{\gamma}^{\mathfrak{o}})= (-1)^{w(\gamma)}\tr(\rho(\gamma))$ for any closed curve $\gamma$ and $\widehat{\rho}(\theta^{1/2}_{\vec{v}})=C^{-1}$ for any $\vec{v} \in \widehat{\mathbb{V}}$.

   \subsection{Classical Shadows}
   
   \par Suppose that $\omega\in \mathbb{C}$ is a root of unity of odd order $N>1$. A \textit{central representation} of the stated skein algebra is a finite dimensional representation $r: \mathcal{S}_{\omega}(\mathbf{\Sigma}) \rightarrow \End(V)$ which sends each element of the image of the morphism $j$ of Theorem \ref{theorem2} to scalar operators. Fix a topological triangulation $\Delta$ of $\mathbf{\Sigma}$ and an orientation $\mathfrak{o}_{\Delta}$ of its edges. Then $r$ induces a character on $\mathcal{S}_{+1}(\mathbf{\Sigma})\xrightarrow[\cong]{\Psi^{(\Delta, \mathfrak{o}_{\Delta})}} \mathbb{C}[\mathcal{X}_{\SL_2}(\mathbf{\Sigma})]$ and this character induces a point in the relative character variety $\mathcal{X}_{\SL_2}(\mathbf{\Sigma})$ that we call the \textit{classical shadow} of $r$, as in \cite{BonahonWong1} in the closed case. By definition, the classical shadow only depends on the isomorphism class of $r$. 
   
   \vspace{2mm}
\par   To motivate the results of this paper, we list three families of central representations. First irreducible representations are obviously central. Then choose for each triangle $\mathbb{T}\in F(\Delta)$ an irreducible representation $r^{\mathbb{T}} : \mathcal{S}_{\omega}(\mathbb{T}) \rightarrow \End(V_{\mathbb{T}})$ and consider the composition:
   $$ r: \mathcal{S}_{\omega}(\mathbf{\Sigma}) \xrightarrow{i^{\Delta}} \otimes_{\mathbb{T}\in F(\Delta)}\mathcal{S}_{\omega}(\mathbb{T}) \xrightarrow{\otimes_{\mathbb{T}} r^{\mathbb{T}}} \End( \otimes_{\mathbb{T}} V_{\mathbb{T}}).$$
   \par Such a representation is central and were called \textit{local representations} in \cite{BonahonBaiLiuLocalRep}. Eventually, consider the balanced Chekhov-Fock algebra $\mathcal{Z}_{\omega}(\mathbf{\Sigma}, \Delta)$ defined in \cite{BonahonWongqTrace} after the original construction of \cite{ChekhovFock}. Given a triangulated marked surface, the authors of \cite{BonahonWongqTrace} defined an algebra morphism (the quantum trace) $\tr : \mathcal{S}_{\omega}(\mathbf{\Sigma}) \rightarrow \mathcal{Z}_{\omega}(\mathbf{\Sigma}, \Delta)$ (see also \cite{LeStatedSkein}). One motivation is the fact that the representation theory of the balanced Chekhov-Fock algebra is easier to study than the one of the skein algebras (see \cite{BonahonLiu, BonahonWong2}). For an irreducible representation $\pi : \mathcal{Z}_{\omega}(\mathbf{\Sigma}, \Delta) \rightarrow \End{V}$ of the balanced Chekhov-Fock algebra, we call \textit{quantum Teichm\"uller representation}, the composition: 
   $$ r: \mathcal{S}_{\omega}(\mathbf{\Sigma}) \xrightarrow{\tr} \mathcal{Z}_{\omega}(\mathbf{\Sigma}, \Delta) \xrightarrow{\pi} \End(V). $$
   Quantum Teichm\"uller representations are central.

   \appendix
   
   \section{Proof of Proposition \ref{proptchebychev} and application}
   
   \subsection{Proof of Proposition \ref{proptchebychev}}
   
   \par We divide the  proof of Proposition \ref{proptchebychev} in five lemmas.
   
   \vspace{2mm}
   \par  {Along this section, we write $A:=\omega^{-2}$.} Denote by $\mathbb{A}=([0,1]\times S^1, \{p, p'\})$ the annulus with punctures $p=\{0\} \times \{1\}$ and $p'=\{1\}\times \{1\}$ in each of its boundary components and denote by $b=\{0\}\times S^1\setminus \{p\}$ and $b'=\{1\}\times S^1\setminus \{p'\}$ its boundary arcs. Let $\gamma\subset [0,1]\times S^1$ be the curve $\{\frac{1}{2}\} \times S^1$. Let  $\delta^{(n)}, \eta^{(n)}\subset [0,1]\times S^1$ be the arcs with endpoints $b$ and $b'$ such that $\delta^{(n)}$ spirals $n$ times in the counterclockwise direction and $\eta^{(n)}$ spirals $n$ times in the clockwise direction while oriented from $b'$ to $b$. The arcs are drawn in Figure \ref{figannulus}. By convention, $\delta^{(0)}$ and $\eta^{(0)}$ represent the empty diagram. Denote by  $\beta$ the arc $[0, 1]\times\{-1\}$. By convention, if $\alpha$ is one of the arcs $\beta, \delta^{(n)}, \eta^{(n)}$, we denote by $\alpha_{\varepsilon \varepsilon'}\in \mathcal{S}_{\omega}(\mathbb{A})$ the class of the corresponding stated tangle with sign $\varepsilon$ in $b$ and $\varepsilon'$ in $b'$. The following lemma and its proof are quite similar, though stated in a different skein algebra, to \cite[Proposition $2.2$]{LeKauffmanBracket}.
   \begin{figure}[!h] 
   	\centerline{\includegraphics[width=8cm]{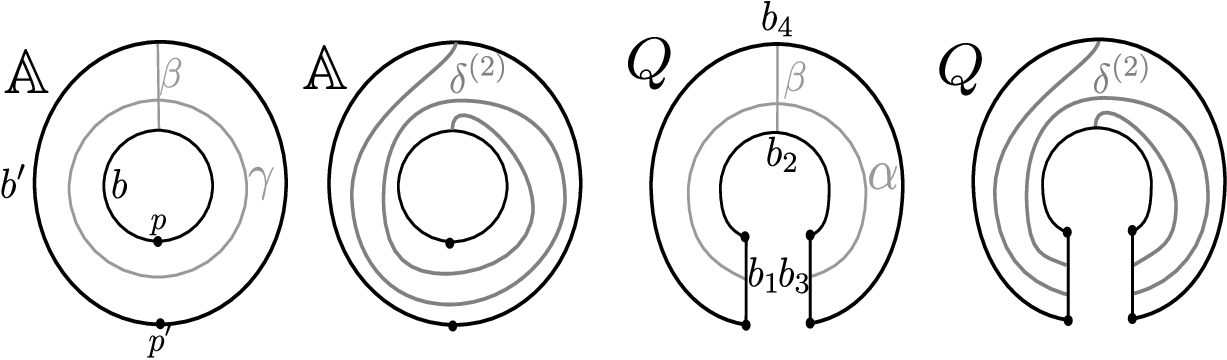} }
   	\caption{The annulus $\mathbb{A}$, the square $Q$ and some arcs and curves.} 
   	\label{figannulus} 
   \end{figure}

   \begin{lemma}\label{lemmaannulus} In $\mathcal{S}_{\omega}(\mathbb{A})$ the elements $T_N(\gamma)$ and $\beta_{\varepsilon \varepsilon'}$ commute. 
   \end{lemma}
   
   \begin{proof} First note that a direct application of the Kauffman bracket skein relations implies that $\gamma \cdot \delta^{(n)}_{\varepsilon \varepsilon'}= A \delta^{(n+1)}_{\varepsilon \varepsilon'}+A^{-1} \delta^{(n-1)}_{\varepsilon \varepsilon'}$ and $\gamma\cdot \eta^{(n)}_{\varepsilon \varepsilon'}= A \eta^{(n-1)}_{\varepsilon \varepsilon'}+A^{-1}\eta^{(n+1)}_{\varepsilon \varepsilon'}$ when $n\geq 1$. Next we show by induction on $n\geq 0$ that $T_n(\gamma)\cdot \beta_{\varepsilon \varepsilon'} = A^n \delta^{(n)}_{\varepsilon \varepsilon'} + A^{-n}\eta^{(n)}_{\varepsilon \varepsilon'}$. The statements is an immediate consequence of the definitions when $n=0$ and a direct application of the Kauffman bracket relations when $n=1$. Suppose that the results holds for $n$ and $n+1$. Then:
   	\begin{eqnarray*}
   		T_{n+2}(\gamma)\beta_{\varepsilon \varepsilon'} &=& \gamma\cdot T_{n+1}(\gamma)\cdot \beta_{\varepsilon \varepsilon'} - T_{n}(\gamma)\cdot \beta_{\varepsilon \varepsilon'}
   		\\ &=& \gamma\cdot (A^{n+1}\delta^{(n+1)}_{\varepsilon \varepsilon'} + A^{-(n+1)}\eta^{(n+1)}_{\varepsilon \varepsilon'}) - (A^n\delta^{(n)}_{\varepsilon \varepsilon'} + A^{-n}\eta_{\varepsilon \varepsilon'}^{(n)})
   		\\ &=& A^{n+2} \delta^{(n+2)}_{\varepsilon \varepsilon'} + A^{-(n+2)}\eta^{n+2}_{\varepsilon \varepsilon'},
   	\end{eqnarray*}
   	and the statement follows by induction. Similarly we show that $\beta_{\varepsilon \varepsilon'} \cdot T_n(\gamma)= A^{-n}\delta^{(n)}_{\varepsilon \varepsilon'} + A^n\eta_{\varepsilon \varepsilon'}^{(n)}$. Hence we have:
   	$$ T_N(\gamma)\cdot \beta_{\varepsilon \varepsilon'} - \beta_{\varepsilon \varepsilon'}\cdot T_N(\gamma) = (A^N-A^{-N})(\delta^{(N)}_{\varepsilon \varepsilon'} - \eta^{(N)}_{\varepsilon \varepsilon'}) =0.$$
   	
   \end{proof}
   
   \par  Denote by $Q$ the square, \textit{i.e.} a disc with four punctures on its boundary. Let $b_1, \ldots, b_4$ be its four boundary arcs labelled in the counter-clockwise order. When gluing $b_1$ along $b_3$, we obtain the annulus with $b_2$ sent to $b$ and $b_4$ sent to $b'$. We denote by  $i_{|b_1\#b_3} : \mathcal{S}_{\omega}(\mathbb{A})\hookrightarrow \mathcal{S}_{\omega}(Q)$ the gluing morphism. Let $\alpha, \beta, \delta^{(n)}, \eta^{(n)}\subset Q$ be the arcs which are glue together to form $\gamma, \beta, \delta^{(n)}$ and $\eta^{(n)}$ respectively as in Figure \ref{figannulus}. Fix an arbitrary orientation $\mathfrak{o}$ of the boundary arcs of $Q$ and consider the filtration $(\mathcal{F}_m)_{m\geq 0}$ associated to $S=\{ b_1, b_3\}$ of Definition \ref{def_filtration}. Write $d:\mathcal{S}_{\omega}(Q)\rightarrow \mathbb{Z}^{\geq 0}$ the corresponding map and  $\mathcal{G}_m:=\quotient{\mathcal{F}_m}{\mathcal{F}_{m-1}}$ the corresponding graduation.
   
   \begin{lemma}\label{lemma1}
   	The following holds:
   	$$ \lt\left( (\alpha_{++}+\alpha_{--})^N\right) = \lt \left(T_N(\alpha_{++}+\alpha_{--})\right) = \alpha_{++}^N + \alpha_{--}^N. $$
   \end{lemma}
   
   \begin{proof}
   	First note that in $\mathcal{G}_{4}$, we have $\alpha_{--}\alpha_{++}= q^2\alpha_{++}\alpha_{--}$. So it follows from Lemma \ref{lemma_qbinomial} that in $\mathcal{G}_{2N}$, we have $\lt\left( (\alpha_{++}+\alpha_{--})^N\right) = \alpha_{++}^N + \alpha_{--}^N $. Since $T_N(X) - X^N$ is a polynomial of degree strictly smaller that $N$,  the degree of $T_N(\alpha_{++}+\alpha_{--}) - (\alpha_{++}+\alpha_{--})^N$ is strictly smaller than $2N$, thus $ \lt \left(T_N(\alpha_{++}+\alpha_{--})\right) = \lt\left( (\alpha_{++}+\alpha_{--})^N\right) $.
   \end{proof}
   
   \vspace{2mm} Let $\alpha^{(n)}$ be the diagram made of $n$ parallel copies of $\alpha$.  Using the identifications $\partial \delta^{(n)}=\partial \eta^{(n)} = \partial \alpha^{(n)} \cup \partial \beta$, we denote by $\delta^{(n)}_{(s,\varepsilon, \varepsilon')}, \eta^{(n)}_{(s, \varepsilon, \varepsilon')} \in \mathcal{S}_{\omega}(Q)$ the classes of the tangles $\delta^{(n)}$ and $\eta^{(n)}$ with states given by a state $s$ of $\alpha^{(n)}$ and a state $(\varepsilon, \varepsilon')$ of $\beta$.
   
   \begin{lemma}\label{lemma2}
   	For $0<n<N$ and $s$ a state  of $\alpha^{(n)}$, we have:
   	$$\lt \left( \left[ [\alpha^{(n)}, s], \beta_{\varepsilon \varepsilon'} \right] \right) = (A^n-A^{-n})\left(\delta^{(n)}_{(s, \varepsilon, \varepsilon')} - \eta^{(n)}_{(s, \varepsilon, \varepsilon')} \right), $$
   	where we used the notation $[x,y]= xy-yx$.
   \end{lemma}
   
   \begin{proof}
   	The diagram obtained by stacking $\alpha^{(n)}$ on top of $\beta$ has $n$ crossings and thus $2^n$ resolutions using the Kauffman bracket relation. We remark that the resolution obtained by replacing each crossing by $\resneg$ is $A^n \delta^{(n)}_{(s,\varepsilon, \varepsilon')}$ while the resolution obtained by replacing each crossing by $\respos$ is $A^{-n} \eta^{(n)}_{(s, \varepsilon, \varepsilon')}$. These two resolutions have degree $2n$ and all the others resolutions have degrees strictly smaller, thus $\lt \left( [\alpha^{(n)}, s] \cdot \beta_{\varepsilon \varepsilon'} \right)=  A^n \delta^{(n)}_{(s,\varepsilon, \varepsilon')} + A^{-n} \eta^{(n)}_{(s, \varepsilon, \varepsilon')}$. We prove similarly that $\lt \left(  \beta_{\varepsilon \varepsilon'}\cdot [\alpha^{(n)}, s]  \right)=  A^{-n} \delta^{(n)}_{(s,\varepsilon, \varepsilon')} + A^{n} \eta^{(n)}_{(s, \varepsilon, \varepsilon')}$ and conclude by taking the difference.
   \end{proof}
   
   \begin{lemma}\label{lemma3}
   	If $x\in \mathcal{S}_{\omega}(Q)$ is a polynomial in $\mathcal{S}_{\omega}(Q)$ in the elements $\alpha_{\varepsilon \varepsilon'}$ such that $d(x)<2N$ and such that $x$ commutes with all elements $\beta_{\mu,\mu'}$, then $x$ is a constant.
   \end{lemma}
   
   \begin{proof}
   	Let $x=\sum_{i\in I} x_i [\alpha^{n_i}, s_i]$ be the decomposition in the basis of stated tangles with increasing states $s_i$ and denote by $2n<2N$ its degree. Suppose by contradiction that $n\neq 0$. Let $J=\{ j\in I \mbox{, such that } n_i=n \}\subset I$ so we have $\lt(x) = \sum_{j\in J} x_j [\alpha^{n}, s_j]$. The  hypothesis on $x$ and Lemma \ref{lemma2} imply that:
   	$$ 0 = \lt \left( [x, \beta_{\varepsilon \varepsilon'} ] \right) = \sum_{j\in J} x_j (A^n -A^{-n})(\delta^{(n)}_{(s_j, \varepsilon, \varepsilon')} - \eta^{(n)}_{(s_j, \varepsilon, \varepsilon')} ). $$
   	\par Since the elements $\delta^{(n)}_{(s_j, \varepsilon, \varepsilon')}$ and $\eta^{(n)}_{(s_j, \varepsilon, \varepsilon')}$ are linearly independent for $n\geq 1$, we conclude that $x_j(A^n-A^{-n})=0$ for all $j\in J$. Since $0<n<N$ and $N$ is odd, we obtain that $x_j=0$ for all $j\in J$ thus $\lt(x)=0$. This gives the contradiction.
   \end{proof}
   
   \vspace{2mm}
   \par The set $\mathcal{B}':=\{\alpha_{-+}^a \alpha_{++}^b \alpha_{+-}^c, a,b,c\geq 0\} \cup \{\alpha_{-+}^a \alpha_{--}^b \alpha_{+-}^c, a,b,c\geq 0\}$ forms a basis of the algebra $\mathcal{S}_{\omega}(\mathbb{B})$.  This fact is Exercise $7$ in Chapter $IV$ Section $6$ of \cite{Kassel}, and is proved as follows. Choose an orientation $\mathfrak{o}$ of the boundary arcs of $\mathbb{B}$ such that $b_L$ and $b_R$ points towards different punctures and consider the filtration associated to $S=\{b_L, b_R\}$. For each element of the basis $\mathcal{B}^{\mathfrak{o}}$, there exists exactly one element of   $\mathcal{B}'$  which has the same leading term.  For $x\in \mathcal{S}_{\omega}(\mathbb{B})$, denote by $c(x)\in \mathcal{R}$ the coefficient of $1$ in the decomposition of the basis $\mathcal{B}'$.
   
   \begin{lemma}\label{lemma4}
   	One has the equality: $ c(T_N(\alpha_{++} + \alpha_{--}))= 0$.
   \end{lemma}
   
   \begin{proof}
   	Let $n\geq 1$ be an odd integer and let us show that $c\left( (\alpha_{++} + \alpha_{--})^n \right) =0$. The proof will then follow from the fact that $T_N(X)$ is an odd polynomial, thus is a linear combination of such elements, and the fact that $c$ is linear. The product $\left( (\alpha_{++} + \alpha_{--})^n \right)$ develops as  a sum of terms of the form $x=x_1\ldots x_n$ where $x_i$ is either $\alpha_{++}$ or $\alpha_{--}$. Using the defining relations of $\mathcal{S}_{\omega}(\mathbb{B})$, we can further develop each term $x$ as a linear combination of terms of the form 
   	$\alpha_{-+}^a \alpha_{++}^b \alpha_{+-}^a$ and $ \alpha_{-+}^a \alpha_{--}^b\alpha_{+-}^a$ where $2a+b$ has the same parity than $n$. Since $n$ is odd, each of these summands satisfies $b\neq 0$ so $c(x)=0$. 
   	
   \end{proof}

   \begin{proof}[Proof of Proposition \ref{proptchebychev}]
   	Consider the element $x:= T_N(\alpha_{++}+\alpha_{--}) - \alpha_{++}^N - \alpha_{--}^N \in \mathcal{S}_{\omega}(Q)$. By Lemma \ref{lemma1}, its degree is strictly smaller that $2N$.  By Lemma \ref{lemmaannulus}, in $\mathcal{S}_{\omega}(\mathbb{A})$ the elements $T_N(\gamma)$ and $\beta_{\varepsilon \varepsilon'}$ commute. The image through the algebra  morphism $i_{|b_1\#b_3} : \mathcal{S}_{\omega}(\mathbb{A})\hookrightarrow \mathcal{S}_{\omega}(Q)$  of $T_N(\gamma)$ and $\beta_{\varepsilon \varepsilon'}$ are respectively $T_N(\alpha_{++} + \alpha_{--})$ and $\beta_{\varepsilon \varepsilon'}$, thus they commute. By Lemma \ref{lemma_baleze}, the elements $\alpha_{++}^N$ and $\alpha_{--}^N$ also commute with $\beta_{\varepsilon \varepsilon'}$ so $x$ commutes with each element $\beta_{\varepsilon \varepsilon'}$. Lemma \ref{lemma3} implies that $x$ is a constant and Lemma \ref{lemma4} implies that this constant is null. This concludes the proof.
   \end{proof}
   
   \subsection{A generalization of a theorem of Bonahon}
   
   Proposition \ref{proptchebychev} provides the following generalization of the main theorem of \cite{BonahonMiraculous}. Let $\mathcal{A}$ be an $\mathcal{R}$-algebra and $\rho : \mathbb{C}_q[\SL_2]^{\otimes k} \rightarrow \mathcal{A}$ be a morphism of algebras. Let $\rho_i $ be the $i$-th component of $\rho$. For $1\leq i \leq k$, consider the following two matrices with coefficients in $\mathcal{A}$:
   \begin{eqnarray*}
   	A_i := \begin{pmatrix} \rho_i(\alpha_{++}) & \rho_i(\alpha_{+-}) \\ \rho_i(\alpha_{-+}) & \rho_i(\alpha_{--}) \end{pmatrix} & 
   	A_i^{(N)} := \begin{pmatrix} \rho_i(\alpha_{++})^N & \rho_i(\alpha_{+-})^N \\ \rho_i(\alpha_{-+})^N & \rho_i(\alpha_{--})^N \end{pmatrix}. 
   \end{eqnarray*}
   
   The following proposition was proved in \cite[Theorem $1$]{BonahonMiraculous} in the particular case where $\rho_i(\alpha_{+-})\rho_i(\alpha_{-+})=0$ for each $i\in \{1, \ldots, k\}$.
   
   \begin{proposition}
   	If $q$ is a root of unity of odd order $N>1$, then one has:
   	$$ T_N\left( \tr (A_1\ldots A_k)\right) = \tr\left( A_1^{(N)} \ldots A_k^{(N)} \right). $$
   \end{proposition}
   
   \begin{proof} By Proposition \ref{proptchebychev} and using that both $\rho$ and the $(k-1)$-th coproduct $\Delta^{(k-1)} : \mathbb{C}_q[\SL_2] \rightarrow \mathbb{C}_q[\SL_2]^{\otimes k}$ are morphisms of algebras, one has:
   	$$ T_N \circ \rho \circ \Delta^{(k-1)} (\alpha_{++} + \alpha_{--}) = \rho \circ \Delta^{(k-1)} (\alpha_{++}^N +\alpha_{--}^N).$$
   	We conclude by remarking that $ \rho \circ \Delta^{(k-1)} (\alpha_{++} + \alpha_{--}) =   \tr (A_1\ldots A_k)$ and $ \rho \circ \Delta^{(k-1)} (\alpha_{++}^N +\alpha_{--}^N)= \tr\left( A_1^{(N)} \ldots A_k^{(N)} \right)$, where the second equality follows from the fact that $j_{\mathbb{B}}$ is a morphism of Hopf algebras (Lemma \ref{lemma_center_bigone}) hence commutes with $\Delta^{(k-1)}$.
   \end{proof}

\bibliographystyle{alpha}
\bibliography{biblio}

\end{document}